\documentclass{mcom-l}

\usepackage{color}
\usepackage{amssymb}
\usepackage{bbm}
\usepackage{mathrsfs}
\usepackage{mathtools}
\usepackage{graphicx}
\usepackage[inline]{enumitem}
\usepackage{stmaryrd}
\usepackage{scalerel}
\usepackage{multicol}
\usepackage{multirow}
\usepackage{adjustbox}
\usepackage{nicefrac}

\usepackage[unicode,colorlinks=true,pagebackref=false]{hyperref}
\usepackage[backgroundcolor=white,textsize=tiny]{todonotes}

\newtheorem{theorem}{Theorem}[section]

\newtheorem{lemma}[theorem]{Lemma}

\newtheorem{proposition}[theorem]{Proposition}

\theoremstyle{definition}

\newtheorem{remark}[theorem]{Remark}

\newtheorem{assumption}[theorem]{Assumption}

\numberwithin{equation}{section}
\numberwithin{theorem}{section}
\numberwithin{table}{section}
\numberwithin{figure}{section}

\newcommand{\tcf}[1]{#1}

\newcommand{\rom}[1]{\uppercase\expandafter{\romannumeral #1}}
\newcommand{\romnum}[1]{\mathrm{\rom{#1}}}

\renewcommand{\Re}{\operatorname{Re}}
\renewcommand{\Im}{\operatorname{Im}}

\newcommand\restr[2]{{
  \left.\kern-\nulldelimiterspace
  #1
  \vphantom{\big|}
  \right|_{#2}
  }}

\newcommand{\eremk} {\hbox{}\hfill\rule{0.8ex}{0.8ex}}

\newcommand\AnaClass[3]{\mathfrak{A}(#1,#2,#3)}
\newcommand\AnaClassbdy[4]{\mathfrak{A}(#1,#2,#3,#4)}
\newcommand\AnaClassmelenk[3]{\mathfrak{A}(#1,#2,#3)}
\newcommand{\HOneT}{H^{1,t}(\Omega, \Gamma)}

\newcommand{\CsolveM}{C_{\mathrm{sol},k}^{-}}

\newcommand{\Can}{C_{\mathrm{ana},k}^{-}}

\newcommand{\smax}{s_{\max}}
\newcommand{\CshiftP}{C_{\mathrm{shift},s}^{+}}

\newcommand{\CcontM}{C_{\mathrm{cont},k}^{-}}

\newcommand{\CcoerP}{C_{\mathrm{coer}}^+}

\newcommand{\TkOmegaM}{T_{k, \Omega}^{-}}
\newcommand{\TkGammaM}{T_{k, \Gamma}^{-}}

\newcommand{\AOmegaM}{A_{\Omega}^{-}}
\newcommand{\AGammaM}{A_{\Gamma}^{-}}

\newcommand{\WP}{{\bf (WP)}}
\newcommand{\WPs}{{\bf (WP) }}

\newcommand{\TkOmegaP}{T_{k, \Omega}^{+}}
\newcommand{\TkGammaP}{T_{k, \Gamma}^{+}}

\newcommand{\bkm}{b_{k}^{-}}

\newcommand{\Skm}{S_{k}^{-}}
\newcommand{\Lkm}{L_{k}^{-}}

\newcommand{\bkp}{b_{k}^{+}}
\newcommand{\Skp}{S_{k}^{+}}
\newcommand{\Lkp}{L_{k}^{+}}

\newcommand\highPassFilterOmega{H^\Omega_{\eta k}}

\newcommand\LNegGamma{L^{neg}_{\Gamma,q}}
\newcommand\HNegGamma{H^{neg}_{\Gamma,q}}

\newcommand{\partition}{\mathscr P}

\newcommand{\anashift}[1]{\vartheta(#1)}

\newcommand{\AP}{{\bf (AP)}}
\newcommand{\APs}{{\bf (AP) }}
\newcommand{\TOM}{T_{k,\Omega}^-}

\newcommand{\TOD}{T_{k,\Omega}^\Delta}

\newcommand{\TGM}{T_{k,\Gamma}^-}

\newcommand{\TGD}{T_{k,\Gamma}^\Delta}

\newcommand{\ROM}{R_{k,\Omega}^-}

\newcommand{\ROD}{R_{k,\Omega}^\Delta}

\newcommand{\RGM}{R_{k,\Gamma}^-}

\newcommand{\RGD}{R_{k,\Gamma}^\Delta}

\newcommand{\AOM}{A_{k,\Omega}^-}

\newcommand{\AOD}{A_{k,\Omega}^\Delta}

\newcommand{\AGM}{A_{k,\Gamma}^-}

\newcommand{\AGD}{A_{k,\Gamma}^\Delta}

\newcommand{\bM}{b_k^-}
\newcommand{\bP}{b_k^+}
\newcommand{\bD}{b_k^\Delta}

\newcommand{\CAM}{C_{{\rm A},k}^-}
\newcommand{\gAM}{\gamma_{\rm A}^-}

\newcommand{\CAD}{C_{{\rm A},k}^\Delta}
\newcommand{\gAD}{\gamma_{\rm A}^\Delta}

\newcommand{\LO}{L^\eta_{k,\Omega}}
\newcommand{\HO}{H^\eta_{k,\Omega}}

\newcommand{\LG}{L^\eta_{k,\Gamma}}
\newcommand{\HG}{H^\eta_{k,\Gamma}}

\newcommand{\id}{\operatorname{id}}

\newcommand{\tf}{\widetilde f}
\newcommand{\tg}{\widetilde g}

\newcommand{\etaF}[1]{\eta^{(#1)}}
\newcommand{\etaA}[1]{\eta^{(\rm exp)}_{#1}}

\newcommand{\CT}{\mathcal T}

\newcommand{\APOL}{\alpha}
\newcommand{\APOLA}{{\alpha_{\rm a}}}
\newcommand{\APOLS}{{\alpha_{\rm s}}}
\newcommand{\APOLC}{{\alpha_{\rm c}}}

\newcommand{\LP}{\mathscr P}

\newcommand{\interface}{\Gamma_{\operatorname{interf}}}

\newcommand{\SIA}[1]{\sigma_{#1}}

\newcommand{\anasplit}[1]{\gamma^\star_{#1}}

\newcommand{\DtN}{\mathrm{DtN}_k}
\newcommand{\DtNz}{\mathrm{DtN}_0}

\newcommand{\normal}{\pmb{n}}

\usepackage{xcolor}

\begin{document}

\title%
[
Wavenumber-explicit analysis of heterogeneous Helmholtz problems
]
{
Wavenumber-explicit stability and convergence analysis of $hp$ finite element
discretizations of Helmholtz problems in piecewise smooth media
}

\author{M. Bernkopf}
\address{Institute for Analysis and Scientific Computing, TU Wien, A-1040 Vienna}
\email{maximilian.bernkopf@asc.tuwien.ac.at}

\author{T. Chaumont-Frelet}
\address{Inria Univ. Lille and Laboratoire Paul Painlev\'e, 59655 Villeneuve-d'Ascq, France}
\email{theophile.chaumont@inria.fr}

\author{J.M. Melenk}
\address{Institute for Analysis and Scientific Computing, TU Wien, A-1040 Vienna}
\email{melenk@tuwien.ac.at}



\date{}

\begin{abstract}
We present a wavenumber-explicit convergence analysis of the $hp$ Finite Element Method
applied to a class of heterogeneous Helmholtz problems with piecewise analytic coefficients at large wavenumber $k$.
Our analysis covers the heterogeneous Helmholtz equation with Robin,
exact Dirichlet-to-Neumann, and second order absorbing boundary conditions, as well as perfectly matched layers.
\end{abstract}

\maketitle

\section{Introduction}
\label{section:introduction}

Time-harmonic wave propagation problems play a major role in
a wide range of physical and industrial applications. For the numerical
treatment of such problems in heterogeneous media, Galerkin discretization
methods of finite element type constitute one of the most efficient
approaches currently available.  
Still, in the high-wavenumber regime where the computational domain
spans many wavelengths, the problem becomes highly indefinite, which
leads to substantially increased computational costs due to increased 
dispersion errors. As a result, significant research efforts have focused 
in the past decades on understanding these dispersion errors 
and the question of how to design finite element methods (FEM) able to cope with them. 

In the present work, as in \cite{bernkopf21}, we focus on the scalar Helmholtz equation,
which is arguably the simplest, yet interesting, PDE model to analyze
phenomena appearing in the high-wavenumber regime. Specifically,
given a domain $\Omega$ with analytic boundary $\Gamma := \partial \Omega$
and $f: \Omega \to \mathbb C$ our model problem is to find $u: \Omega \to \mathbb C$
such that
\begin{equation*}
\left \{
\begin{array}{rcll}
-k^2 \mu u - \nabla \cdot(A\nabla u) &=& f & \text{ in } \Omega, 
\\
A\nabla u \cdot n -T_k u &=& 0 & \text{ on } \Gamma,
\end{array}
\right .
\end{equation*}
where $\mu,A$ are given piecewise analytic coefficients (with analytic interfaces where $\mu$,
$A$ may jump), and $T_k$ is an operator acting on functions on the boundary. For instance,
$T_k$ could be the Dirichlet-to-Neumann (DtN) map of the exterior Helmholtz problem
in $\mathbb R^d \setminus \overline{\Omega}$ or $T_k = ik$ in the simpler case of
impedance boundary conditions. The ``homogeneous'' Helmholtz
equation where $\mu = 1$ and $A = I$ is already largely covered
in the literature, and the present work focuses on the ``heterogeneous''
case where these coefficients are allowed to vary on $\Omega$.

A first milestone towards a better understanding of the behavior of
finite element methods was achieved in the 90s in the seminal works
\cite{ihlenburg-babuska95,ihlenburg-babuska97}.
There, the authors focus on homogeneous one-dimensional media, and present error estimates
that are explicit in terms of the wavenumber $k$, the mesh size $h$, and the polynomial degree
$p$. They also coined the term ``pollution effect'' for the dispersion errors observed in the 
numerics, which expresses the observation 
that in the high-wavenumber regime, quasi-optimality of the finite element
solution is lost, unless the number of degrees of freedom per wavelength is increased
as $k$ is increased. Similar estimates were obtained for two and three-dimensional
homogeneous media in the same period, but limited to the lowest-order case where
$p=1$ \cite{melenk95}.

The pollution effect has been further analyzed in the 2000s, in particular using
dispersion analysis \cite{ainsworth04}. While dispersion analysis is extremely
insightful from a practical viewpoint, it is limited to homogeneous media
discretized with translation-invariant grids, and  right-hand sides as well as boundary conditions
are not taken into account. 
%As a result, strictly speaking, a dispersion analysis does 
%not actually provide any error estimates.

The direct extension of the first-order result of \cite{melenk95} to higher-order
elements is non-trivial. Indeed, the error analysis follows a duality technique
commonly known as ``Schatz argument'' \cite{schatz74}, where the finite element error is
abstractly used as a right-hand side of an adjoint problem. One is then led to
study the regularity of the solution to the adjoint problem with an $L^2$ right-hand
side, even if the actual right-hand side to the physical problem is more regular.
This is an issue when analyzing high-order methods, since their high convergence
rate do not apply to solutions with low regularity. The situation was unlocked
in the 2010s thanks to the key concept of regularity splitting
\cite{melenk-sauter10,melenk-sauter11}. The idea is to identify two distinct contributions
in the solution, namely, to write $u = u_{\rm F} + u_{\rm A}$, where $u_{\rm F}$ only possesses
finite regularity, but behaves well in the high-wavenumber regime, whereas
$u_{\rm A}$ is oscillatory but analytic.

The splitting proposed in \cite{melenk-sauter10,melenk-sauter11} is based
on low- and high-pass filters defined via the Fourier transformation. For
homogeneous media with analytic boundaries, it permits to establish the quasi-optimality
of the finite element solution under the assumption that
\begin{equation}
\label{eq_stability_h}
\frac{kh}{p} + k^{\theta}\left (\frac{kh}{\sigma p}\right )^p \leq c,
\end{equation}
for some problem-dependent $\theta$  and $\sigma$ and a generic constant $c$ that does not depend on $k$, $h$, or $p$.
We can also rewrite \eqref{eq_stability_h} as
\begin{equation}
\label{eq_stability_hp}
\frac{kh}{p} \leq c_1 \quad \text{ and } p \geq c_2 \log k,
\end{equation}
where $c_1$ and $c_2$ are again independent of $k$, $h$, and $p$.
Informally, \eqref{eq_stability_hp} suggests that the FEM 
is stable if the number of degrees of freedom per wavelength is kept constant
($kh/p \leq c_1$) and the polynomial degree is mildly increased
in the high-wavenumber regime ($p \geq c_2 \log k$). It also
means that $hp$-FEMs are pollution free,
provided the polynomial degree $p$ is increased logarithmically
with the wavenumber. Together with the dispersion analysis \cite{ainsworth04}, the 
stability condition \eqref{eq_stability_hp} strongly encourages the use of high-order
methods to solve high-wavenumber problems to reduce the pollution
effect. In fact, the performance of high-order methods is also
clearly observed numerically, and therefore, they are now widely used
in industry \cite{beriot-prinn-gabard16,taus-zepedanunez-hewett-demanet17}.

The idea of regularity splitting is actually not limited to finite
element methods, and has been used repeatedly in the literature. Important
examples include the sharp treatment of corner singularities \cite{chaumontfrelet-nicaise18}
and unfitted meshes \cite{chaumontfrelet16}, as well as other discretization techniques
such as discontinuous Galerkin \cite{melenk-parsania-sauter13}, continuous interior penalty
\cite{du-wu15,zhu-wu13} and multiscale \cite{chaumontfrelet-valentin20} methods.
Besides, the technique is not only useful to derive {\it a priori} error estimates,
but is also important in the context of {\sl a posteriori} error estimation
\cite{chaumontfrelet-ern-vohralik21,dorfler-sauter13}. We finally mention that
the idea can also be extended to more complicated wave propagation problems
\cite{melenk-sauter21,chaumontfrelet-vega21}.

Regularity splittings are thus an important ingredient of wavenumber-explicit analysis of 
FEMs. Yet, the technique proposed in \cite{melenk-sauter10,melenk-sauter11} heavily
relies on the Green's functions of the Helmholtz operator, and, as a
result, it appears to be limited to homogeneous media. So far, two alternative
procedures to build a regularity splitting in heterogeneous media have
been proposed. On the one hand, an iterative argument based on elliptic
regularity is proposed in \cite{chaumontfrelet-nicaise19}. This technique
can handle general wave operators in piecewise smooth media, but unfortunately,
it only provides \eqref{eq_stability_h} with a constant $c$ depending on $p$,
which prevents the analysis of $hp$-FEMs, and does not lead to
\eqref{eq_stability_hp}.  On the other hand, novel ideas based on semi-classical
analysis have been recently introduced in
\cite{lafontaine-spence-wunsch_2020,lafontaine-spence-wunsch_2021,galkowski2022helmholtz}.
While this last splitting provides \eqref{eq_stability_h} and \eqref{eq_stability_hp}
with $p$-robust constants, it requires \emph{globally} smooth parameters.
In particular, the important case of piecewise constant parameters, modelling
a medium consisting of different materials, is not covered. An advantage
of \cite{lafontaine-spence-wunsch_2020,lafontaine-spence-wunsch_2021,galkowski2022helmholtz}
over the present splitting, however, is that only finite (although global) regularity of the 
coefficients is required, as compared to (only piecewise) analyticity here. The two approaches
are thus complementary.

It should be stressed that both approaches, like in fact all $k$-explicit analyses of numerical methods for Helmholtz problems, 
rely on $k$-explicit bounds $\CsolveM$ for the solution operator of the continuous problem. If $\CsolveM$ 
grows only polynomially in $k$, then exponential approximability of high order methods can be exploited, as is done 
in both approaches. We refer to a more detailed discussion of $\CsolveM$ in Remark~\ref{remark_polynomial_stability}. 

In this work, we modify the initial Fourier filtering approach of
\cite{melenk-sauter10,melenk-sauter11} and adapt it to be able to develop
a regularity splitting of solutions of wave propagation problems in piecewise smooth
media and subsequently apply a duality argument (``Schatz argument'') to Galerkin discretizations. 
We cover several classes of scalar Helmholtz problems in
heterogeneous media with different types of boundary conditions, providing quasi-optimality 
under the resolution conditions \eqref{eq_stability_h} and \eqref{eq_stability_hp} with $p$-robust
constants. The approach taken in the present work is quite general and thus not limited to the 
scalar Helmholtz problems under consideration. 

The remainder of this work is organized as follows. 
Section~\ref{section:assumptions_and_problem_specific_notation} presents
the settings and key notations for our work. We devise our regularity
splitting in Section~\ref{section:abstract_contraction_argument}
and subsequently apply it to Galerkin discretizations in 
Section~\ref{section:galerkin}. The analysis of 
Sections~\ref{section:abstract_contraction_argument}
and \ref{section:galerkin} is performed in an abstract setting under
general assumptions. In Section~\ref{section:covered_problems}, we verify 
that these abstract assumptions indeed hold true for several classes of relevant 
scalar Helmholtz problems, namely, 
the heterogeneous Helmholtz equation with 
a) the classical impedance boundary conditions and 
perfectly matched matched layers on a circular/spherical domain; 
b) exact boundary conditions expressed in terms of the Dirichlet-to-Neumann operator; 
c) second order absorbing boundary conditions. 
We also refer to \cite{bernkopf21} where in addition to the present examples the case 
of the elasticity equation on circular domains with exact Dirichlet-to-Neumann boundary 
conditions is analyzed with the present techniques. 
%We also present
%numerical examples in Section~\ref{section:numerical_examples} before
%drawing concluding remarks in Section \ref{section:conclusion}. 
Appendix~\ref{section:dtn_via_bios}
presents technical results about Dirichlet-to-Neumann operators which 
are of independent interest. 
%%%  Appendix~\ref{section:analytic_regularity_second_order_abcs} shows analytic regularity for Helmholtz problems
%%%  equipped with second order absorbing boundary conditions.

% !TEX root = main.tex
\section{Assumptions and Problem Specific Notation}\label{section:assumptions_and_problem_specific_notation}

\subsection{Wavenumber}

We use the letter $k$ for the wavenumber. Since we are especially
interested in the high-wavenumber regime, we assume for the sake
of simplicity that $k$ is bounded away from zero, i.e., $k \geq k_0 > 0$ for some $k_0 > 0$ that is fixed throughout the present work. 

\subsection{Generic constants}

Throughout this work, $C$ will denote a generic constant that
is independent of the wavenumber $k$ and other critical parameters such as 
the mesh size $h$ and the polynomial degree $p$. 
The constant may, however, depend on the operators
appearing in assumptions \WPs and \APs
introduced in Sections \ref{subsection_WP} and \ref{subsection_AP}
below. It may also depend
on the constant $C_\mathrm{affine}$ and $C_\mathrm{metric}$ introduced
in Assumption~\ref{assumption:quasi_uniform_regular_meshes} that describes
the ``shape regularity'' of the elements of the mesh. 
If $a$, $b,\dots$
are other quantities appearing in the analysis, we employ the notation
$C_{a,b,\dots}$ for a constant that is additionally allowed to depend
on $a$, $b,\dots$. Finally, we will employ explicit names together with
a subscript $k$ for constants allowed to depend on the wavenumber.

\subsection{Geometry} 

We consider a Lipschitz domain $\Omega \subset \mathbb R^d$,
$d\in \{2,3\}$, with boundary $\Gamma \coloneqq \partial \Omega$.
We set $\Omega^+ := {\mathbb R}^d \setminus \overline{\Omega}$. 
We introduce the interior trace $\gamma^{int}_0 u:= (u|_{\Omega})|_\Gamma$ 
and the exterior trace $\gamma^{ext}_0 u := (u|_{\Omega+})|_\Gamma$.
We assume that $\Omega$ is partitioned into a set $\LP$ of non-overlapping
Lipschitz subdomains such that $\cup_{P \in \LP} \overline{P} = \overline{\Omega}$.
The internal interface $(\cup_{P \in \LP} \partial P )\setminus \Gamma$ will be denoted
$\interface$. While the abstract results in Sections~\ref{section:abstract_contraction_argument}
and \ref{section:galerkin} do not require much regularity of $\Gamma$ and $\interface$, we flag 
at this point already that in the applications discussed in Section~\ref{section:covered_problems} 
the assumptions on the geometry are such that 
$\Gamma$ and the pieces $\partial P$, $P \in \LP$, of $\interface$ are analytic 
(Assumption~\ref{assumption:smoothness_domain_boundary_and_interface}).
Notice that this assumption prevents the presence of corners and edges,
and therefore intersection points in particular.

%%% The boundary $\Gamma \coloneqq \partial \Omega$
%%% is assumed to be analytic. In order to account for piecewise smooth
%%% coefficients, we introduce for a $d-1$ manifold interface $\Gi \subset \Omega$
%%% representing the interface(s) where discontinuities are allowed.
%%% $\Gi$ can possibly consisting of a finite number of connected components,
%%% but we require that it is non-intersecting analytical manifold separated
%%% from the boundary $\Gamma$.

\subsection{Function spaces}
\label{sec:function-spaces}

The notations $L^2(\Omega)$ and $L^2(\Gamma)$ stand for the usual Lebesgue spaces
of square-integrable complex-valued functions defined on $\Omega$
and $\Gamma$, respectively. The notations $\|{\cdot}\|_{0,\Omega}$
and $\|{\cdot}\|_{0,\Gamma}$ stand for the usual norms of these spaces.
We also write $(\cdot,\cdot)$ and $\langle \cdot,\cdot \rangle$
for the inner product of $L^2(\Omega)$ and $L^2(\Gamma)$ and (with a standard
abuse of notation) also for the usual (anti-)duality pairings. 
For $s$, $t \geq 0$, $H^s(\Omega)$
and $H^t(\Gamma)$ are (possibly) fractional Sobolev spaces of functions defined
on $\Omega$ and $\Gamma$. $\widetilde H^{-s}(\Omega)$ and $H^{-t}(\Gamma)$ are
then the (topological) antidual of $H^s(\Omega)$ and $H^t(\Gamma)$.
For all $s$, we denote by $\|\cdot\|_{s,\Omega}$ the norm of $H^s(\Omega)$,
and if $s > 0$, by $|\cdot|_{s,\Omega}$ its semi-norm. We also use similar notations
for $H^s(\Gamma)$. For $0 \leq t \leq 1$, we introduce the ``energy'' space
$\HOneT \coloneqq \{ u \in H^1(\Omega) \colon u \in H^t(\Gamma) \}$
with $k$-dependent norm
\begin{equation}
\label{eq:t-k-norm}
\|u\|_{1,t,k}^2
\coloneqq
k^2 \|u\|_{0,\Omega}^2 + k^{1-2t} |u|_{t,\Gamma}^2 + |u|_{1,\Omega}^2.
\end{equation}
We introduce this space with $t=1$ to handle second-order absorbing boundary conditions
that involve surface differential operators. One easily sees that for $t \leq 1/2$ the space $\HOneT$
coincides with $H^1(\Omega)$, but we still employ the notation $\HOneT$ to unify our presentation.

For $r \geq 0$, we will also use the broken Sobolev spaces
\begin{equation*}
H^r(\LP)
:=
\left \{
v \in L^2(\Omega)
\; | \;
v|_P \in H^r(P) \; \forall P \in \LP
\right \}
\end{equation*}
with their norm and and semi-norm given by
\begin{equation*}
\|v\|_{s,\LP}^2 := \sum_{P \in \LP} \|v\|_{s,P}^2, 
\qquad
|v|_{s,\LP}^2 := \sum_{P \in \LP} |v|_{s,P}^2.
\end{equation*}

In addition to Sobolev spaces, we will employ spaces of (piecewise) analytic functions.
For open sets $\omega\subset {\mathbb R}^d$ and $M$, $\gamma\ge 0$, we set  
\begin{equation}
\label{eq:analyticity}
\AnaClass{M}{\gamma}{\omega} \coloneqq \{v \in L^2(\omega)\,|\, 
\|\nabla^n v\|_{0,\omega}
\leq M \gamma^n \max\{n, k\}^n
\quad \forall n \in \mathbb{N}_0
\}, 
\end{equation}
where $\mathbb N_0 := \{1,2,\dots\}$ and
$|\nabla^n v(x)|^2 = \sum_{|\alpha| = n} \frac{n!}{\alpha!} |D^\alpha v(x)|^2$. 
The analyticity class 
\begin{equation}
\label{eq:analyticity-class}
\AnaClass{M}{\gamma}{\partition}\coloneqq \{v \in L^2(\Omega)\,|\, v|_P \in \AnaClass{M}{\gamma}{P} \quad \forall P \in \partition\}
\end{equation}
is a class of piecewise analytic functions on $\Omega$. 
%Specifically, if $M,\gamma \geq 0$, we introduce the analyticity class in the volume
%\begin{equation*}
%\AnaClass{M}{\gamma}{\LP}
%\coloneqq
%\{
%v \in L^2(\Omega)
%\; | \;
%%{\color{red} k} \|\nabla^n v\|_{0,P} % 
%\|\nabla^n v\|_{0,P}
%\leq M \gamma^n \max\{n, k\}^n,
%\quad n \in \mathbb{N}_0
%\quad \forall P \in \LP
%\}.
%\end{equation*}
We will also employ analytic functions defined on the boundary $\Gamma$.
For the sake of simplicity, analyticity classes on $\Gamma$ are defined as 
traces of analytic functions. That is, for an open neighborhood $T$ of $\Gamma$ and 
$\gamma > 0$, we set
\begin{equation*}
\AnaClassbdy{M}{\gamma}{T}{\Gamma}
:=
\{
\gamma^{int}_0 V \,|\, V \in \AnaClass{M}{\gamma}{T \cap \Omega}
\}. 
\end{equation*}
By a slight abuse of notation, for a function $g \in L^2(\Gamma)$ we will write $g \in \AnaClassbdy{M}{\gamma}{T}{\Gamma}$ 
to indicate the existence of a function $G \in \AnaClassbdy{M}{\gamma}{T \cap \Omega}{\Gamma}$ with $\gamma^{int}_0 G = g$. 

\begin{remark}
\label{remk:bdy-fcts}
\begin{enumerate}[nosep, start=1, label=(\roman*),leftmargin=*]
\item 
\label{item:remk:bdy-fcts-i}
Functions $G \in \AnaClass{M}{\gamma}{T \cap \Omega}$ are analytic on $\overline{T \cap \Omega}$
and therefore have a trace on $\Gamma$. Indeed, for $T\cap \Omega$ Lipschitz we have by the Sobolev
embedding theorem for $t \in {\mathbb N}$ with $t > d/2$ that
$H^t(T \cap \Omega)\subset C(\overline{T\cap \Omega})$. Hence,
$\|\nabla^n G\|_{L^\infty(T \cap \Omega)} \leq C \gamma^{n+t} \max\tcf{\{}n+2,k\}^{n+t}$ for all 
$n \in {\mathbb N}_0$ for suitable $C$. 
In turn, this implies that the Taylor series of $G$ about each point $x_0 \in T \cap\Omega$ converges with radius of convergence $1/(e \gamma)$.  
\item 
\label{item:remk:bdy-fcts-ii}
The choice of interior traces in $\AnaClassbdy{M}{\gamma}{T}{\Gamma}$ is arbitrary. One could equivalently define
them as traces of functions analytic in subsets of $\Omega^+$: 
It is shown in \cite[Lemma~{B.5}]{melenk12} that if 
$G \in \AnaClass{M}{\gamma}{T \cap \Omega}$, then there exist a neighborhood $T'$ of $\Gamma$ (depending only on $\gamma$), constants $M'$, $\gamma'$, and 
a function $G' \in \AnaClass{M'}{\gamma'}{T' \cap \Omega^+}$ with $\gamma^{int}_0 G = \gamma^{ext}_0 G'$. 
\item
\label{item:remk:bdy-fcts-iii}
Let $T$ be an open neighborhood of $\Gamma$ and $G \in \AnaClass{M}{\gamma}{T\setminus\Gamma}$.  
By \ref{item:remk:bdy-fcts-i}, the traces $\gamma^{int}_0 G$ and $\gamma^{ext}_0 G$ exist so that the 
jump $\llbracket G\rrbracket := \gamma^{ext}_0 G - \gamma^{int}_0 G $ exists. In view of the arguments given in 
\ref{item:remk:bdy-fcts-ii}, we have $\llbracket G\rrbracket \in \AnaClassbdy{M'}{\gamma'}{T'}{\Gamma}$ for some $\gamma'$, $T'$ depending only
on $\gamma$ and $T$ and some $M'$ depending additionally on $M$. 
\eremk
%It is worth noting that the width of the tubular neighborhood $T$ is essentially given by $1/\gamma$ so that 
%$T$ and $\gamma$ are not fully independent of each other. 
%\item 
%\label{item:remk:bdy-fcts-iv}
%As an alternative to defining the analyticity
%classes $\AnaClassbdy{M}{\gamma}{T}{\Gamma}$ as traces, it also possible to define them based solely in terms of the function behavior on $\Gamma$, e.g., 
%e.g., by stipulating exponential decay of expansions in terms of the Laplace-Beltrami operator; see, e.g., \cite[Lemma~{C.1}]{melenk12} 
%for details. 
\end{enumerate}
\end{remark}

%For the sake of simplicity, they are characterized as \JMM{normal} traces of analytic
%functions defined on $\Omega$. That is, for $M$, $\gamma > 0$, \JMM{and the outer normal
%vector $\normal$ on $\Gamma$} we introduce
%\begin{equation*}
%\AnaClass{M}{\gamma}{\Gamma}
%\coloneqq
%\{ v \in L^2(\Gamma)
%\; | \;
%\exists V \in \AnaClass{M}{\gamma}{\LP};
%\;
%v = \nabla V \cdot \bn
%\}.
%\end{equation*}
We note that if $\gamma \leq \gamma'$ and $M \leq M'$, then
\begin{equation*}
\AnaClass{M}{\gamma}{\LP} \subset \AnaClass{M'}{\gamma'}{\LP}, 
\qquad
\AnaClassbdy{M}{\gamma}{T}{\Gamma} \subset \AnaClassbdy{M'}{\gamma'}{T}{\Gamma}. 
\end{equation*}

We also notice that the trace inequality
\begin{equation}
\label{eq_weighted_trace}
k^{1/2} \|v\|_{s,\Gamma} + \|v\|_{s+1/2,\Gamma}
\leq
C_s
\left (
k \|v\|_{s,\LP} + \|v\|_{s+1,\LP}
\right )
\end{equation}
holds true for every $s \geq 0$.  
The analyticity classes are invariant under analytic changes of variables 
and multiplication by analytic functions: 
\begin{lemma}[\protect{\cite[Lemma~{2.6}]{melenk-sauter21}}]
\label{lemma:lemma2.6}
Let $\omega_1$, $\omega_2 \subset {\mathbb R}^d$ be bounded open and $g: \omega_1 \rightarrow \omega_2$ 
be a bijection and analytic on the closed set $\overline{\omega}_1$. Let $f_1$ be analytic on the closed
set $\overline{\omega}_2$ and $f_2 \in \AnaClass{M_f}{\gamma_f}{\omega_2}$. Then there are constants 
$C_1$, $\gamma_1 > 0$ depending only on $f_1$, $g$, $\omega_1$, $\omega_2$, and $\gamma_f$ such that 
$f_1 (f_2 \circ g) \in \AnaClass{C_1 M_f}{\gamma_1}{\omega_1}$. 
\end{lemma}

%----------------------------------------------------
\subsection{Informal explanation of the proof of the main result and the assumptions}
%----------------------------------------------------

%% TODO: say that the + problem is a parametrix for high-frequency data

We perform a general analysis under abstract assumptions
given in Sections~\ref{subsection_WP} and \ref{subsection_AP}
below. Before rigorously presenting the proof and stating the requirements,
let us rather informally present the key arguments and motivate
the main assumptions.

We consider a time-harmonic wave propagation problems of the form
\begin{equation}
\label{eq_wp_intuitive}
\left \{
\begin{array}{rcll}
-\TkOmegaM u - \nabla \cdot (A\nabla u) &=& f & \text{ in } \Omega,
\\
\partial_n u + \TkGammaM u &=& g & \text{ on } \Gamma,
\end{array}
\right .
\end{equation}
where $A: \Omega: \to \mathbb R^{d \times d}$
is a piecewise analytic coefficient. The protoypical example is
the heterogeneous Helmholtz equation with impedance boundary conditions
with
\begin{equation*}
\TkOmegaM u = k^2 \mu u, \qquad \TkGammaM u = ik u,
\end{equation*}
where $\mu: \Omega \to \mathbb R$ is a piecewise analytic coefficient,
and the superscript $^{-}$ reminds us that, in general,
there is negative zero-order term in the PDE. In what follows,
we will take $g=0$ for simplicity.

Our goal is, given $f \in L^2(\Omega)$, to split the solution
$u$ as $u = u_{\rm F} + u_{\rm A}$ where $u_{\rm F} \in H^2(\LP)$
with $|u_{\rm F}|_{2,\LP} \leq C \|f\|_{0,\Omega}$ and
$u_{\rm A}$ is (piecewise) analytic, but with a norm growing (in a controlled way) with $k$.
An intuitive (but slightly inexact) explanation of how
this splitting is obtained is as follows.

{\bf Step 1.} We introduce the following stable decomposition
of the right-hand side:
\begin{equation*}
f = f_{\rm F} + f_{\rm A},
\end{equation*}
where $f_{\rm A}$ is (piecewise) analytic, and $f_{\rm F} \in L^2(\Omega)$. 
This splitting is obtained
through filter operators defined using cutoffs in
the Fourier domain. These are introduced in Section
\ref{subsection_filtering_operators} below.
Crucially, $f_{\rm A}$ contains the ``low frequencies''
of $f$ so that it is analytic, and $f_{\rm F}$ consists
of the remainding ``high frequencies''.

{\bf Step 2.} The analytic part $u_{\rm A}$ of the splitting
is essentially obtained by solving \eqref{eq_wp_intuitive}
with $f$ replaced by $f_{\rm A}$. Assuming that
\eqref{eq_wp_intuitive} possesses appropriate regularity
shifts, $u_{\rm A}$ will indeed be (piecewise) analytic, but include
oscillations with frequency $O(k)$ that have to be expected for Helmholtz problems. 
%and oscillate ``normally'' in $k$.

{\bf Step 3.} The subtle part of the analysis is to show
that the $H^2(\LP)$ norm of $u_{\rm F} := u-u_{\rm A}$ 
does not grow with $k$. The reason why this can 
happen is that $u_{\rm F}$ solves \eqref{eq_wp_intuitive}
with a right-hand side $f_{\rm F}$ that does not 
contain the natural frequency $k$ of
the equation (it is made up of the higher-frequency
content of $f$). To rigorously quantify this 
in our analysis we consider an auxiliary coercive
problem that acts as a parametrix for high-frequency data.
Specifically, we show that, in an appropriate sense,
$u_{\rm F} \sim \widetilde u_{\rm F}$, where $\widetilde u_{\rm F}$
solves the auxiliary problem
\begin{equation}
\label{eq_ap_intuitive}
\left \{
\begin{array}{rcll}
-\TkOmegaP \widetilde u_F
-
\nabla \cdot(A\nabla \widetilde u_F)
&=&
f_F & \text{ in } \Omega,
\\
\partial_n \widetilde u_F + \TkGammaP \widetilde u_F
&=& 0
& \text{ on } \Gamma.
\end{array}
\right .
\end{equation}
The operators $\TkOmegaP$ and $\TkGammaP$ appear only in the analysis, to which they
are tuned,  but not in the numerical realization. 
For the Helmholtz equation with Robin boundary conditions,
we can select $\TkOmegaP = -\TkOmegaM = -k^2\mu$ and
$\TkGammaP = 0$. The superscript $^{+}$ reminds us that,
in general, the PDE in  \eqref{eq_ap_intuitive} will
contain a positive zero-order term. Since $f_F \in L^2(\Omega)$,
we cannot expect more than
$H^2(\LP)$ regularity for $\widetilde u_F$.
However, since \eqref{eq_ap_intuitive} is a coercive
problem, we have by Lax-Milgram and elliptic regularity 
$\|\widetilde u_{\rm F}\|_{2,\LP}
\leq
C \|f_{\rm F}\|_{0,\Omega}
\leq
C \|f\|_{0,\Omega}$
which is the desired scaling. A nice way of summarizing 
this last step is that for high-frequency data, we are
allowed to flip the ``bad sign'' in the Helmholtz problem.

We now summarize the key properties that makes such a proof
work for the case $\TkOmegaM = -\TkOmegaP = k^2\mu$, $\TkGammaM = ik$
and $\TkGammaP = 0$. We then expand on how these properties may be
generalized, leading to our main assumptions below.

We formulate a first set of assumptions on the wave propagation
problem itself. These are given in \WPs in Section \ref{subsection_WP}
below. Consider first the case $\TkOmegaM = -k^2\mu$ and $\TkGammaM = ik$, for which 
the following properties are essential: 
\begin{enumerate*}[label=(\arabic*) ]
\item The sesquilinar form associated with the problem is uniformly continuous w.r.t.\ the
$\|\cdot\|_{1,t,k}$ norm (with $t=1/2$ in this case).
\item 
For $f \in L^2(\Omega)$, \eqref{eq_wp_intuitive}
admits a unique solution $u \in H^1(\Omega)$. In general,
$\|u\|_{1,t,k}$ is not bounded by $ C \|f\|_{0,\Omega}$ with a $C > 0$ independent of 
in $k$, but we will require that the growth be merely polynomial in $k$. 
\item 
If the right-hand side $f$ in \eqref{eq_wp_intuitive} is
(piecewise) analytic, then so is the solution $u$.
\end{enumerate*}

These are exactly the assumptions given in \ref{WP1}, \ref{WP2}, 
and \ref{WP3}, up to one possible generalization. Namely, 
in \ref{WP1}, we allow for the continuity constant of the sesquilinear form to
grow (polynomially) in $k$. In exchange, we require in \ref{WP4} that
the operators $\TkOmegaM$ and $\TkOmegaP$ may be split into operators
that are uniformly bounded in $k$ and smoothening operators that maps
into sets of analytic functions. Roughly speaking, we allow polynomial
growth of the continuity constant, if this additional growth is compensated
by analytic regularity. Such considerations are crucial when dealing with
DtN operators on general surfaces, since the DtN operator does not have
to be uniformly bounded in $k$. We also allow for more
general second-order terms than $-\nabla(A\nabla \cdot)$ in \ref{WP5}
as long as the corresponding sesquilinear form $a$ is uniformly continuous
in $k$.

We then need a set of assumptions on the auxiliary problem 
\eqref{eq_ap_intuitive}. First, as seen from the above discussion,
it is key that \eqref{eq_ap_intuitive} is coercive and that it
possess a $k$-uniform $L^2(\Omega)$ to $H^2(\LP)$ regularity shift.
These requirements are stated in \ref{AP1} and \ref{AP2}. The final
assumption is that the wave propagation problem \eqref{eq_wp_intuitive}
and the auxiliary problem \eqref{eq_ap_intuitive} are suitably
linked to one another, in such a way that $u_{\rm F} \sim \widetilde u_{\rm F}$.
These assumptions are expressed by considering the differences
$\TOD := \TkOmegaP-\TkOmegaM$ and $\TGD := \TkGammaP-\TkGammaM$.
Notice that in the case of the impedance boundary condition considered
above, we have $\TOD = 2k^2\mu$ and $\TGD = ik$. The crucial point
that enables our analysis is then that $\TOD$ and $\TGD$ are compact
operator (considering functions in $H^1(\Omega)$ and $H^{1/2}(\Gamma)$)
with norm suitably controlled in $k$. Specifically, the powers of $k$
in the operator norms are compensated by the smoothening degree of
the operators. \ref{AP3} essentially states this requirement, while
allowing for a generalization similar to \ref{WP4}. More precisely,
we allow the difference operator $\TOD$ and $\TGD$ to have norms
polynomially bounded in $k$, if this is compensated by smoothening
properties, i.e., if the parts of the operators responsible for
the growth in $k$ map into sets of analytic functions.
Again, such a generalization is required to handle, e.g., DtN operators.

We finally mention that in the discussion above, we have
considered regularity shifts from $L^2(\Omega)$ to $H^2(\LP)$.
However, our assumptions are slightly more general,
as we also consider shifts between $H^r(\LP)$ and $H^{r+2}(\LP)$.
The goal of this last generalization is to derive regularity splittings
where the regular part lies in $H^{r+2}(\LP)$ with $r > 0$, provided that
$f \in H^r(\LP)$.

%------------------------
\subsection{Time-harmonic wave propagation problem}
\label{subsection_WP}
%------------------------

Given $f \in L^2(\Omega)$ and $g \in L^2(\Gamma)$, our model problem
consists in finding $u \in \HOneT$ such that
\begin{equation}
\label{eq:S_k_minus_problem}
\bkm(u,v) = (f,v) + \langle g , v \rangle \qquad \forall v \in  \HOneT
\end{equation}
with
\begin{equation}
\bkm(u, v)
\coloneqq
a(u,v)
-
(\TkOmegaM u,v)
-
\langle \TkGammaM u,v \rangle,
\end{equation}
where $a(\cdot,\cdot)$ is a sesquilinear form on $\HOneT$ and
the operators $\TkOmegaM$ and $\TkGammaM$ are differential operators acting
on ``volume'' and ``surface'' functions. Before providing a detailed, formal description
of these operators, we first explain why $\bkm$ is decomposed in the form given. 

The sesquilinear form $a(\cdot,\cdot)$ is meant to represent the ``elliptic part''
of the wave operator. A typical example is the ``grad-grad'' form in the scalar Helmholtz case, but one may
imagine other situation such as the elasticity operator. The operator
$\TOM$ then accounts for the ``frequency part'' of the operator, i.e., the
time-derivative transformed into $ik$ in the frequency domain. The most basic
situation is then $\TOM u \coloneqq k^2 u$, but we allow for more generality.
%%  Finally, $\TGM$ stands for the Dirichlet-to-Neumann map, or its approximation
%%  on the boundary.
Finally, $\TGM$ is an operator on $\Gamma$; relevant 
examples are the  DtN map or some approximation of it.

We now summarize our assumptions 
\WPs on the {\bf w}ave propagation {\bf p}roblem.
\begin{enumerate}[nosep, start=1, label={\bf (WP\arabic*)},leftmargin=*]
\item 
\label{WP1}
The sesquilinear form $\bkm \colon \HOneT \times \HOneT \to \mathbb{C}$
is continuous, i.e., there exists a constant $\CcontM \geq 0$, possibly
depending on $k$, such that
\begin{equation*}
|\bkm(u,v)| \leq \CcontM \|u\|_{1,t,k} \|v\|_{1,t,k} \qquad \forall u,v \in \HOneT.
\end{equation*}

\item 
\label{WP2}
Problem \eqref{eq:S_k_minus_problem} is well posed,
i.e., for every $f \in L^2(\Omega)$ and $g \in L^2(\Gamma)$
it admits a unique weak solution $\Skm(f,g) \coloneqq u \in \HOneT$, 
and the stability estimate
\begin{equation}\label{eq:stability_estimate}
\|\Skm(f,g)\|_{1,t,k}
\leq
\CsolveM ( \| f \|_{0, \Omega} + k^{1/2} \| g \|_{0, \Gamma} )
\end{equation}
holds true with a constant $\CsolveM$ that may depend on $k$,
but is independent of $f$ and $g$ (See also Remark~\ref{rem:infsup}). 
\end{enumerate}
Assumptions~\ref{WP1} and \ref{WP2} simply state that the problem is well-posed
and that the direct and inverse operators are bounded with continuity
constants $\CcontM$ and $\CsolveM$. For our quasi-optimality result
we also need the solution to be (piecewise) analytic
when the right-hand sides are (piecewise) analytic.

\begin{enumerate}[nosep, start=3, label={\bf (WP\arabic*)},leftmargin=*]
\item
\label{WP3} 
For each tubular neighborhood $T$ of $\Gamma$ 
there exists a constant $\gamma_0$ and there exists a non-decreasing function
$\vartheta: [\gamma_0,+\infty) \to \mathbb R$, both independent of $k$,
and there exists a possibly $k$-dependent constant $\Can$ 
such that for all $\gamma \geq \gamma_0$ 
the following holds: for all piecewise analytic data
$f \in \AnaClass{M_f}{\gamma}{\LP}$
and all boundary data 
$g \in \AnaClassbdy{M_g}{\gamma}{T}{\Gamma}$
the solution $\Skm(f,g)$ 
to Problem~\eqref{eq:S_k_minus_problem} satisfies
$\Skm(f,g) \in \AnaClass{M_u}{\anashift{\gamma}}{\LP}$, i.e., it is 
again piecewise analytic, and  
\begin{equation}
\label{eq:Mf_Mg}
%% M_u \leq C \frac{\CsolveM}{k} (M_f + M_g).
M_u \leq \Can \CsolveM k^{-1} (M_f + k M_g).
\end{equation}
\end{enumerate}

\begin{remark} \ref{WP3} permits the data $f$, $g$ to depend on $k$. In particular, the 
constants $M_f$, $M_g$ may depend on $k$ but \ref{WP3} requires the constant $M_u$ 
to be bounded as given in \eqref{eq:Mf_Mg}. 
\eremk
\end{remark}

For many model problems, the continuity constant $\CcontM$ of $\bkm$ does not depend on $k$, and the three above
assumptions are sufficient to proceed with our analysis. However,
$\CcontM$ may depend on $k$, for example, when considering DtN operators on general
surfaces $\Gamma$, or when dealing with 
Maxwell's equations \cite{melenk-sauter21}. As a result,
we include an additional assumption to treat these cases.

\begin{enumerate}[nosep, start=4, label={\bf (WP\arabic*)},leftmargin=*]
\item
\label{WP4}
The linear operators $\TOM \colon \HOneT \to \HOneT^\prime$ and
$\TGM \colon H^{t}(\Gamma) \to H^{-t}( \Gamma)$ admit splittings into linear operators
\begin{equation}
\label{eq_splitting_TOM_TGM}
\TOM = \ROM + \AOmegaM, \qquad \TGM = \RGM + \AGammaM
\end{equation}
such that the ``$R$ part'' is uniformly bounded in $k$, while the ``$A$ part''
maps into analytic functions: We assume 
\begin{equation}
\label{eq_continuity_ROM_RGM}
|(\ROM u, v)| + |\langle \RGM u, v \rangle| \leq C \|u\|_{1,t,k} \|v\|_{1,t,k}
\qquad \forall u,v \in \HOneT,
\end{equation}
and we assume 
the existence of a constant $\CAM$ (possibly depending on $k$)
and a constant $\gAM$ and a tubular neighborhood $T$ (both independent of $k$) 
such that for all $u \in \HOneT$
%and $v \in H^t(\Gamma)$, 
we have $\AOmegaM u \in \AnaClass{M_u}{\gAM}{\LP}$
and $\AGammaM (u|_\Gamma) \in \AnaClassbdy{M_v}{\gAM}{T}{\Gamma}$ with
\begin{equation*}
% M_u \leq \CAM \|u\|_{1,\tcf{t,k}},
M_u \leq \CAM \|u\|_{1,t,k},
\qquad
%\JMM{k} M_v \leq \CAM \|v\|_{t,\Gamma}.
% k M_v \leq \CAM \|u\|_{1,\tcf{t,k}}.
k M_v \leq \CAM \|u\|_{1,t,k}.
\end{equation*}
\end{enumerate} %end WPs
\begin{enumerate}[nosep, start=5, label={\bf (WP\arabic*)},leftmargin=*]
\item
We assume that $a(\cdot,\cdot)$ is uniformly-in-$k$ continuous, i.e.,
\label{WP5}
\begin{equation}
\label{eq_continuity_a}
|a(u,v)| \leq C \|u\|_{1,t,k}\|v\|_{1,t,k} \quad \forall u,v \in \HOneT.
\end{equation}
\end{enumerate}

\subsection{Auxiliary positive problem}
\label{subsection_AP}

Our analysis hinges on an auxiliary problem that is meant to be a ``positive''
version of the time-harmonic problem. We thus introduce the sesquilinear form
\begin{equation}\label{eq:S_k_plus_problem}
\bkp(u, v)
\coloneqq
a(u,v)
-
(\TkOmegaP u, v)
-
\langle \TkGammaP u , v \rangle
\quad
\forall u,v \in \HOneT,
\end{equation}
where $\TkOmegaP: H^{1,t}(\Omega) \to \left (H^{1,t}(\Omega) \right )'$
and $\TkGammaP: H^t(\Omega) \to H^{-t}(\Omega)$ are problem specific
(continuous linear) operators chosen to conduct
the analysis.  In the simplest case when $\TkOmegaM = +k^2$ and $\TkGammaM = -ik$,
we can simply select $\TkOmegaP = -k^2$ and $\TkGammaP = 0$. However,
allowing for more generality enables us to treat a much wider class
of problems.

The assumptions \APs on the {\bf a}uxiliary {\bf p}roblem are listed as follows: 

% %% TODO maybe we do not need continuity
% {\bf AP1.} The sesquilinear form $\bkp$ is continuous, i.e., there exists a constant
% $\CcontP \geq 0$ such that
% \begin{equation*}
% | \bkp(u, v) | \leq \CcontP \| u \|_{1,t,k} \| v \|_{1,t,k}, \qquad \forall u,v \in \HOneT.
% \end{equation*}

\begin{enumerate}[nosep, start=1, label={\bf (AP\arabic*)},leftmargin=*]
\item
\label{AP1} 
$\bkp$ is coercive: there exist $\CcoerP > 0$ (independent of $k$) and   $\sigma \in \mathbb{C}$ with
$|\sigma| = 1$ such that
\begin{equation}
\label{eq_coercivity_bkp}
\Re ( \sigma  \bkp(u, u) ) \geq \CcoerP \| u \|_{1,t,k}^2 \qquad \forall u \in \HOneT.
\end{equation}
% where $C$ does not depend on $k$.
\end{enumerate}
Thus, for $f \in \widetilde H^{-1}(\Omega)$ and
$g \in H^{-t}(\Gamma)$, we may define $\Skp(f,g)$ as the unique element of $\HOneT$ such that
\begin{equation}
\label{eq_definition_Skp}
\bkp(\Skp(f,g),v) = (f,v) + \langle g,v \rangle \quad \forall v \in \HOneT.
\end{equation}
\begin{enumerate}[nosep, start=2, label={\bf (AP\arabic*)},leftmargin=*]
\item 
\label{AP2} 
There exists 
an integer $\smax \geq 0$ and for each $0 \leq s \leq \smax$  there exists a 
$\CshiftP>0$ independent of $k$ such that 
if $f \in H^s(\LP)$
and $g \in H^{s+1/2}(\Gamma)$ we have $\Skp(f,g) \in H^{s+2}(\LP)$ with
\begin{equation}
\label{eq_shift_Skp}
\!\!
\|\Skp(f,g)\|_{s+2,\LP}
\leq
\CshiftP %C_s
\!\!
\left [
k^s \|f\|_{0,\Omega}
+
\|f\|_{s,\LP}
+
k^{s+1/2} \|g\|_{0,\Gamma}
+
\|g\|_{s+1/2, \Gamma}
\right ]\!\!.
\hspace*{-5mm}
\end{equation}
% where the constant $\CshiftP$ may depend on $s$, but is independent of $k$.
\end{enumerate}
We further need the auxiliary problem to have similar Sobolev regularity properties as the original wave propagation problem.
This is measured by the sesquilinear form
\begin{equation*}
\bD(u,v)
\coloneqq
\bkm(u,v) - \bkp(u,v)
=
(T_{k,\Omega}^\Delta u,v)
+
\langle T_{k,\Gamma}^\Delta u,v \rangle
\quad
\forall u,v \in \HOneT, 
\end{equation*}
where $T_{k,\Omega}^\Delta = T_{k,\Omega}^+ - T_{k,\Omega}^-$
and $T_{k,\Gamma}^\Delta = T_{k,\Gamma}^+ - T_{k,\Gamma}^-$.
Similar to the discussion on assumption \ref{WP4} it is enough
for many applications
to assume that $b_k^\Delta$ is bounded uniformly in $k$. 
To be fully general, however, we need a last
splitting assumption.

\begin{enumerate}[nosep, start=3, label={\bf (AP\arabic*)},leftmargin=*]
\item 
\label{AP3} 
We have the splittings
\begin{equation}
\label{eq_splitting_TOD_TGD}
\TOD = \ROD + \AOD,
\qquad
\TGD = \RGD + \AGD.
\end{equation}
The linear operator $\ROD$ is of order at most one and satisfies
for $0 \leq s \leq s_{\max}$ 
\begin{equation}
\label{eq_ROD}
\|\ROD u\|_{s,\LP}
\leq
C_s
\left (
k^2 \|u\|_{s,\LP} + k \|u\|_{s+1,\LP}
\right )
\qquad \forall u \in H^{s+1}(\LP), 
\end{equation}
whereas the linear operator $\RGD$ is of order zero and satisfies
\begin{equation}
\label{eq_RGD}
\|\RGD v\|_{s+1/2,\Gamma}
\leq
C \left (
k^{3/2} \|v\|_{s,\Gamma} + k \|v\|_{s+1/2,\Gamma}
\right )
\qquad \forall v \in H^{s+1/2}(\Gamma).
\end{equation}
Finally, there exists a constant $\CAD$, possibly depending on $k$, and there exist a constant 
$\gAD$ and a tubular neighborhood $T$ of $\Gamma$ (both independent of $k$) such that for all $u \in \HOneT$ 
%and $v \in H^t(\Gamma)$, 
we have 
$\AOD u \in \AnaClass{M_u}{\gAD}{\LP}$ and $\AGD (u|_\Gamma) \in \AnaClassbdy{M_v}{\gAD}{T}{\Gamma}$ with
\begin{equation}
\label{eq_A_PLUS}
M_u \leq \CAD \|u\|_{1,t,k}, 
\qquad
k M_v \leq \CAD \|u\|_{1,t,k}. 
\end{equation}
\end{enumerate}

\begin{remark} [On the stability assumption \ref{WP2}]
\label{rem:infsup}
The stability estimate \eqref{eq:stability_estimate} does not make minimal regularity assumptions on the data $f$, $g$. 
However, the stability constant $\CsolveM$ \eqref{eq:stability_estimate} is closely related to the inf-sup constant 
\begin{equation}
\label{eq:inf-sup}
\gamma_{\rm infsup}
:=
\inf_{u \in \HOneT} \sup_{v \in \HOneT}
\frac{|\bkm(u,v)|}{\|u\|_{1,t,k} \|v\|_{1,t,k}}
\end{equation} 
if \ref{AP1} and \ref{AP3} hold. Let $B^-_k: \HOneT \rightarrow \HOneT^\prime$ be the operator induced by $\bkm$. 
We claim that if $B^-_k$ is an isomorphism, in which case $1/\gamma_{\rm infsup} = \|(B^\prime_k)^{-1} \|_{\HOneT\leftarrow \HOneT^\prime} $, 
then \eqref{eq:stability_estimate} holds with $\CsolveM = O(1/(k\gamma_{\rm infsup}))$. To see this, note that 
for the solution $u = \Skm(f,g)$ of \eqref{eq:S_k_minus_problem} 
there is $v \in \HOneT$ with $\|v\|_{1,t,k} = 1$ such that 
\begin{equation*}
\gamma_{\rm infsup} \|u\|_{1,t,k} \leq |\bkm(u,v)|  = |(f,v) + \langle g,v\rangle | 
\lesssim (k^{-1} \|f\|_{0,\Omega} + k^{-1/2} \|g\|_{0,\Gamma} ) \|v\|_{1,t,k} 
\end{equation*}
so that \eqref{eq:stability_estimate} holds with $\CsolveM = O( 1/(k \gamma_{\rm infsup}))$. 
Conversely, if \eqref{eq:stability_estimate} holds, then the inf-sup constant satisfies \eqref{eq:infsup-estimate} below 
as we now show using the arguments used, e.g., 
in \cite[Prop.~{8.2.7}]{melenk95} or \cite[Thm.~{2.5}]{esterhazy-melenk12}. 
The assumption \ref{AP3} implies that $B^-_k$ is a compact perturbation of a coercive operator. Since \eqref{eq:stability_estimate} 
implies uniqueness, the Fredholm alternative provides that $B^-_k$ is an isomorphism. Since the norm of $(B^-_k)^{-1}$ equals 
the norm of $((B^-_k)^\prime)^{-1}$, we compute the latter. 
To see this,  let $u \in \HOneT$ be arbitrary and 
$z \in \HOneT$ solve, for all $w \in \HOneT$, 
\begin{align*}
\bkm(z,w) & = \bkm(u,w) - \bkp(u,w)\\
& = 
(\TOD u, w)
+
\langle \TGD u, w\rangle
=
(\ROD u + \AOD u, w)
+
\langle \RGD u + \AGD u , w\rangle.
\end{align*}
Noting 
\begin{align*}
|(\TOD u,w)| & \lesssim (k + \CAD) \|u\|_{1,t,k} \|w\|_{0,\Omega},  \\
|\langle\TGD u,w\rangle| & \lesssim \|\TGD u\|_{0,\Gamma} \|w\|_{0,\Gamma} 
\lesssim \left(\|\RGD u\|_{1/2,\Gamma} + k^{-1/2} \CAD \|u\|_{1,t,k} \right) \|w\|_{0,\Gamma} \\
&\lesssim \left(k^{3/2} \|u\|_{0,\Gamma} + k \|u\|_{1/2,\Gamma} + \CAD k^{-1/2} \|u\|_{1,t,k} \right) \|w\|_{0,\Gamma}  \\
& \lesssim (k + \CAD k^{-1/2})\|u\|_{1,t,k} \|w\|_{0,\Gamma}. 
\end{align*}
By \eqref{eq:stability_estimate} and these two bounds, we get 
\begin{equation*}
\|z\|_{1,t,k} \lesssim (k^{3/2} + \CAD) \CsolveM \|u\|_{1,t,k}.  
\end{equation*}
We conclude 
\begin{align*} 
\sup_{v \ne 0} \frac{|\bkm(v,u)|}{\|u\|_{1,t,k} \|v\|_{1,t,k}} 
 \stackrel{v = u - z}{\ge } \frac{|\bkp(u,u)|}{\|u\|_{1,t,k} \|u-z\|_{1,t,k}} 
\stackrel{\ref{AP1}}{\gtrsim }
\frac{\|u\|^2_{1,t,k}}{\|u\|_{1,t,k} (\|u\|_{1,t,k} + \|z\|_{1,t,k})} 
\end{align*}
and arrive at 
\begin{equation}
\label{eq:infsup-estimate}
\gamma_{\rm infsup} \gtrsim \frac{1}{\CsolveM (k^{3/2} + \CAD)}. 
\end{equation}
\eremk
\end{remark}

% !TEX root = main.tex
\section{Abstract Contraction Argument}
\label{section:abstract_contraction_argument}

\subsection{Filtering operators}
\label{subsection_filtering_operators}

Following \cite{melenk-sauter10,melenk-sauter11}, a key ingredient of our
splittings are filtering operators defined using cut-offs in the Fourier domain.
We need to filter both surface and volume functions, leading to two pairs 
of filters.

Before providing proper mathematical definitions, we sumarize
the ideas behind the construction of the filtering operators. To this end,
consider a function $f \in L^2(\mathbb R^d)$ together with its Fourier
transform $\widehat f$. The decomposition
\begin{equation*}
\widehat f = \chi_\eta \widehat f + (1-\chi_\eta) \widehat f, 
\end{equation*}
where $\chi_\eta$ is the characteristic function of the ball $B_\eta(0)$, induces the splitting
\begin{equation*}
f = L^\eta f + H^\eta f, 
\end{equation*}
where $L^\eta$ and $H^\eta$ are the inverse Fourier transforms
of $\chi_\eta \widehat f$ and $(1-\chi_\eta) \widehat f$, respectively. 
By construction the Fourier transform of $L^\eta f$ is supported in $\overline{B(0,\eta)}$ so that
we call $L^\eta$ a low-pass filter, whereas $H^\eta$ is referred
to as a high-pass filter. Crucially, $L^\eta f$ is analytic,
which is key in the construction of the analytic part of the
solution splitting.
Next, we extend the above construction to handle
functions defined on the boundary $\Gamma$, and deal
with piecewise analyticity instead of global analyticity.

On the boundary $\Gamma$, we can employ the surface filters
introduced in \cite[Lemmas~{4.2}, 4.3]{melenk-sauter11}. Proposition
\ref{proposition:high_low_pass_filter_Gamma} sums up their essential properties.

\begin{proposition}[Surface filters]
\label{proposition:high_low_pass_filter_Gamma}
For each $\eta > 1$, there exist two operators
$\LG,\HG \colon L^2(\Gamma) \rightarrow L^2(\Gamma)$
such that $\HG + \LG = \id$. For $0 \leq s^\prime \leq s$
the operator $\HG$ satisfies
\begin{equation}
\label{eq_high_pass_filter_Gamma}
\|\HG g\|_{H^{s^\prime}(\Gamma)}
\leq
C_{s,s^\prime} (\eta k)^{s^\prime-s} \| g \|_{H^{s}(\Gamma)}
\qquad
\forall g \in H^s(\Gamma).
\end{equation}
There exists a tubular neighborhood $T$ of $\Gamma$ that depends solely on the analyticity properties of 
$\Gamma$ and there exists a constant $\eta'\ge \eta$ depending only on $\eta$ and $\Gamma$ such that $\LG g \in \AnaClassbdy{M}{\eta'}{T}{\Gamma}$ with
\begin{equation}
\label{eq_low_pass_filter_Gamma}
M \leq C_s \|g\|_{-1/2,\Gamma}. 
\end{equation}
%We emphasize that in \eqref{eq_high_pass_filter_Gamma} 
%the constant $C_{s,s'}$ is independent of $\eta$.
\end{proposition}
\begin{proof}
The operators $\HG$ and $\LG$ are defined in \cite[Lemmas~{4.2}, 4.3]{melenk-sauter11} and the 
properties of $\HG$ are shown there. For the properties of $\LG$, we note the following: Let $T$ 
be a tubular neighborhood on which the analytic continuation $\normal^\ast$ of the normal vector $\normal$
exists and let, as in construction in \cite{melenk-sauter11},  $G \in H^{s+3/2}(\Omega)$ be such that $\partial_n G = g$ on $\Gamma$ and $\|G\|_{s'+3/2,\Omega} \lesssim \|g\|_{s',\Gamma}$ for $-1/2 \leq s' \leq s$. 
Then $\LG g$ is defined as the restriction to $\Gamma$ of $\normal^\ast \cdot \nabla \LO G$. By construction, 
$\LO G$ is analytic on ${\mathbb R}^d$ with 
\begin{equation*}
\|\nabla^n \LO G\|_{0,{\mathbb R}^d} \leq C (\eta k)^{n-1} \|G\|_{1,\Omega}  \qquad \forall n \in {\mathbb N}. 
\end{equation*}
We conclude from Lemma~\ref{lemma:lemma2.6} that $\normal^\ast \cdot \nabla \LO G \in \AnaClassbdy{C \|G\|_{1,\Omega}}{\eta'}{T}{\Gamma}$ for 
some $C$, $\eta'$ depending only on $\eta$ and $\Gamma$. Without loss generality, $\eta' \ge \eta$. We conclude with the estimate 
$\|G\|_{1,\Omega} \lesssim \|g\|_{-1/2,\Gamma}$. 
\end{proof}

\begin{proposition}[Volume filters]
\label{proposition:high_low_pass_filter_Omega_pw}
For each $\eta > 1$, there exist two operators $\LO,\HO \colon L^2(\Omega) \rightarrow L^2(\Omega)$
such that $\HO + \LO = \id$. For $0 \leq s^\prime \leq s$ and $0 \leq \varepsilon < 1/2$,
the operator $\HO$ satisfies
\begin{align}
\label{eq_high_pass_filter_Omega}
\|\HO f\|_{s^\prime,\partition}
& \leq
C_{s,s^\prime} (\eta k)^{s^\prime - s} \|f\|_{s,\partition}
\qquad
\forall f \in H^s(\partition), \\
\label{eq:estimate_high_pass_filter_negative_norms}
\|\HO f\|_{\widetilde{H}^{-\varepsilon}(\Omega)}
& \leq
C_\varepsilon
(\eta k)^{-\varepsilon} \| f \|_{0, \Omega}
\qquad
\forall f \in L^2(\Omega)
\end{align}
for constants $C_{s,s'}$, $C_\varepsilon$ independent of $f$, $k$, and $\eta$. 
% where $C_\varepsilon$ is a constant independent of $k$ and $\eta$, but possibly
% depending on $\varepsilon$.
Finally, we have $\LO f \in \AnaClass{M}{\eta}{\partition}$ with
\begin{equation}
\label{eq_low_pass_filter_Omega}
M \leq C_\Omega  \|f\|_{0,\Omega}. 
\end{equation}
\end{proposition}

\begin{proof}
For all $P \in \partition$, we set
\begin{equation*}
\left .  \left ( \HO f \right ) \right |_P
=
\left (H_{\eta k} (E_P f)\right )|_P
\qquad
\left .
\left (
\LO f
\right )
\right |_P
=
\left (L_{\eta k} (E_P f)\right )|_P,
\end{equation*}
where $E_P: H^s(P) \to H^s(\mathbb R^d)$ is the Stein
extension operator, \cite[Chap.~{VI}]{stein70},   
and $H_{\eta k}$ and $L_{\eta k}$
are the ``full space'' high and low pass filters introduced
in \cite{melenk-sauter11}. Then, the proof of \eqref{eq_high_pass_filter_Omega}
and \eqref{eq_low_pass_filter_Omega} can be found in \cite[Lemmas~{4.2}, {4.3}]{melenk-sauter11}.

Regarding \eqref{eq:estimate_high_pass_filter_negative_norms},
note that for $0 \leq \varepsilon < 1/2$ 
the space of compactly supported smooth functions $\bigcup_{P \in \partition} C^\infty_0(P)$
is dense in $H^\varepsilon(\Omega)$. It is easy to check that for all $0 \leq s \leq 1$ 
\begin{equation}
\label{eq:proposition:high_low_pass_filter_Omega_pw-100}
\|H_{\eta k} f\|_{-s,\mathbb R^d}
\leq
C_s (\eta k)^{-s} \| f \|_{0,\mathbb R^d}
\qquad \forall f \in L^2(\mathbb R^d), 
\end{equation}
\cite[Section~{4.1.1}]{melenk-sauter11}.  
For $v \in C_0^\infty(P)$ let $\widetilde v $ denote the trivial extension
by zero on all of $\mathbb{R}^d$ of $v$. These observations give
\begin{align*}
\left| (\highPassFilterOmega f,v)_\Omega\right|
& =\left|
\sum_{P \in \partition}
(H_{\eta k} E_P(f),v)_{L^2(P)}
\right|
=
\left| 
\sum_{P \in \partition}
(H_{\eta k} E_P(f),\widetilde v_P)_{L^2(\mathbb R^d)}
\right|  \\
& \leq
\sum_{P \in \partition}
\|H_{\eta k} E_P(f)\|_{-\varepsilon,\mathbb R^d} \|\widetilde v_P\|_{\varepsilon,\mathbb R^d}.
\end{align*}
Additionally, we have
\begin{align*}
\|H_{\eta k} E_P(f)\|_{-\varepsilon,\mathbb R^d}
& \stackrel{\eqref{eq:proposition:high_low_pass_filter_Omega_pw-100}}{ \leq}
C_\varepsilon
(\eta k)^{-\varepsilon} \|E_P(f)\|_{0,\mathbb R^d}
\leq
C_\varepsilon
(\eta k)^{-\varepsilon} \|f\|_{0,P},  \\
\|\widetilde v\|_{\varepsilon,\mathbb R^d}
&\stackrel{\varepsilon <1/2}{\lesssim}
C_\varepsilon \|v\|_{\varepsilon,P}, 
\end{align*}
which leads to 
\begin{equation*}
\left| (\HO f,v)_{L^2(\Omega)}\right| 
\leq
C_\varepsilon
(\eta k)^{-\varepsilon}
\sum_{P \in \partition}
\|f\|_{0,P} \|v\|_{\varepsilon,P}
\leq
C_\varepsilon
(\eta k)^{-\varepsilon}
\|f\|_{0,\Omega} \|v\|_{\varepsilon,\partition},
\end{equation*}
and the result follows since
$\|v\|_{\varepsilon,\partition} \leq \|v\|_{\varepsilon,\Omega}$. 
\end{proof}

\begin{remark}[Other constructions of high and low pass filters]
The presented high and low pass filters are by no means the only possible constructions.
It is, e.g., possible to construct the low-frequency filter on bounded domains 
by truncating the expansion in eigenfunctions of the Laplacian endowed with Neumann boundary conditions. 
See also \cite[Sec.~{6.1}]{melenk12} for similar considerations.
\eremk
\end{remark}

%--------------------------------------------
\subsection{A contraction argument}
%--------------------------------------------

We employ the low and high pass filters to formulate a preliminary 
decomposition result. For a given right-hand side, the decomposition
consists of splitting of the desired form plus a remainder that is the solution to
the time-harmonic problem with new right-hand side that is stricly smaller
in norm than the original one. We thus call the preliminary result a ``contraction''
argument. This contraction argument is key in our analysis, and established
in Lemma~\ref{lemma_contraction} below.

Due to the duality between $H^\varepsilon(\Omega)$ and $\widetilde H^{-\varepsilon}(\Omega)$,
we have for $f \in L^2(\Omega)$ and $g \in L^2(\Gamma)$
with a constant $C_\varepsilon$ depending only on $\varepsilon \in [0,1]$ 
\begin{equation*}
|(f,v) + \langle g,v \rangle|
\leq
C_\varepsilon \left (
k^{\varepsilon-1} \|f\|_{\widetilde{H}^{-\varepsilon}(\Omega)}
+
k^{-1/2} \|g\|_{0,\Gamma}
\right )
\|v\|_{1,k,\Omega}
\quad
\forall v \in H^1(\Omega).
\end{equation*}
Since $\bkp$ is coercive by \ref{AP1}, this leads to
\begin{equation}
\label{eq_stability_Skp}
\|\Skp(f,g)\|_{1,t,k}
\leq
C_\varepsilon
\left (
k^{\varepsilon-1} \| f \|_{\widetilde{H}^{-\varepsilon}(\Omega)}
+
k^{-1/2} \|g\|_{0, \Gamma}
\right ).
\end{equation}

\begin{lemma}[Unified contraction argument]
\label{lemma_contraction}
Let \WPs and \APs hold.
Let $q \in (0, 1)$ and $s \in [0,\smax]$ be given. Then, there exists
a parameter $\eta > 1$ depending on $q$ and $s$ such that for all $f \in H^s(\partition)$
and $g \in H^{s+1/2}(\Gamma)$ we can write
\begin{equation*}
\Skm(f,g) = u_{\rm F} + u_{\rm A} + \Skm(\widetilde{f},\widetilde{g}), 
\end{equation*}
where
\begin{equation*}
u_{\rm F}
\coloneqq
\Skp(\HO f, \HG g)
\qquad
u_{\rm A}
\coloneqq
\Skm(\LO f,\LG g)
+
\Skm(\AOD u_{\rm F},\AGD u_{\rm F}),
\end{equation*}
and the pair $(\tf,\tg) \in H^s(\partition) \times H^{s+1/2}(\Gamma)$ satisfies 
\begin{equation}
\label{eq_contraction_data}
\|\tf\|_{s,\partition} + \|\tilde g\|_{s+1/2,\Gamma}
\leq
q (\| f \|_{s,\partition} + \| g \|_{s+1/2,\Gamma}).
\end{equation}
In addition, we have $u_{\rm F} \in H^{s+2}(\LP)$ with
\begin{equation}
\label{eq_contraction_uF}
\|u_{\rm F}\|_{s+2,\partition}
\leq
C_s % C_s \CshiftP
\left (
\|f\|_{s,\partition} + \|g\|_{s+1/2,\Gamma}
\right ),
\end{equation}
% with a constant $C_s$ possibly depending on $s$ but not on $k$,
%and $u_{\rm A} \in \AnaClass{M}{\anashift{\eta}}{\partition}$ with
and $u_{\rm A} \in \AnaClass{M}{\eta'}{\partition}$ for some
$\eta'>1$ depending only on $\eta$ and $\Gamma$ with 
\begin{equation}
\label{eq_contraction_uA}
M
\leq
 C \Can \CsolveM(1 + k^{-2} \CAD) (\|f\|_{0,\Omega} + \|g\|_{1/2,\Gamma}).
% \frac{\CshiftM\CsolveM(\CAD + C)}{k} (\|f\|_{0,\Omega} + \|g\|_{1/2,\Gamma}),
\end{equation}
% where the constant $C$ does not depend on $k$ or $s$.
\end{lemma}

\begin{proof}
Let $\eta > 1$ to be fixed later on. For $f \in H^s(\Omega)$ and $g \in H^{s+1/2}(\Gamma)$,
we set $u \coloneqq \Skm(f,g)$. We introduce
\begin{equation*}
u_{\rm F} \coloneqq \Skp(\HO f,\HG g), 
\quad
u_{{\rm A,I}} \coloneqq \Skm(\LO f,\LG g)
\end{equation*}
so that $u = u_{\rm F} + u_{\rm A,I} + r_{\rm I}$ with
\begin{equation*}
r_{\rm I}
\coloneqq
- \Skm(\TOD u_{\rm F},\TGD u_{\rm F})
= -
\Skm(\ROD u_{\rm F},\RGD u_{\rm F})
-
\Skm(\AOD u_{\rm F},\AGD u_{\rm F}),
\end{equation*}
where we employed \ref{AP3}. We then introduce
\begin{equation*}
u_{{\rm A},\romnum{2}} \coloneqq -\Skm(\AOD u_{\rm F},\AGD u_{\rm F}),
\quad
\tf \coloneqq -\ROD u_{\rm F},
\quad
\tg \coloneqq -\RGD u_{\rm F},
\end{equation*}
leading to the splitting $u = u_{\rm F} + u_{\rm A} + \Skm(\tf,\tg)$
with
\begin{equation*}
u_{\rm A}
\coloneqq
u_{{\rm A},\romnum{1}} + u_{{\rm A},\romnum{2}}
=
\Skm(\LO f,\LG g)
-
\Skm(\AOD u_{\rm F},\AGD u_{\rm F}).
\end{equation*}

Before establishing the contraction estimate \eqref{eq_contraction_data},
we prove the regularity estimate \eqref{eq_contraction_uF}.
Using \eqref{eq_shift_Skp} from \ref{AP2} as well as the mapping properties
\eqref{eq_high_pass_filter_Gamma} and \eqref{eq_high_pass_filter_Omega}
of the high pass filters, we have
\begin{align*}
& \|u_{\rm F}\|_{s+2,\partition} \\
&
\stackrel{\ref{AP2}}{\leq}
C_s % \CshiftP
\left (
\|\HO f\|_{s,\partition} + k^s \| \HO f\|_{0,\Omega}
+
\|\HG g\|_{s+1/2,\Gamma} + k^{s+1/2} \|\HG g\|_{0,\Gamma}
\right )
\\
&
\stackrel{\eqref{eq_high_pass_filter_Gamma},\eqref{eq_high_pass_filter_Omega}}{\leq}
C_s % C \CshiftP
\left (
\|f\|_{s,\partition} + k^s (k\eta)^{-s} \|f\|_{s,\partition}
+
\|g\|_{s+1/2,\Gamma} + k^{s+1/2} (k\eta)^{-s-1/2} \|g\|_{s+1/2,\Gamma}
\right )
\\
&\leq
C_s % C \CshiftP
\left (
\|f\|_{s,\partition} + \|g\|_{s+1/2, \Gamma}
\right ), 
\end{align*}
which proves \eqref{eq_contraction_uF}.

We next establish \eqref{eq_contraction_data}. Recalling the definition of $\tf$ and $\tg$,
estimates \eqref{eq_ROD} and \eqref{eq_RGD} combined with trace inequality
\eqref{eq_weighted_trace} show that
\begin{equation}
\label{eq:tmp_contraction-200}
\|\tf\|_{s,\partition}
+
\|\tg\|_{s+1/2,\Gamma}
\leq C \left (
k^2 \|u_{\rm F}\|_{s,\partition} + k\|u_{\rm F}\|_{s+1,\partition}
\right ).
\end{equation}
% where $C$ does not depend on $k$.

The case $s=0$ is treated separately.
Fixing $\varepsilon = 1/4$, we have from \eqref{eq_stability_Skp} and 
the properties \eqref{eq_high_pass_filter_Gamma} and
\eqref{eq:estimate_high_pass_filter_negative_norms} of the high-pass filters 
\begin{align}
\nonumber 
k \|u_{\rm F}\|_{1,t,k} 
%k^2 \|u_{\rm F}\|_{0,\Omega} + k\|u_{\rm F}\|_{1,\Omega}
& \stackrel{\eqref{eq_stability_Skp}}{\leq}
C_\varepsilon k
\left (
k^{\varepsilon-1}\|\HO f\|_{\widetilde H^{-\varepsilon}(\Omega)} + k^{-1/2} \|\HG g\|_{0,\Gamma}
\right ) \\
\nonumber 
&\stackrel{\eqref{eq_high_pass_filter_Gamma}, \eqref{eq:estimate_high_pass_filter_negative_norms}}{\leq}
C_\varepsilon
k
\left (k^{\varepsilon-1}(k\eta)^{-\varepsilon} \|f\|_{0,\Omega}
+ k^{-1/2} (k\eta)^{-1/2} \|g\|_{1/2,\Gamma} \right )
\\
\label{eq:tmp_contraction-100}
&\leq
C \eta^{-\varepsilon}
\left ( \|f\|_{0,\Omega} + \|g\|_{1/2,\Gamma} \right ), 
\end{align}
where, since $\varepsilon=1/4$ is fixed, we absorbed the constant $C_\varepsilon$
in a constant $C$ independent of $k$ and $\varepsilon$; we also used that
$\eta^{-1/2} \leq \eta^{-\varepsilon}$ since $\eta > 1$.
We conclude from 
\eqref{eq:tmp_contraction-200} and 
\eqref{eq:tmp_contraction-100}
%Then properties \eqref{eq_high_pass_filter_Gamma} and
%\eqref{eq:estimate_high_pass_filter_negative_norms} of the high-pass filters show 
\begin{align}
\label{tmp_contraction_s0}
\|\tf\|_{0,\Omega} + \|\tg\|_{1/2,\Gamma}
&\leq
%C_\varepsilon
%k
%\left (k^{\varepsilon-1}(k\eta)^{-\varepsilon} \|f\|_{0,\Omega}
%+ k^{-1/2} (k\eta)^{-1/2} \|g\|_{1/2,\Gamma} \right )
%\\
%&\leq
C \eta^{-\varepsilon}
\left ( \|f\|_{0,\Omega} + \|g\|_{1/2,\Gamma} \right ). 
\end{align}

We turn to the case $0 < s \leq \smax$. With $\varepsilon = 0$ in \eqref{eq_stability_Skp} we have
together with the mapping properties \eqref{eq_high_pass_filter_Gamma} and
\eqref{eq_high_pass_filter_Omega} of the high-pass filters that
\begin{align*}
k^2 \|u_{\rm F}\|_{0,\Omega}
&\leq
C k
\left (
k^{-1}\|\HO f\|_{0,\Omega} + k^{-1/2} \|\HG g\|_{0,\Gamma}
\right ) \\
&\leq
C_s
(k \eta)^{-s} \|f\|_{s,\partition} + k^{1/2} (k \eta)^{-s-1/2} \|g\|_{s+1/2,\Gamma}.
\end{align*}
Hence, we find
\begin{equation}
\label{eq_contraction_uF_zero_norm}
k^{s+2} \|u_{\rm F}\|_{0,\Omega}
\leq
C_s
\eta^{-s} \left (\|f\|_{s,\partition} + \|g\|_{s+1/2,\Gamma} \right ).
\end{equation}
% with a constant $C$ independent of $k$.
We note the multiplicative interpolation estimate for $0 \leq r \leq s+2$: 
\begin{equation}
\label{eq_interpolation_r_tmp}
\|u_{\rm F}\|_{r,\partition}
\leq
C_r
\|u_{\rm F}\|_{0,\Omega}^{\frac{s+2-r}{s+2}}
\|u_{\rm F}\|_{s+2,\partition}^{\frac{r}{s+2}}.
\end{equation}
Inserting \eqref{eq_contraction_uF_zero_norm} and \eqref{eq_contraction_uF}
into \eqref{eq_interpolation_r_tmp} we find
\begin{equation}
\label{eq_interpolation_r}
\|u_{\rm F}\|_{r,\partition}
\leq
C_r k^{r-s-2} \eta^{-s\frac{s+2-r}{s+2}}
\left (\|f\|_{s,\partition} + \|g\|_{s+1/2,\Gamma}\right ).
\end{equation}
Using \eqref{eq_interpolation_r} with $r = s$ and with $r=s+1$ in 
\eqref{eq:tmp_contraction-200}
 we find
\begin{equation}
\label{tmp_contraction_sgreaterzero}
\begin{alignedat}{1}
\|\tf\|_{s,\Omega}
+
\|\tg\|_{s+1/2,\Gamma}
&\leq
C_s 
\left (
k^2 \|u_{\rm F}\|_{s,\partition} + k\|u_{\rm F}\|_{s+1,\partition}.
\right )
\\
&\leq
C_s % C_s \CshiftP
\left( \eta^{-s\frac{2}{s+2}} + \eta^{-s\frac{1}{s+2}} \right)
\left( \|f\|_{s,\partition} + \|g\|_{s+1/2, \Gamma}\right).
\end{alignedat}
\end{equation}
% where the constant $C_s$ only depend on $s$ and $\partition$.
At this point, since the constants appearing in \eqref{tmp_contraction_s0} and
\eqref{tmp_contraction_sgreaterzero} are independent of $\eta$, $f$, $g$, we can
select $\eta$ large enough ($\eta$ may depend on $s$) to obtain \eqref{eq_contraction_data}
with $q < 1$. For the sake of simplicity, we additionally select $\eta$ larger than
the constants $\gamma_0$ of \ref{WP4} and $\gAD$ of \ref{AP3}.

We next establish \eqref{eq_contraction_uA}. For the first component $u_{\rm A,I}$,
properties \eqref{eq_low_pass_filter_Gamma} and \eqref{eq_low_pass_filter_Omega}
of the low pass filters show that
\begin{equation*}
\LO f \in \AnaClass{M_f}{\eta}{\partition},
\qquad
\LG g \in \AnaClassbdy{M_g}{\eta'}{T}{\Gamma}
\end{equation*}
with
\begin{equation*}
M_f \leq C \|f\|_{0,\Omega},
\qquad
M_g \leq C \|g\|_{-1/2,\Gamma}
\end{equation*}
and a tubular neighborhood $T$ of $\Gamma$ and a constant $\eta' \ge \eta$ that depend solely on $\eta$ and $\Gamma$. 
% and $C$ only depends on $k_0$ (recall that $\eta > 1$ and $k \geq k_0$).
From \ref{WP3} and $\eta' \ge \eta \geq \gamma_0$, we get 
$u_{{\rm A,I}} \in \AnaClass{M_{\rm I}}{\anashift{\eta'}}{\partition}$ with
\begin{equation*}
M_{\rm I}
\leq
\Can \CsolveM k^{-1}
% C \Can \CsolveM k^{-1}
\left (M_f + k M_g \right )
\leq
 C \Can \CsolveM k^{-1} 
% C \frac{\CshiftM\CsolveM}{k}
\left (\|f\|_{0,\Omega} + k \|g\|_{-1/2,\Gamma} \right ).
\end{equation*}
We turn to estimating $u_{\rm A,II}$. From \eqref{eq:tmp_contraction-100}
we get $\|u_{\rm F}\|_{1,t,k} \lesssim k^{-1} \left( \|f\|_{0,\Omega} + \|g\|_{1/2,\Gamma}\right)$. 
Using \ref{AP3}, we arrive at 
\begin{equation*}
f_{\rm II} \coloneqq \AOD u_{\rm F} \in \AnaClass{M_{f,{\rm II}}}{\gAD}{\partition}, 
\qquad
g_{\rm II} \coloneqq \AGD u_{\rm F} \in \AnaClassbdy{M_{g,{\rm II}}}{\gAD}{T}{\partition}
\end{equation*}
with 
\begin{equation*}
M_{f,{\rm II}} + k M_{g,{\rm II}} \leq
\CAD k^{-1} \left (\|f\|_{0,\Omega} + \|g\|_{1/2,\Gamma}\right ).
\end{equation*}
Since we chose $\eta$ larger than $\gamma_0$ and $\gAD$, we also have
\begin{equation*}
f_{\rm II} \coloneqq \AOD u_{\rm F} \in \AnaClass{M_{f,{\rm II}}}{\eta'}{\partition}, 
\qquad
g_{\rm II} \coloneqq \AGD u_{\rm F} \in \AnaClassbdy{M_{g,{\rm II}}}{\eta'}{T}{\Gamma}, 
\end{equation*}
and it follows from \ref{WP3} that
$u_{\rm A,II} \in \AnaClass{M_{\rm II}}{\anashift{\eta'}}{\partition}$ with
\begin{equation*}
M_{\rm II}
\leq
 C \Can \CsolveM k^{-1}
% \frac{\CshiftM\CsolveM}{k}
\left (M_{f,{\rm II}} + k M_{g,{\rm II}}\right )
\leq
 C \Can \CsolveM k^{-2} \CAD
% \frac{\CshiftM\CsolveM\CAD}{k}
\left (
\|f\|_{0,\Omega} + \|g\|_{1/2,\Gamma}
\right ).
\end{equation*}
Estimate \eqref{eq_contraction_uA} now follows since
\begin{equation*}
u_{\rm A}
\coloneqq
u_{\rm A,I}
+
u_{\rm A,II}
\in
\AnaClass{M_{\rm I} + M_{\rm II}}{\anashift{\eta'}}{\partition}.
\end{equation*}
Relabelling the parameter $\anashift{\eta'}$ as $\eta'$ 
and estimating generously $\|g\|_{-1/2,\Gamma} \lesssim \|g\|_{1/2,\Gamma}$ 
gives the result. 
\end{proof}

\subsection{Regularity splitting of time-harmonic solutions}

The contraction result of Lemma~\ref{lemma_contraction} can be iterated to yield 
a splitting of the solution of Helmholtz problems into a part $u_{\rm F}$ with finite
regularity and an analytic part $u_{\rm A}$:  

\begin{theorem}[Abstract regularity splitting]
\label{theorem_splitting}
Assume \WPs and \AP. 
For all $0 \leq s \leq \smax$, and for each $f \in H^s(\partition)$ and $g \in H^{s+1/2}(\Gamma)$ there is a 
splitting
\begin{equation*}
\Skm(f,g) = u_{\rm F} + u_{\rm A}
\end{equation*}
with $u_{\rm F} \in H^{s+2}(\partition)$ satisfying
\begin{equation}
\label{eq_splitting_uF}
\|u_{\rm F}\|_{s+2, \partition}
\leq
C_s % C_s \CshiftP
\left ( \|f\|_{s,\partition} + \|g\|_{s+1/2,\Gamma} \right ),
\end{equation}
%and $u_{\rm A} \in \AnaClass{M}{\anashift{\anasplit{s}}}{\partition}$ with
and $u_{\rm A} \in \AnaClass{M}{\anasplit{s}}{\partition}$ with
\begin{equation}
\label{eq_splitting_uA}
M \leq
 \Can C\CsolveM (1 + k^{-2} \CAD)
% C\frac{\CshiftM \CsolveM (1 + \CAD)}{k}
(\| f \|_{0,\Omega} + \| g \|_{1/2,\Gamma}),
\end{equation}
and a constant $\anasplit{s}$ depending solely on $s$ and the parameters appearing in \WP, \AP. 
\end{theorem}

\begin{proof}
The proof iterates the contraction argument in Lemma \ref{lemma_contraction}.
To fix ideas, we select $q \coloneqq 1/2$ and denote by
$\eta_s$, $\eta_s'$ the parameters $\eta$, $\eta'$ given by Lemma~\ref{lemma_contraction}.
Letting $f^{(0)} \coloneqq f$ and $g^{(0)} \coloneqq g$, the splitting
\begin{equation*}
\Skm(f^{(0)}, g^{(0)}) = u_{\rm F}^{(0)} + u_{\rm A}^{(0)} + \Skm(f^{(1)}, g^{(1)}),
\end{equation*}
holds true with contracted data
\begin{equation*}
\| f^{(1)} \|_{s,\partition} + \|g^{(1)}\|_{s+1/2,\Gamma}
\leq
q (\| f^{(0)} \|_{s,\partition} + \| g^{(0)} \|_{s+1/2,\Gamma});
\end{equation*}
the regular part $u_{\rm F}^{(0)} \in H^{s+2}(\partition)$ satisfies 
\begin{equation*}
\|u_{\rm F}^{(0)}\|_{s+2,\partition}
\leq
C_s % C_s \CshiftP
\left (
\|f^{(0)}\|_{s,\partition} + \|g^{(0)}\|_{s+1/2,\Gamma}
\right )
\end{equation*}
%and the analytic part $u_{\rm A}^{(0)} \in \AnaClass{M^{(0)}}{\anashift{\eta_s}}{\partition}$
and the analytic part $u_{\rm A}^{(0)} \in \AnaClass{M^{(0)}}{{\eta_s'}}{\partition}$
with
\begin{equation*}
M^{(0)}
\leq
C\CsolveM (1+k^{-2} \CAD)
% \frac{\CshiftM\CsolveM (\CAD+C)}{k}
\left (\|f^{(0)}\|_{0,\Omega} + \|g^{(0)}\|_{1/2,\Gamma}\right ).
\end{equation*}
We can then repeat this argument to split $\Skm(f^{(1)},g^{(1)})$, and so on,
resulting in an inductive definition of
$(f^{(i)},g^{(i)}) \in H^s(\partition) \times H^{s+1/2}(\Gamma)$,
$u_{\rm F}^{(i)} \in H^{s+2}(\partition)$ and
%$u_{\rm A}^{(i)} \in \AnaClass{M^{(i)}}{\anashift{\eta_s}}{\partition}$
$u_{\rm A}^{(i)} \in \AnaClass{M^{(i)}}{{\eta_s'}}{\partition}$
with
\begin{align*}
\| f^{(i)} \|_{s,\partition} + \|g^{(i)}\|_{s+1/2,\Gamma}
&\leq 
q^i (\| f^{(0)} \|_{s,\partition} + \| g^{(0)} \|_{s+1/2,\Gamma}),
\\
\|u_{\rm F}^{(i)}\|_{s+2,\partition}
&\leq
C_s % \CshiftP
q^i
\left (
\|f^{(0)}\|_{s,\partition} + \|g^{(0)}\|_{s+1/2,\Gamma}
\right )
\\
M^{(i)}
&\leq 
q^i
C\CsolveM (1+k^{-2} \CAD)
% \frac{\CsolveM \CshiftM(C+\CAD)}{k}
\left (\|f^{(0)}\|_{0,\Omega} + \|g^{(0)}\|_{1/2,\Gamma}\right ).
\end{align*}
It follows that we have
\begin{equation*}
u =
\underbrace{\sum_{i \in \mathbb{N}_0} u_{F}^{(i)}}_{\eqqcolon u_{F}}
+
\underbrace{\sum_{i \in \mathbb{N}_0} u_{A}^{(i)}}_{\eqqcolon u_{A}},
\end{equation*}
where the two series converge due to the geometric decrease of 
the factors $q^i$. 
Then, introducing $\anasplit{s} \coloneqq \eta_s'$, estimates
\eqref{eq_splitting_uF} and \eqref{eq_splitting_uA} simply follow since
$q = 1/2$ and $\sum_{q \in \mathbb N_0} q^i = 2$. 
\end{proof}

% !TEX root = main.tex
\section{Stability and Convergence of Abstract Galerkin Discretizations}\label{section:galerkin}

In this section, we utilize the splitting introduced in Theorem~\ref{theorem_splitting} to analyze Galerkin
discretizations. We first establish quasi-optimality and error estimates for abstract
discrete spaces. These results are then applied to the particular case of the $hp$ finite
element method, where explicit stability conditions are derived.

The proofs of this section hinge on duality techniques, and require properties
of the adjoint wave propagation problem. For the sake of simplicity, we assume here
that the considered wave propagation and auxiliary problems are symmetric in the sense
that
\begin{equation}
\label{eq_symmetry}
\bkm(v,u) = \bkm(\overline{u},\overline{v}), 
\qquad
\bkp(v,u) = \bkp(\overline{u},\overline{v}) 
\qquad 
\forall u,v \in \HOneT. 
\end{equation}
We mention that \eqref{eq_symmetry} is not mandatory; however,
without \eqref{eq_symmetry} one needs to assume \WPs and \APs for the
adjoint problems rather than the primal problem, which we avoid here to simplify the presentation. 

%------------------------------------------------
\subsection{Abstract Galerkin discretizations}
%------------------------------------------------

We start by considering a generic finite-dimensional subspace $V_h \subset \HOneT$.
The corresponding Galerkin approximation $u_h \in V_h$ to the exact solution
$u \in \HOneT$ of \eqref{eq:S_k_minus_problem} is given by 
\begin{equation}
\label{eq:Galerkin-approximation}
\bkm(u_h,v_h) = (f,v_h) + \langle g , v_h \rangle \qquad \forall v_h \in  V_h. 
\end{equation}
The convergence analysis of the Galerkin method is based on a duality argument and thus hinges on the ability of
$V_h$ to approximate solution to the (adjoint) Helmholtz problem. 
Following
previous works on the subject
\cite{chaumontfrelet-nicaise19,lafontaine-spence-wunsch_2020,melenk-sauter10,sauter06,schatz74},
the key approximation properties of $V_h$ are encoded in real numbers $\eta$, called approximation factors.
Specifically,
for $m \ge 1$, $\gamma > 0$, and the (fixed) tubular neighborhood $T$ of $\Gamma$ given as 
the intersection of the two tubular neighborhoods given by \ref{WP4} and \ref{AP3},
these adjoint approximation factors $\etaF{m}$, $\etaA{\gamma}$ are defined as follows:
\begin{subequations}
\label{eq_approximation_factors}
\begin{align}
\label{eq_definition_etaF}
\etaF{m}
& \coloneqq
\sup_{\substack{
(f,g) \in H^{m-1}(\partition) \times H^{m-1/2}(\Gamma)
\\
\|f\|_{m-1,\partition} + \|g\|_{m-1/2,\Gamma} = 1
}}
\inf_{v_h \in V_h} \|\Skm(f,g)-v_h\|_{1,t,k}, \\
\label{eq_definition_etaA}
\etaA{\gamma}
& \coloneqq
\sup_{\substack{
(f,g) \in \AnaClass{M_f}{\gamma}{\partition} \times \AnaClassbdy{M_g}{\gamma}{T}{\Gamma}
\\
M_f + k M_g = 1/2
}}
\inf_{v_h \in V_h} \|\Skm(f,g)-v_h\|_{1,t,k},
\end{align}
\end{subequations}
To ease the presentation, we introduce the projection
\begin{equation*}
\Pi_h v \coloneqq \arg \min_{v_h \in V_h} \|v-v_h\|_{1,t,k}. 
\end{equation*}
%which is well-defined for all $v \in \HOneT$ due to convexity considerations.

\begin{lemma}
\label{lemma_continuity}
Let $u \in \HOneT$ and assume that $u_h \in V_h$ satisfies
\begin{equation}
\label{eq_galerkin_orthogonality}
\bkm(u-u_h,v_h) = 0 \quad \forall v_h \in V_h.
\end{equation}
Then, the estimate
\begin{equation}
\label{eq:cont_error}
|\bkm(u-u_h,v)|
\leq
C \left (1 + \CcontM \CAM \etaA{\anashift{\gAM}} \right )
\|u-u_h\|_{1,t,k}\|v\|_{1,t,k}
\end{equation}
holds true for all $v \in H^{1,t}(\Omega, \Gamma)$, where $C$
is a constant independent of $k$, $u$, and $v$.
\end{lemma}

\begin{proof}
For the sake of brevity we write $e_h \coloneqq u-u_h$. For an arbitrary
$v \in H^{1,t}(\Omega,\Gamma)$, we have
\begin{align*}
|\bkm(e_h,v)|
&\stackrel{\eqref{eq_symmetry}}{=}| \bkm(\overline{v},\overline{e_h})| 
= |
a(\overline{v}, \overline{e_h}) 
-
(\TkOmegaM \overline{v} , \overline{e_h} )
-
\langle \TGM \overline{v} , \overline{e_h} \rangle
|
\\
&\stackrel{\eqref{eq_splitting_TOM_TGM}}{\leq}
|a(\overline{v}, \overline{e_h})|
+
|(\ROM \overline{v} , \overline{e_h}) |
+
|\langle \RGM \overline{v},\overline{e_h} \rangle|
+
|(\AOM \overline{v},\overline{e_h}) + \langle \AGM \overline{v},\overline{e_h} \rangle|
\\
&\stackrel{\eqref{eq_continuity_ROM_RGM},\eqref{eq_continuity_a}}{\leq}
C\| e_h \|_{1,t,k} \| v \|_{1,t,k}
+
|(\AOM \overline{v},\overline{e_h}) + \langle \AGM \overline{v},\overline{e_h} \rangle|.
\end{align*}
% where $C$ does not depend on $k$.
%
The analytic part is now treated using a duality argument: 
Using \eqref{eq_symmetry}, we have
by setting $\overline{\phi} \coloneqq \Skm(\AOM\overline{v},\AGM\overline{v})$
\begin{equation*}
(\AOM \overline{v} , \overline{e_h})
+
\langle \AGM \overline{v} , \overline{e_h} \rangle
\stackrel{\eqref{eq_symmetry}}{=}
\bkm(\Skm(\AOM\overline{v},\AGM\overline{v}),\overline{e_h})
=
\bkm(e_h,\phi)
=
\bkm(e_h,\phi-\Pi_h \phi).
\end{equation*}
%where $\overline{\phi} \coloneqq \Skm(\AOM\overline{v},\AGM\overline{v})$.
Then, using \ref{WP4}, we have
\begin{equation*}
\|\phi-\Pi_h \phi\|_{1,t,k}
\! = \! 
\|\overline{\phi}-\Pi_h \overline{\phi}\|_{1,t,k}
= \!\!\!
\inf_{w_h \in V_h} \!
\|\Skm(\AOM\overline{v},\AGM\overline{v})-w_h\|_{1,t,k}
\leq
\CAM \etaA{\anashift{\gAM}}\|v\|_{1,t,k},
\end{equation*}
and the conclusion follows from \ref{WP1} since
\begin{equation*}
|(\AOM \overline{v} , \overline{e_h})
+
\langle \AGM \overline{v} , \overline{e_h} \rangle|
=
|\bkm(e_h,\phi-\Pi_h\phi)|
\leq
\CcontM \|e_h\|_{1,t,k}\|\phi-\Pi_h\phi\|_{1,t,k}.
\qedhere
\end{equation*}
\end{proof}

\begin{lemma}
\label{lemma_quasi_optimality_explicit}
Assume that $u \in \HOneT$ and $u_h \in V_h$ are such that
\begin{equation}
\label{eq_galerkin_orthogonality_2}
\bkm(u-u_h,v_h) = 0 \quad \forall v_h \in V_h.
\end{equation}
Then, there exist two constants $\rho$ and $\mu$ independent of $k$, $u$, and $u_h$ such that
\begin{multline}
\label{eq_quasi_optimality_explicit}
\left (
\mu -
\left (1 + \CcontM\CAM \etaA{\anashift{\gAM}}\right )
\left ( k\etaF{1} + \CAD \etaA{\anashift{\gAD}} \right )
\right )
\|u-u_h\|_{1,t,k}
\\
\leq
\rho
\left (
1 + \CcontM \CAM \etaA{\anashift{\gAM}}
\right )
\|u-\Pi_h u\|_{1,t,k}.
\end{multline}
\end{lemma}

\begin{proof}
We introduce the Galerkin error $e_h \coloneqq u-u_h$. 
Recalling \eqref{eq_coercivity_bkp} and \eqref{eq_galerkin_orthogonality_2},
%there exists a constant $C_0$ independent of $k$, $u$, and $u_h$ such that
we estimate with $C_0:=1/\CcoerP $ 
\begin{multline*}
\frac{1}{C_0}
\|e_h\|_{1,t,k}^2
\leq
\Re \bP(e_h,\sigma e_h)
\leq
|\bP(e_h,e_h)|
\leq
\\
|\bD(e_h,e_h)| + |\bM(e_h,e_h)|
=
|\bD(e_h,e_h)| + |\bM(e_h,u-\Pi_h u)|.
\end{multline*}
With Lemma~\ref{lemma_continuity}  we may then estimate 
\begin{equation*}
\|e_h\|_{1,t,k}^2
\leq
C_0 |\bD(e_h,e_h)|
+
C_1 \left (1 + \CcontM\CAM \etaA{\anashift{\gAM}}\right )
\|e_h\|_{1,t,k} \|u-\Pi_h u\|_{1,t,k},
\end{equation*}
where $C_1$ does not depend on $k$, $u$ or $u_h$. On the other hand, we have
upon setting $\overline{\phi_{\rm F}} \coloneqq \Skm(\ROD\overline{e_h},\RGD \overline{e_h})$
and
$\overline{\phi_{\rm A}} \coloneqq \Skm(\AOD\overline{e_h},\AGD \overline{e_h})$:
\begin{align*}
\bD(e_h,e_h)
&\stackrel{\eqref{eq_symmetry}}{=}
\bD(\overline{e_h},\overline{e_h})
\\ &
\stackrel{\eqref{eq_splitting_TOD_TGD}}{=}
(\ROD \overline{e_h},\overline{e_h}) + \langle \RGD \overline{e_h},\overline{e_h} \rangle
+
(\AOD \overline{e_h},\overline{e_h}) + \langle \AGD \overline{e_h},\overline{e_h} \rangle
\\
&=
\bkm(\Skm(\ROD\overline{e_h},\RGD \overline{e_h}),\overline{e_h})
+
\bkm(\Skm(\AOD\overline{e_h},\AGD \overline{e_h}),\overline{e_h})
\\
&=
\bkm(e_h,\phi_{\rm F}) + \bkm(e_h,\phi_{\rm A}).
\end{align*}
%where
%$\overline{\phi_{\rm F}} \coloneqq \Skm(\ROD\overline{e_h},\RGD \overline{e_h})$
%and
%$\overline{\phi_{\rm A}} \coloneqq \Skm(\AOD\overline{e_h},\AGD \overline{e_h})$.

Using Lemma~\ref{lemma_continuity}, we can then write
\begin{align*}
& |\bkm(e_h,\phi_{\rm F})|
=
|\bkm(e_h,\phi_{\rm F}-\Pi_h\phi_{\rm F})|
\\
&\qquad \stackrel{\text{L.~\ref{lemma_continuity}}}{\leq}
C \left (1 + \CcontM \CAM \etaA{\anashift{\gAM}} \right )
\|e_h\|_{1,t,k}
\|\phi_{\rm F}-\Pi_h\phi_{\rm F}\|_{1,t,k}
\\
&\qquad \stackrel{(\ref{eq_definition_etaF})}{\leq}
C \left (1 + \CcontM \CAM \etaA{\anashift{\gAM}} \right )
\etaF{1} \|e_h\|_{1,t,k}
\left (
\|\ROD \overline{e_h}\|_{0,\Omega}
+
\|\RGD \overline{e_h}\|_{1/2,\Omega}
\right )
\\
&\qquad \stackrel{\ref{AP3}}{\leq}
C_2
\left (1 + \CcontM \CAM \etaA{\anashift{\gAM}} \right )
k\etaF{1} \|e_h\|_{1,t,k}^2,
\end{align*}
where $C_2$ is independent of $k$, $u$, and $u_h$.
Note that in the above estimate, we only use \ref{AP3}
for the operators $\ROD$ and $\RGD$, i.e., equations
\eqref{eq_ROD} and \eqref{eq_RGD}, which does not introduce
any dependence on $k$ in the constant $C_2$. Similarly, we have
\begin{align*}
|\bkm(e_h,\phi_{\rm A})|
&\stackrel{\text{L.~\ref{lemma_continuity}}} {\leq}
C \left (1 + \CcontM \CAM \etaA{\anashift{\gAM}} \right )
\|e_h\|_{1,t,k}
\|\phi_{\rm A}-\Pi_h\phi_{\rm A}\|_{1,t,k}
\\
&\leq
C \left (1 + \CcontM \CAM \etaA{\anashift{\gAM}} \right )
\CAD \etaA{\anashift{\gAD}} \|e_h\|_{1,t,k}^2,
\end{align*}
which leads to
\begin{equation*}
|\bD(e_h,e_h)|
\leq
C_3 \left (1 + \CcontM \CAM \etaA{\anashift{\gAM}} \right )
\left (
k\etaF{1} + \CAD \etaA{\anashift{\gAD}}
\right )
\|e_h\|_{1,t,k}^2,
\end{equation*}
with $C_3$ independent of $k$, $u$, and $u_h$.  Therefore, 
\begin{multline}
\label{tmp_quasi_optimality}
\left (1 - C_0C_3 \left (1 + \CcontM \CAM \etaA{\anashift{\gAM}} \right )
\left (
k\etaF{1} + \CAD \etaA{\anashift{\gAD}}
\right )
\right )
\|e_h\|_{1,t,k}^2
\leq
\\
C_1 \left (1 + \CcontM\CAM \etaA{\anashift{\gAM}}\right )
\|e_h\|_{1,t,k} \|u-\Pi_h u\|_{1,t,k}.
\end{multline}
Estimate \eqref{eq_quasi_optimality_explicit} then follows by setting
$1/\mu \coloneqq C_0 \max(1,C_3)$ and $\rho \coloneqq \mu C_1$ and multiplying
both sides of \eqref{tmp_quasi_optimality} by $\mu$.
\end{proof}

\begin{theorem}[abstract quasi-optimality]
\label{theorem_quasi_optimality}
Assume \WPs and \AP. 
Let $s \in [0,\smax]$, $f \in H^s(\partition)$, and $g \in H^{s+1/2}(\Gamma)$.
Recall the constant $\mu$ from Lemma~\ref{lemma_quasi_optimality_explicit},
and assume that
\begin{equation}
\label{eq_abstract_resolution_condition}
k\etaF{1} + \CAD\etaA{\anashift{\gAD}} \leq \frac{\mu}{4}, 
\qquad
\CcontM \CAM \etaA{\anashift{\gAM}}
\leq
1.
\end{equation}
Then, there exists $C > 0$ independent of $k$, $f$, $g$, and there exists 
a unique $u_h \in V_h$ such that (\ref{eq:Galerkin-approximation}) holds, 
%\begin{equation}
%\label{eq_def_uh}
%\bkm(u_h,v_h) = (f,v_h) + \langle g,v_h \rangle \quad \forall v_h \in V_h
%\end{equation}
and for $u \coloneqq \Skm(f,g)$, we have
\begin{equation}
\label{eq_quasi_optimality}
\|u-u_h\|_{1,t,k} \leq C \inf_{v_h \in V_h} \|u-v_h\|_{1,t,k}. 
\end{equation}
\end{theorem}

\begin{proof}
Fix $s \in [0,\smax]$, $f \in H^s(\partition)$, and $g \in H^{s+1/2}(\Gamma)$.
Let $u_h$ be any element of $V_h$ such that \eqref{eq:Galerkin-approximation} holds true. We
easily see that $u \coloneqq \Skm(f,g)$ and $u_h$ satisfy the assumptions of 
Lemma~\ref{lemma_quasi_optimality_explicit}. Injecting \eqref{eq_abstract_resolution_condition}
in \eqref{eq_quasi_optimality_explicit}, we have
\begin{equation*}
\frac{\mu}{2} \|u-u_h\|_{1,t,k}
\leq
2\rho
\|u-\Pi_h u\|_{1,t,k}
=
2\rho
\inf_{v_h \in V_h} \|u-v_h\|_{1,t,k},
\end{equation*}
so that \eqref{eq_quasi_optimality} holds for all $u_h$ satisfying
\eqref{eq:Galerkin-approximation} with $C \coloneqq 4\rho/\mu$.

Next, we observe that having established \eqref{eq_quasi_optimality}
for all $u_h$ satisfying \eqref{eq:Galerkin-approximation} implies, by linearity, the uniqueness
of the solution to \eqref{eq_quasi_optimality}. Since \eqref{eq_quasi_optimality}
corresponds to finite-dimensional linear system, this proves the existence and
uniqueness of $u_h$.
%
%Finally, \eqref{eq_error_estimate} is a direct consequence of
%\eqref{eq_approximation_factors} and \eqref{eq_quasi_optimality}.
\end{proof}
%----------------------------------------------
\subsection{Application to $hp$-FEM}
\label{subsection:application_to_hp_fem}
%----------------------------------------------

We now focus on the case where the discretization subspace $V_h$
is an $hp$-FEM  space of piecewise (mapped) polynomials. We consider the case where
the mesh $\CT_h$ associated with the finite element space is 
aligned with the partition $\partition$ and excludes hanging nodes. Since we
are working with analytic interfaces, dedicated assumptions
are required \cite{melenk-sauter10}.

\begin{assumption}[Quasi-uniform regular fitted meshes]
\label{assumption:quasi_uniform_regular_meshes}
Let $\widehat{K}$ be the reference simplex in spatial dimension $d=2$, $3$.
The mesh $\CT_h$ is a partition of $\Omega$ into non-overlapping elements $K$ such that
$\cup_{K \in \CT_h} \overline{K} = \overline{\Omega}$. For
each $K \in \CT_h$, there exists a bijective mapping
$F_K \colon \widehat K \to K$ that can be written as $F_K = R_K \circ A_K$,
where $A_K$ is affine and $R_K$ is an analytic map. 
Specifically, there exist constants $C_\mathrm{affine}, C_\mathrm{metric}, \gamma > 0$
independent of $K$ such that
\begin{equation*}
\begin{alignedat}{2}
& \| A^\prime_K \|_{L^\infty( \widehat{K} )}        \leq C_\mathrm{affine} h_K, \qquad &&\| (A^\prime_K)^{-1} \|_{L^\infty( \widehat{K} )} \leq C_\mathrm{affine} h^{-1}_K, \\
&\| (R^\prime_K)^{-1} \|_{L^\infty( \tilde{K} )} \leq C_\mathrm{metric},     \qquad &&\| \nabla^n R_K \|_{L^\infty( \tilde{K} )} \leq C_\mathrm{metric} \gamma^n n! \qquad \forall n \in \mathbb{N}_0.
\end{alignedat}
\end{equation*}
Here, $\tilde{K} = A_K(\widehat{K})$ and $h_K > 0$ denotes the element diameter.
We assume (``no hanging nodes'') that the element maps of elements sharing an edge or a face induce the same parametrization on that edge or face. 
(See \cite[Sec.~{8.1}]{melenk-sauter22} for a precise formulation for the case $d = 3$ and \cite[Def.~{2.4.1}]{melenk02} for the case $d = 2$.)
Furthermore, we assume that the mesh $\CT_h$ resolves the interfaces of $\partition$,
i.e., each element $K$ lies entirely in one subdomain $P \in \partition$.
\eremk
\end{assumption}

For $p \geq 1$, we introduce the space $S_p(\mathcal{T}_h)$ as the space of piecewise mapped polynomials of degree $p$: 
\begin{equation*}
S_{p}(\mathcal{T}_h)
\coloneqq
\left \{
u \in H^1(\Omega)
\colon
\left.\kern-\nulldelimiterspace{u}\vphantom{\big|} \right|_{K} \circ F_K
\in
\mathcal{P}_{p}(\widehat{K})
\text{ for all } K \in \CT_h
\right\}.
\end{equation*}
We note $S^p(\mathcal{T}_h) \subset \HOneT$. 
%It is important to note the conformity of the space $S_{p}(\CT_h)$ for all $t \leq 1$.
%%% with the limiting case $t=1$ included to cover second order ABCs. % removed by TCF
%
We will require approximation operators that approximate functions in 
the $\|\cdot\|_{1,t,k}$-norm: 
\begin{proposition}[Interpolation errors]
\label{prop:interpolation}
% For $1 \leq m \leq p$ there exists a constant $C_{{\rm i},m}$ such that
% for all $v \in H^{m+1}(\partition)$, we have
Let Assumption~\ref{assumption:quasi_uniform_regular_meshes}  be valid. 
There is an operator $I_h:H^2(\partition)\cap H^1(\Omega) \rightarrow S_p(\mathcal{T}_h)$ with the following
properties: 
\begin{enumerate}[nosep, start=1, label=(\roman*),leftmargin=*]
\item 
\label{item:prop:interpolation-i}
$I_h$ is defined elementwise, i.e., $(I_h v)|_K$ depends solely on $v|_{\overline{K}}$. 
\item
\label{item:prop:interpolation-ii}
For $1 \leq m \leq p$ and for all $v \in H^{m+1}(\partition)$, we have
\begin{equation}
\label{eq_interpolation_F}
\|v-I_h v\|_{1,t,k}
\leq
C_m %\CIF{m}
\left ( 1 + \left (\frac{kh}{p}\right )^{1/2-t} + \frac{kh}{p}\right )
\left (\frac{h}{p}\right )^m
\|v\|_{m+1,\partition}.
\end{equation}
\item
\label{item:prop:interpolation-iii}
Fix $\widetilde{C} > 0$ and assume $k h/p \leq \widetilde{C}$. 
For all $\gamma$, there exist 
% constants $\CIA{\gamma}$ and $\SIA{\gamma}$
constants $\SIA{\gamma}$, $C_\gamma$ solely depending on $\gamma$, $C_{\rm metric}$ of 
Assumption~\ref{assumption:quasi_uniform_regular_meshes}, and $\widetilde{C}$ 
such that for all $v \in \AnaClass{M}{\gamma}{\partition}$,
\begin{equation}
\label{eq_interpolation_A_exp}
\|v-I_h v\|_{1,t,k}
\leq
 C_\gamma
% \CIA{\gamma} M k
\left( 1 + \left (\frac{kh}{p}\right )^{1/2-t} + \frac{kh}{p} \right )
\left (
\left (\frac{h}{h+\SIA{\gamma}}\right )^p
+
k \left (\frac{kh}{\SIA{\gamma}p}\right )^p
\right )
 M.
\end{equation}
\item 
\label{item:prop:interpolation-iv}
Finally, 
for $v \in \AnaClass{M}{\gamma}{\partition}$ and 
$0 \leq q \leq p$, we have
\begin{equation}
\label{eq_interpolation_A}
\|v-I_h v\|_{1,t,k}
\leq
 C_{q,\gamma}
% \CIA{\gamma} M k
\left( 1 + \left (\frac{kh}{p}\right )^{1/2-t} + \frac{kh}{p} \right )
\left (
\left (\frac{h}{p}\right )^q
+
k \left (\frac{kh}{\SIA{\gamma}p}\right )^p
\right ) M.
\end{equation}
\end{enumerate}
\end{proposition}

\begin{proof}
The construction is essentially that of \cite[Thm.~{5.5}]{melenk-sauter10}, 
where the approximation in the norm $\|\nabla \cdot\|_{0,\partition} + k \|\cdot\|_{0,\Omega}$ is considered
instead of $\|\cdot\|_{1,t,k}$. We will therefore be brief. 
The operator $I_h$ is defined on the reference simplex and taken to be the 
one of \cite[Lemma~{8.3}]{melenk-sauter22}, which in turn is actually
the operator $\widehat\Pi^{\operatorname{grad},3d}_{p}$ (for $d = 3$) or the operator
$\widehat\Pi^{\operatorname{grad},2d}_{p}$ (for $d = 2$) of \cite{melenk-rojik18}. 

\emph{Proof of \ref{item:prop:interpolation-ii}:} 
The $p$-dependence of the approximation properties are spelled out in \cite[Lemma~{8.3}]{melenk-sauter22}
and the $h$-dependence follows from scaling arguments.

\emph{Proof of \ref{item:prop:interpolation-iii}:} 
Denote by $\widehat I_h$ the operator on the reference simplex, and introduce for $K \in {\mathcal T}_h$ 
and $t \in \{0,1\}$ the norm 
$\|v\|^2_{1,t,k,K}:=  \|\nabla v\|^2_{L^2(K)} + k^2 \|v\|^2_{L^2(K)} + k^{1-2t} |v|^2_{H^t(\partial K)}$. 
By standard scaling arguments and the approximation properties of $\widehat I_h$ 
(see the discussion in the proof of \cite[Lemma~{8.3}]{melenk-sauter22})  we get 
\begin{equation*}
\|v - I_h v\|_{1,t,k,K} \lesssim h^{d/2-1} p^{-1} \left( 1 + \frac{kh}{p} + \left( \frac{kh}{p}\right)^{1/2-t}\right) |\widehat v|_{2,\widehat K}, 
\end{equation*}
where $\widehat u:= u \circ F_K$ is the pull-back of $u$ to $\widehat K$. Noting that $\widehat{I}_h$ is a projection onto ${\mathcal P}_p(\widehat K)$ (and thus
$I_h$ a projection onto $S_p({\mathcal T}_h)$) we arrive at 
\begin{equation*}
\|v - I_h v\|_{1,t,k,K} \lesssim h^{d/2-1} p^{-1} \left( 1 + \frac{kh}{p} + \left( \frac{kh}{p}\right)^{1/2-t}\right) \inf_{w \in {\mathcal P}_p(\widehat K)} |\widehat v - w|_{2,\widehat K}. 
\end{equation*}
For $v \in \AnaClass{M}{\gamma}{\partition}$, we can now proceed as in the proof of \cite[Thm.~{5.5}]{melenk-sauter10} using in particular the 
approximation result \cite[Lemma~{C.2}]{melenk-sauter10} for $\inf_{w \in {\mathcal P}_p(\widehat K)} |\widehat v - w|_{2,\widehat K}$. 
This leads to the desired result for $t \in \{0,1\}$. For $t \in (0,1)$ the boundary estimate $|v - I_h v|_{t,\Gamma}$ then follows by the interpolation inequality. 

\emph{Proof of \ref{item:prop:interpolation-iv}:} 
%Estimate \eqref{eq_interpolation_F} is classical and follows from standard
%scaling arguments. Estimate \eqref{eq_interpolation_A_exp} can be found in
%\cite{melenk-sauter10}. 
If $q \leq p$, the function
\begin{equation*}
h \mapsto \left (\frac{h}{h+\SIA{\gamma}}\right )^{p-q}
\end{equation*}
is non-decreasing. Hence, since $h \leq C$, we have
\begin{multline*}
\left (
\frac{h}{h+\SIA{\gamma}}
\right )^p
=
\left (
\frac{h}{h+\SIA{\gamma}}
\right )^q
\left (
\frac{h}{h+\SIA{\gamma}}
\right )^{p-q}
\\
\leq
\SIA{\gamma}^{-q} h^q
\left (
\frac{C}{C+\SIA{\gamma}}
\right )^{p-q}
=
\SIA{\gamma}^{-q} h^q
\rho_\gamma^{p-q}
=
\left (\rho_\gamma\SIA{\gamma}\right )^{-q}
h^q \rho_\gamma^p, 
\end{multline*}
where $\rho_\gamma < 1$. It is then clear that
$\rho_\gamma^p \leq C_{\gamma,q} p^{-q}$, and it follows that
\begin{equation*}
\left (
\frac{h}{h+\SIA{\gamma}}
\right )^p
\leq
C_{q,\gamma} \left (\frac{h}{p}\right )^q.
\end{equation*}
Then, estimate \eqref{eq_interpolation_A} follows from \eqref{eq_interpolation_A_exp}.
\end{proof}
Our key result concerning finite element discretizations is established under an additional
assumption of polynomial well-posedness that we detail now: 
\begin{assumption}[Polynomial well-posedness]
\label{assumption:everything_is_polynomial_well_posed}
There exist constants $\APOLA$, $\APOLS$, $\APOLC$, $C \ge 0$,  independent of $k$ such that
\begin{align}
\Can & \leq C k^{\APOLA}, \\
\CsolveM &\leq C k^{\APOLS}, \\ \label{assumption:polynomial_stability}
\CcontM + \CAM + k^{-2} \CAD & \leq C k^{\APOLC}.
\end{align}
\end{assumption}

\begin{lemma}[Approximation factors for $hp$ finite elements]
\label{lemma_approximation_factors_FEM}
Assume that $p \geq 1+\APOLA+2\APOLC+\APOLS$. Then, we have
\begin{multline}
\label{eq_estimate_etaF}
\etaF{m}
\leq
C_{m,\APOLS,\APOLC}
\left( 1 + \left (\frac{kh}{p}\right )^{1/2-t} + \frac{kh}{p} \right )
\\
\left (
\left (\frac{h}{p}\right )^m
+
\frac{h}{p} \left (\frac{kh}{p}\right )^{\APOLA+\APOLS+\APOLC}
+
k^{\APOLA+ \APOLS+\APOLC+1} \left (\frac{kh}{\SIA{{\anasplit{m}}}p}\right )^p
\right )
\end{multline}
for all $m \geq 1$, and
\begin{multline}
\label{eq_estimate_etaA}
\etaA{\gamma} \leq C_{\gamma,\APOLS}
k^{-1} \left( 1 + \left (\frac{kh}{p}\right )^{1/2-t} + \frac{kh}{p} \right )
\\
\left (
\left (\frac{h}{p}\right )^{2\APOLC+1} \left (\frac{kh}{p}\right )^{\APOLA+\APOLS}
+
k^{\APOLA+\APOLS+1} \left (\frac{kh}{\SIA{\anashift{\gamma}}p}\right )^p
\right )
\end{multline}
for all $\gamma > 0$.
\end{lemma}

\begin{proof}
Let $1 \leq m \leq p$, and fix $(f,g) \in H^{m-1}(\partition) \times H^{m-1/2}(\Gamma)$
with
\begin{equation*}
\|f\|_{m-1,\partition} + \|g\|_{m-1/2,\Gamma} = 1.
\end{equation*}
For $u \coloneqq \Skm(f,g)$, Theorem~\ref{theorem_splitting}
provides the splitting $u = u_{\rm F} + u_{\rm A}$, where $u_{\rm F} \in H^{m+1}(\Omega)$
and $u_{\rm A} \in \AnaClass{M}{{\anasplit{m}}}{\partition}$ with
\begin{align}
\nonumber 
\|u_{\rm F}\|_{m+1,\partition} & \leq C_{m-1} , \\
\label{eq:lemma_approximation_factors_FEM-100}
M
& \leq
C  \Can \CsolveM (1+k^{-2} \CAD)
% C \frac{\CsolveM (1+\CAD)}{k}
\leq
 C k^{\APOLA+ \APOLS+\APOLC}.
% C \CPOL k^{2\APOL-1}
\end{align}
It is convenient to abbreviate 
\begin{equation}
\label{eq:Ftkhp}
F_{t,k,h,p}:= 
\left (
1 + \left (\frac{kh}{p} \right )^{1/2-t} + \frac{kh}{p}
\right ). 
\end{equation}
Applying \eqref{eq_interpolation_F}, we have
\begin{align*}
\|u_{\rm F} - I_h u_{\rm F}\|_{1,t,k}
&\leq
C_m % \CIF{m}
%\left (
%1 + \left (\frac{kh}{p} \right )^{1/2-t} + \frac{kh}{p}
%\right )
F_{t,k,h,p} 
\left (\frac{h}{p}\right )^m
\|u_{\rm F}\|_{m+1,\partition}
%\\
%&\leq
\leq 
C^{\prime}_m
%\left (
%1 + \left (\frac{kh}{p} \right )^{1/2-t} + \frac{kh}{p}
%\right )
F_{t,k,h,p}
\left (\frac{h}{p}\right )^m.
\end{align*}
On the other hand, since $p \geq 1+\APOLA + \APOLS+2\APOLC$ by assumption, 
we can use \eqref{eq_interpolation_A} with $q=\APOLA+\APOLS+\APOLC+1$
%and $\gamma = \anashift{\anasplit{m}}$, showing that
and $\gamma = {\anasplit{m}}$, showing that
\begin{align*}
\|u_{\rm A} - I_h u_{\rm A}\|_{1,t,k}
&\leq
C_{q}
% \CIA{\anashift{\anasplit{m}}} Mk
%\left( 1 + \left (\frac{kh}{p}\right )^{1/2-t} + \frac{kh}{p} \right )
F_{t,k,h,p}
\left (
\left (\frac{h}{p}\right )^{\APOLA+ \APOLS+\APOLC+1}
+
k \left (\frac{kh}{\SIA{{\anasplit{m}}}p}\right )^p
\right )
M
\\
& 
\stackrel{ (\ref{eq:lemma_approximation_factors_FEM-100})}{\leq}
C_{q} % \CIA{\anashift{\anasplit{m}}}
k^{\APOLA+ \APOLS+\APOLC}
%\left( 1 + \left (\frac{kh}{p}\right )^{1/2-t} + \frac{kh}{p} \right )
F_{t,k,h,p}
\left (
\left (\frac{h}{p}\right )^{\APOLA+\APOLS+\APOLC+1} + k \left (\frac{kh}{\SIA{{\anasplit{m}}}p}\right )^p
\right )
\\
&\leq
C_{q}
%\left( 1 + \left (\frac{kh}{p}\right )^{1/2-t} + \frac{kh}{p} \right )
F_{t,k,h,p}
\left (
\frac{h}{p}
\left (\frac{kh}{p}\right )^{\APOLA + \APOLS+\APOLC}
+
k^{\APOLA+ \APOLS+\APOLC+1}
\left (\frac{kh}{\SIA{{\anasplit{m}}}p}\right )^p
\right ).
\end{align*}
Then, we have
\begin{align*}
&\|u-I_hu\|_{1,t,k}
\leq
\|u_{\rm F}-I_hu_{\rm F}\|_{1,t,k}
+
\|u_{\rm A}-I_hu_{\rm A}\|_{1,t,k}
\\
&\leq
C_{m,q}
%\left( 1 + \left (\frac{kh}{p}\right )^{1/2-t} + \frac{kh}{p} \right )
F_{t,k,h,p}
\left (
\left (\frac{h}{p}\right )^m
+
\frac{h}{p} \left (\frac{kh}{p}\right )^{\APOLA+ \APOLS+\APOLC}
+
k^{\APOLA+\APOLS+\APOLC+1}
\left (\frac{kh}{\SIA{{\anasplit{m}}}p}\right )^p
\right ),
\end{align*}
and \eqref{eq_estimate_etaF} follows from the definition of $\etaF{m}$
in \eqref{eq_definition_etaF}.

Consider now $(f,g) \in \AnaClass{M_f}{\gamma}{\partition} \times \AnaClassbdy{M_g}{\gamma}{T}{\Gamma}$
with $M_f + k M_g = 1/2$, and set $u \coloneqq \Skm(f,g)$. We know from \ref{WP3} that
$u \in \AnaClass{M_u}{\anashift{\gamma}}{\partition}$ with
\begin{equation*}
M_u
\leq C \Can \CsolveM k^{-1} 
% \frac{\CshiftM\CsolveM}{k}
(M_f + k M_g)
=
\Can  C\CsolveM k^{-1} 
\leq
C  k^{\APOLA+\APOLS-1}.
\end{equation*}
It then follows from \eqref{eq_interpolation_A} with $q = \APOLA+\APOLS  +2 \APOLC + 1$
and $\gamma = \anashift{\gamma}$ that
\begin{equation*}
\|u-I_h u\|_{1,t,k}
\leq
C_{\gamma,\APOL} % \CIA{\anashift{\gamma}}
k^{\APOLA+\APOLS-1}
%\left( 1 + \left (\frac{kh}{p}\right )^{1/2-t} + \frac{kh}{p} \right )
F_{t,k,h,p}
\left (
\left (\frac{h}{p}\right )^{\APOLA+\APOLS+2\APOLC+1}
+
k \left (\frac{kh}{\SIA{\anashift{\gamma}}p}\right )^p
\right ),
\end{equation*}
and \eqref{eq_estimate_etaA} follows from the definition of $\etaA{\gamma}$
in \eqref{eq_definition_etaA}.
\end{proof}

Our main result is then a direct consequence of 
Theorem~\ref{theorem_quasi_optimality} and Lemma~\ref{lemma_approximation_factors_FEM}.

\begin{theorem}[discrete stability of $hp$-FEM]
Let \WPs and \APs  as well as Assumptions~\ref{assumption:quasi_uniform_regular_meshes} and
\ref{assumption:everything_is_polynomial_well_posed} hold true. Then, for all
$c_2 > 0$ there exist constants $c_1$, $C$ depending solely on $c_2$ and 
the constants appearing in \WP, \AP, Assumptions~\ref{assumption:everything_is_polynomial_well_posed}, \ref{assumption:quasi_uniform_regular_meshes}
such that under the scale resolution condition 
\begin{equation}
\label{eq:scale-resolution}
\frac{kh}{p} \leq c_1 \qquad \text{and} \qquad p \geq 1 + 2\APOLC +\APOLA +\APOLS + c_2 \log k
\end{equation}
the Galerkin solution $u_{hp} \in S_{p}(\mathcal{T}_h)$ of \eqref{eq:Galerkin-approximation}
exists, is unique, and satisfies
\begin{equation}
\label{eq:quasi-optimality}
\|u - u_{hp}\|_{1,t,k}
\leq
C \inf_{v_{hp} \in S_{p}(\mathcal{T}_h)} \|u - v_{hp}\|_{1,t,k}.
\end{equation}
\end{theorem}

\begin{proof}
We may assume that $kh/p \leq c_1 < 1$ so that easy calculations show 
\begin{align}
\label{eq:Ftkhp-10} 
\sup_{t \in [0,1]} F_{t,k,h,p} & = F_{1,k,h,p},  \\
\label{eq:Ftkhp-20} 
F_{1,k,h,p} \left(\frac{kh}{p}\right)^{1/2}  & \leq 1 + \left(\frac{kh}{p}\right)^{1/2} + \left(\frac{kh}{p}\right)^{3/2} 
\leq 1 + c_1^{1/2} + c_1^{3/2} \leq 3. 
\end{align}
Let $c_2 > 0$ and fix $p \geq 1 + 2\APOLC +\APOLA+\APOLS + c_2 \log k$. Then,
we have in particular that $\log k \leq (p-1/2)/c_2$, so that for any $\sigma$, $b > 0$ 
\begin{equation*}
k^{b} \leq \left (e^{b/c_2}\right )^{p-1/2},
\end{equation*}
and
\begin{equation}
\label{tmp_kb_khp}
k^b \left (\frac{kh}{\sigma p}\right )^{p-1/2}
\leq
\left (e^{b/c_2} \right )^{p-1/2} \left (\frac{c_1}{\sigma}\right )^{p-1/2}
=
\left (\frac{e^{b/c_2}}{\sigma} c_1 \right )^{p-1/2}.
\end{equation}
Since $p \geq 1 + 2\APOLC + \APOLA + \APOLS$ by assumption, we can use
\eqref{eq_estimate_etaF} with $m=1$. Since $t \leq 1$, picking $\sigma = \SIA{{\anasplit{1}}}$
and $b=\APOLA+\APOLS+\APOLC+2$ in \eqref{tmp_kb_khp}, we get 
\begin{align*}
k\etaF{1}
&\stackrel{ (\ref{eq_estimate_etaF}), (\ref{eq:Ftkhp-10} )}{\leq} 
C_{}
%\left (
%1 + \left (\frac{kh}{p}\right )^{-1/2} + \frac{kh}{p}
%\right )
F_{1,k,h,p}
\left (
\frac{kh}{p}
+
\left (\frac{kh}{p}\right )^{\APOLA+\APOLS+\APOLC+1}
+
k^{\APOLA+\APOLS+\APOLC+2} \left (\frac{kh}{\SIA{{\anasplit{1}}}p}\right )^{p}
\right )
\\
&\stackrel{(\ref{eq:Ftkhp-20})}{\leq} 
3 C_{}
%\left (
%1 + \left (\frac{kh}{p}\right )^{1/2} + \left (\frac{kh}{p}\right )^{3/2}
%\right )
\left (
\left (\frac{kh}{p}\right )^{1/2}
\!\!\!
+
\left (\frac{kh}{p}\right )^{\APOLA+\APOLS+\APOLC+1/2}
+
k^{\APOLA+\APOLS+\APOLC+2}
\SIA{{\anasplit{1}}}^{-1/2}
\left (\frac{kh}{\SIA{{\anasplit{1}}}p}\right )^{p-1/2}
\right )
\\
&\leq
3 C_{}
%\left (
%1 + c_1^{1/2} + c_1^{3/2}
%\right )
\left (
c_1^{1/2}
+
c_1^{\APOLA+\APOLS+\APOLC+1/2}
+
\SIA{{\anasplit{1}}}^{-1/2}
\left (\frac{e^{(\APOLA+\APOLS+\APOLC+2)/c_2}}{\SIA{{\anasplit{1}}}} c_1 \right )^{p-1/2}
\right ).
\end{align*}
We then employ \eqref{eq_estimate_etaA} with $\gamma = \anashift{\gAM}$, leading to
\begin{align*}
&\CcontM \CAM \etaA{\anashift{\gAM}}
\\
&\leq
C_{\APOLS,\APOLC} k^{2\APOLC-1}
%\left( 1 + \left (\frac{kh}{p}\right )^{1/2-t} + \frac{kh}{p} \right )
F_{t,k,h,p}
\left (
\left (\frac{h}{p}\right )^{2\APOLC+1} \left (\frac{kh}{p}\right )^{\APOLA+\APOLS}
+
k^{\APOLA+\APOLS+1} \left (\frac{kh}{\SIA{\anashift{\gamma}}p}\right )^p
\right )
\\
&\stackrel{(\ref{eq:Ftkhp-10}), (\ref{eq:Ftkhp-20})}{\leq}
3 C_{} 
%\left( 1 + \left (\frac{kh}{p}\right )^{1/2} + \left (\frac{kh}{p}\right )^{3/2} \right )
\left (
k^{-1} 
\left (\frac{h}{p}\right )^{1}\!\! \left (\frac{kh}{p}\right )^{2\APOLC+\APOLA+\APOLS-1/2}
+
\SIA{\anashift{\gamma}}^{-1/2}
k^{2\APOLC+\APOLA+ \APOLS} \left (\frac{kh}{\SIA{\anashift{\gamma}}p}\right )^{p-1/2}
\right )
\\
%&\leq
%\JMM{3} C_{} 
%%\left( 1 + \left (\frac{kh}{p}\right )^{1/2} + \left (\frac{kh}{p}\right )^{3/2} \right )
%\left (
%\left (\frac{kh}{p}\right )^{2\APOLC+\APOLA+\APOLS+1/2}
%+
%\SIA{\anashift{\gamma}}^{-1/2}
%k^{2\APOLC+\APOLA+\APOLS} \left (\frac{kh}{\SIA{\anashift{\gAM}}p}\right )^{p-1/2}
%\right )
%\\
&\stackrel{(\ref{tmp_kb_khp})}{\leq}
3 C_{}
%\left( 1 + c_1^{1/2} + c_1^{3/2} \right )
\left (
c_1^{2\APOLC+\APOLA+\APOLS+1/2}
+
\SIA{\anashift{\gamma}}^{-1/2}
\left (
\frac{e^{(2\APOLC+\APOLA+\APOLS)/c_2}}{\SIA{\anashift{\gAM}}} c_1
\right )^{p-1/2}
\right ),
\end{align*}
where we have used \eqref{tmp_kb_khp} with $\sigma = \SIA{\anashift{\gAM}}$ and $b = 2\APOLC+\APOLA+\APOLS$. 
Finally, since $k \geq k_0$, a similar reasoning shows that
\begin{multline*}
\CAD \etaA{\anashift{\gAD}}
\leq
Ck^{\APOLC} \etaA{\anashift{\gAD}}
\leq
Ck^{2\APOLC} \etaA{\anashift{\gAD}}
\\
\leq
3 C_{\APOL}
%\left( 1 + c_1^{1/2} + c_1^{3/2} \right )
\left (
c_1^{2\APOLC+\APOLA+\APOLS+1/2}
+
\SIA{\anashift{\gAD}}^{-1/2}
\left (
\frac{e^{(2\APOLC+\APOLA+\APOLS)/c_2}}{\SIA{\anashift{\gAD}}} c_1
\right )^{p-1/2}
\right ).
\end{multline*}
At this point, we assume in addition to $c_1 \leq 1$  that 
\begin{equation}
\label{tmp_requirements_c1}
\frac{e^{(\APOLA+\APOLS+\APOLC+2)/c_2}}{\SIA{{\anasplit{1}}}} c_1 \leq 1, 
\quad
\frac{e^{(2\APOLC+\APOLA+\APOLS)/c_2}}{\SIA{\anashift{\gAM}}} c_1 \leq 1, 
\quad
\frac{e^{(2\APOLC+\APOLA+\APOLS)/c_2}}{\SIA{\anashift{\gAD}}} c_1 \leq 1,
\end{equation}
which leads to the simplified expressions (using $p \ge 1$)
\begin{equation*}
k\etaF{1} + \CAD \etaA{\anashift{\gAD}}
\leq
C_{\APOLS,\APOLC}
\!\!
\left [
1 \!
+
\SIA{\anasplit{1}}^{-1/2}
\!
\sqrt{\frac{e^{(\APOLA+\APOLC+\APOLS+2)/c_2}}{\SIA{{\anasplit{1}}}}}
+
\SIA{\anashift{\gAD}}^{-1/2}
\sqrt{\frac{e^{(2\APOLC+\APOLA+\APOLS)/c_2}}{\SIA{\anashift{\gAD}}}}
\right ]
c_1^{1/2}
\end{equation*}
and
\begin{equation*}
\CcontM \CAM \etaA{\anashift{\gAM}}
\leq
C_{\APOLC,\APOLS}
\left (
1
+
\SIA{\anashift{\gAM}}^{-1/2}
\sqrt{\frac{e^{(2\APOLC+\APOLA+\APOLS)/c_2}}{\SIA{\anashift{\gAM}}}}
\right )
c_1^{1/2}.
\end{equation*}
Then, we see that requiring, in addition to \eqref{tmp_requirements_c1}, that
\begin{align*}
& C_{\APOLC,\APOLS}
\left (
1
+
\frac{\sqrt{e^{(\APOLA+\APOLC+\APOLS+2)/c_2}}}{\SIA{{\anasplit{1}}}}
+
\frac{\sqrt{e^{(2\APOLC+\APOLA+\APOLS)/c_2}}}{\SIA{\anashift{\gAM}}}
+
\frac{\sqrt{e^{(2\APOLC+\APOLA+\APOLS)/c_2}}}{\SIA{\anashift{\gAD}}}
% \sqrt{\frac{e^{(\APOLC+\APOLS)/c_2}}{\SIA{{\anasplit{1}}}}}
% +
% \sqrt{\frac{e^{(2\APOLC+\APOLS)/c_2}}{\SIA{\anashift{\gAM}}}}
% +
% \sqrt{\frac{e^{(2\APOLC+\APOLS)/c_2}}{\SIA{\anashift{\gAD}}}}
\right )
c_1^{1/2} \\
& 
\stackrel{!}{\leq}
\min(\mu/4,1),
\end{align*}
then the requirements of Theorem~\ref{theorem_quasi_optimality} are met, which leads to the result.
\end{proof}

\section{Application to wave progagation problems}
\label{section:covered_problems}

In this section, we verify that the Assumptions~\WP, \AP\,
of Section~\ref{section:assumptions_and_problem_specific_notation} are in fact
satisfied for a variety of time-harmonic wave propagation problems.
We emphasize that Assumption~\ref{WP2} and, in particular,
the polynomial stability assumption expressed in \eqref{assumption:polynomial_stability}
is assumed and has to be proved separately. We discuss cases where such a stability
property holds true in Remark~\ref{remark_polynomial_stability} below.

Throughout, we will consider the heterogeneous Helmholtz equation, and in all examples 
the sesquilinear form $a(\cdot,\cdot)$ will be of the form 
\begin{equation}
\label{eq:a}
a(u,v) = (A(x) \nabla u, \nabla v). 
\end{equation}
The strong form of problem \eqref{eq:S_k_minus_problem} is 
\begin{equation}
\label{eq:Lkm}
\Lkm u:= -k^2 n^2 u -\nabla \cdot (A(x) \nabla u) = f \quad \mbox{ in $\Omega$}, 
\qquad  
\partial_{n_A} u - \TkGammaM u = g \quad \mbox{ on $\Gamma$,}
\end{equation}
with the co-normal derivative $\partial_{n_A} v:= \normal \cdot (A \nabla v)|_{\Gamma}$ and the outer
normal vector $\normal$.  Hence, in all examples the volume operator is given by $\TkOmegaM = k^2 n^2$. 
The complex-valued coefficients $A \in L^\infty(\Omega,\operatorname{GL}({\mathbb C}^d))$ 
and $n \in L^\infty(\Omega,{\mathbb C})$ are assumed to satisfy the following:
\begin{subequations}
\label{eq:HH-robin-assumption-coefficients}
$A$ is uniformly elliptic, i.e., 
\begin{equation}
\label{eq:HH-robin-assumption-coefficients-a}
\operatorname{Re} (\xi^H A(x) \xi) \ge a_{\min} |\xi|^2 \qquad \forall \xi \in {\mathbb C}^d \quad \forall x \in \Omega.
\end{equation}
Furthermore, $A$ and $n$ are piecewice analytic, specifically, 
\begin{align}
\label{eq:A-analytic}
  &\| \nabla^p  A  \|_{L^\infty( \Omega \setminus \interface )} \leq C_A \gamma_A^p p! \qquad \forall p \in \mathbb{N}_0, \\
\label{eq:n-analytic}
  &\| \nabla^p n \|_{L^\infty( \Omega \setminus \interface )}
  \leq
  C_n \gamma_n^p p! \qquad \forall p \in \mathbb{N}_0,
\end{align}
for fixed constants $C_A,\gamma_A$ and $C_n,\gamma_n$.
\end{subequations}

\begin{remark}
The coefficients $A$ and $n$ are allowed to depend on $k$. We merely require that the constants in 
\eqref{eq:HH-robin-assumption-coefficients-a}--\eqref{eq:n-analytic} not depend on $k$. 
Therefore, equations including  volume damping, classically written in the form
$-k^2 n^2 u + i k m u - \nabla \cdot (A \nabla u) $ with some $m$ satisfying 
the condition that $\|\nabla^p m\|_{L^\infty(\Omega\setminus\interface)} \leq C_m \gamma_m^p p!$ for all $p \in \mathbb{N}_0$ 
are included in this setting. Also certain PML-variants with $k$-dependent matrices $A$ are included in this setting. 
\eremk
\end{remark}

The strong form of problem \eqref{eq_definition_Skp} will be 
\begin{equation}
\label{eq:Lkp}
\Lkp u:= -\nabla \cdot (A(x) \nabla u) + k^2 u = f \quad \mbox{ in $\Omega$}, 
\qquad  
\partial_{n_A} u - \TkGammaP u = g \quad \mbox{ on $\Gamma$}
\end{equation}
so that $\TkOmegaP = k^2$. The problems \eqref{eq:Lkm}, \eqref{eq:Lkp}, which we discuss in more detail in 
Sections~\ref{problem:HH_with_robin_bc}--\ref{problem:HH_with_second_order_abcs}, thus differ in the boundary conditions, i.e., 
the operators $\TkGammaM$ and $\TkGammaP$. 
%We present a synopsis of the covered Helmholtz problems in Table~\ref{table:overview_table}.

Throughout this Section~\ref{section:covered_problems}, we assume:  
\begin{assumption}
\label{assumption:smoothness_domain_boundary_and_interface} 
The boundary $\Gamma$ and the interface $\interface$ are analytic. 
\eremk
\end{assumption}

\begin{remark}
\label{remark_polynomial_stability}
The polynomial stability assumption is known to hold true in a variety of situations.
The growth of the stability constant is governed by the trajectories of the rays
(solutions to the bi-characteristic equation) associated with the wave propagation
problem. This is intuitive to understand: if some rays are trapped, more time will be
needed for the energy to leak out of the domain, resulting in a larger stability constant.
On the mathematical level, the connection between waves and rays can be made rigorous through
semi-classical analysis \cite{lax_phillips_1964a,melrose_1975a,vainberg_1975a}, and three
scenarios must be distinguished.

First of all, we say that the problem is non-trapping if all rays escape the domain
in finite time, i.e., there are no trapped rays \cite{galkowski_spence_wunsch_2020a}.
In this case, it is well-known that $\CsolveM$ is bounded uniformly in $k$. Early works
on the subject have focused on the scattering problem from a Dirichlet obstacle
(see e.g. \cite{lax_phillips_1964a,melrose_1975a}), but the estimate also applies
to problems with heterogeneous media under radial monotonicity conditions on the coefficients
(see,  e.g., \cite{moiola_spence_2019a}).

The next situation to consider is the case of ``weak'' trapping
\cite{chandlerwilde-spence-gibbs-smyshlyaev20a}. In this case,
rays may be trapped, but all the trajectories of trapped rays are unstable.
This is for instance the case for the scattering problem from a Dirichlet obstacle
consisting of two balls or two squares. In this case, although rays may be trapped
between the two parts of the obstacle, polynomial well-posedness holds true. For instance,
it is known that $\CsolveM \lesssim \log k$ for the case of two balls \cite{ikawa_1988a},
and that $\CsolveM \lesssim k^2$ for the case of two squares
\cite{chandlerwilde-spence-gibbs-smyshlyaev20a}. Similar estimates holds true
when the weak trapping is generated by space-varying coefficients
\cite{chaumontfrelet_spence_2023a}.

Finally, the case of ``strong'' trapping arises when trapped rays have stable
trajectories. This is for instance the case when scattering from a ``$C$ shaped''
obstacle. In can be shown that polynomial stability does not hold true for all
frequencies \cite{moiola_spence_2019a}. However, the set of frequencies causing
the super-algebraic growth is, in a sense, small. Specifically, it is shown in
\cite{lafontaine-spence-wunsch19} that polynomial stability holds true even in
this case with $\CsolveM \lesssim k^{(5d+1)/2}$, up to excluding a set of frequencies
of arbitrarily small measure.

The above stability bounds are stated for a Helmholtz problem with a radiation condition
at infinity or, equivalently, an exact DtN operator. However, these bounds
still hold true if the DtN operator is replaced by a PML, see
\cite{galkowski_lafontaine_spence_2023a}. Bounds for impedance boundary conditions
are also given, e.g., in \cite{barucq-chaumont-frelet-gout17,melenk95}.
Higher-order absorbing boundary conditions are also considered in \cite{galkowski2021local}.
\eremk
\end{remark}

%--------------------
%
%\begin{table}[t]
%	\input{00_overview_table.tex}
%	\label{table:overview_table}
%\end{table}
%
%
%We will see that for the model problems of Sections~\ref{problem:HH_with_robin_bc}, 
%\ref{problem:HH_full_space}, 
%\ref{problem:HH_with_second_order_abcs}, 
%\ref{problem:PML}
%the operator $\TkGammaP$
%actually satisfies the stronger assumption (\ref{eq:condition-on-Tkp}), which implies the coercivity of $\bkp$:  
%%
%\begin{lemma}
%\label{lemma:b_k_plus_coercive_via_T_k_Gamma_plus}
%Let $A \in L^\infty(\Omega, \operatorname{GL}({\mathbb C}^d)$ satisfy 
%(\ref{eq:HH-robin-assumption-coefficients-a}). 
%Let $\TkGammaP$ satisfy 
%\begin{subequations}
%\label{eq:condition-on-Tkp}
%\begin{equation}
%\label{eq:positivity-Tkp}
%    -\Re ( \sigma \langle \TkGammaP u , u \rangle )  \geq 0 \qquad \forall u \in \HOneT
%\end{equation}
%and for $t  \in (1/2,1]$ the stronger condition 
%  \begin{equation}
%\label{eq:coercivity-Tkp}
%    -\Re ( \sigma \langle \TkGammaP u , u \rangle )
%    \gtrsim
%    k^{-2t + 1} |u|_{t, \Gamma}^2,
%    \qquad \forall u \in \HOneT.
%  \end{equation}
%\end{subequations}
%
%  Then (\ref{eq_coercivity_bkp}) (coercivity of $\bkp$)
%  in \ref{AP1} holds for the choice $\Lkp u = - \nabla \cdot (A \nabla u) + k^2 u$, 
%i.e., $a(u,v) = (A(x)\nabla u, \nabla v)$ and $\TkOmegaP u = -k^2 u$. 
%If furthermore $\| \TkGammaP u\|_{t,\Gamma} \leq C \|u\|_{1,\tcf{t,k}}$, then $\bkp$ is also bounded
%(uniformly in $k$). 
%\end{lemma}
%\begin{proof}
%  Trivial.
%\end{proof}
%-----------------------------
\subsection{Heterogeneous Helmholtz problem with Robin boundary conditions and PML}
\label{problem:HH_with_robin_bc}
%-----------------------------
The heterogeneous Helmholtz problem with Robin boundary conditions corresponds to the choice 
\begin{equation*}
\TkGammaM u:= ik G u
\end{equation*}
in \eqref{eq:Lkm} for some function $G$ that is analytic in a fixed tubular neighborhood
$T$ of $\Gamma$. In \eqref{eq:Lkp}, we select $\TkGammaP:= 0$. 
Note that we allow for complex-valued coefficients $A$ and $n$, which allows us to cover forms of 
PML where the path deformation into the complex plane is based on the use of polar/spherical coordinate
\cite{collino-monk98,galkowski2022helmholtz}. Specifically, it can be seen following the
discussion in \cite[Remark 2.1]{chaumontfrelet-vega21} and the expression of $A$ given
in \cite[Eq (1.11)]{galkowski2022helmholtz} that \ref{WP3} holds true if the damping parameter
is small enough (the width of the PML, however, is not constrained).

\begin{lemma}[\APs for Robin b.c.]
\label{lemma:HH-robin-AP}
Let Assumption~\ref{assumption:smoothness_domain_boundary_and_interface}  and 
\eqref{eq:HH-robin-assumption-coefficients} be valid. 
Let $G$ be analytic in a tubular neighborhood $T$ of $\Gamma$. 
For $\Lkm$, $\Lkp$ given by \eqref{eq:Lkm} and \eqref{eq:Lkp} 
with $\TkGammaM = ikG $ and $\TkGammaP = 0$ the following holds 
with $t= 1/2$: 
\begin{enumerate}[nosep, start=1, label={(\roman*)},leftmargin=*]
\item 
\label{item:lemma:HH-robin-AP-i}
$\bkp$ is bounded uniformly in $k$, and \ref{AP1} holds with $\sigma =  1$. 
\item 
\label{item:lemma:HH-robin-AP-ii}
$\|\Skp(f,g)\|_{1,t,k} \leq C \left[ k^{-1} \|f\|_{0,\Omega} + k^{-1/2} \|g\|_{0,\Gamma}\right]$, and 
\ref{AP2} holds for any fixed $\smax \ge 0$. 
\item 
\label{item:lemma:HH-robin-AP-iii}
The splittings of $\TkOmegaM - \TkOmegaP$ and $\TkGammaM - \TkGammaP$ of \ref{AP3} are given by 
$\ROD u = k^2 (n^2 +1) u$, $\AOD = 0$ and $\RGD u= i k G u$, $\AGD = 0$ and hence satisfy \ref{AP3}
for any fixed $\smax \ge 0$.  
\end{enumerate}
\end{lemma}
\begin{proof}
\emph{Proof of \ref{item:lemma:HH-robin-AP-i}, \ref{item:lemma:HH-robin-AP-iii}:} 
Continuity of $\bkp$ follows by inspection. The coercivity \ref{AP1} of $\bkp$ with $\sigma = 1$ 
follows from \eqref{eq:HH-robin-assumption-coefficients-a} and $\TkGammaP = 0$. 

\emph{Proof of \ref{item:lemma:HH-robin-AP-ii}:} 
The coercivity of $\bkp$ implies for $u:= \Skp(f,g)$ 
\begin{equation*}
\|u\|^2_{1,t,k} \lesssim \left| (f,u)\right| + \left|(g,u)_\Gamma\right| 
\lesssim k^{-1} \|f\|_{0,\Omega} \|u\|_{0,\Omega} + k^{-1/2} \|g\|_{0,\Gamma} k^{1/2} \|u\|_{0,\Gamma}. 
\end{equation*}
The trace inequality $k \|v\|_{0,\Gamma} \lesssim \|v\|_{1,\Omega} + k \|v\|_{0,\Omega} \lesssim \|v\|_{1,t,k}$ then 
implies the {\sl a priori bound}.
The validity of \ref{AP2} follows from standard elliptic regularity theory for Neumann problems since $\Gamma$ and $\interface$ are smooth. 
The fact that $A$ is ${\mathbb C}^{d\times d}$-valued instead of pointwise SPD is discussed in more detail in the proof of Lemma~\ref{lemma:HH-robin-WP} below.  

\emph{Proof of \ref{item:lemma:HH-robin-AP-iii}:} follows by inspection. 
\end{proof}

\begin{lemma}[\WPs for Robin b.c.]
\label{lemma:HH-robin-WP}
Let Assumption~\ref{assumption:smoothness_domain_boundary_and_interface} 
and \eqref{eq:HH-robin-assumption-coefficients} be valid. 
Let $G$ be analytic in a tubular neighborhood $T$ of $\Gamma$. 
Assume \ref{WP2}. Then:  
\begin{enumerate}[nosep, start=1, label={(\roman*)},leftmargin=*]
\item 
\label{item:lemma:HH-robin-WP-i}
\ref{WP1} holds with $\CcontM = O(1)$, and \ref{WP3} holds with $\Can = O(1)$ (both uniformly in $k$).
\item 
\label{item:lemma:HH-robin-WP-ii}
\ref{WP4} holds with $\ROM u = k^2 n^2 u$, $\AOM  = 0$ and $\RGM = i k G u$, $\AGM = 0$. 
In particular, for the constants in Assumption~\ref{assumption:everything_is_polynomial_well_posed}, we have 
$\CcontM = O(1)$ (uniformly in $k$) and $C^-_{A,k} = 0$ and $C^\Delta_{A,k} = 0$. 
\item 
\label{item:lemma:HH-robin-WP-iii}
\ref{WP5} holds. 
\end{enumerate}
\end{lemma}
\begin{proof}
\emph{Proof of \ref{item:lemma:HH-robin-WP-i}:} 
$\CcontM = O(1)$ follows by inspection. To see \ref{WP3}, we first discuss the case that the matrix 
$A$ is pointwise SPD as this case is essentially covered by the literature. Specifically, we observe that 
the function $u = \Skm(\widetilde{f},\widetilde{g})$ with 
$\widetilde{f} \in \AnaClass{C_{\widetilde f}}{\gamma_{\widetilde{f}}}{\partition}$, 
$\widetilde{g} \in \AnaClassbdy{C_{\widetilde g}}{\gamma_{\widetilde{g}}}{T}{\Gamma}$, 
satisfies, upon setting $\varepsilon = 1/k$, 
  \begin{equation}\label{eq:rescaled_heterogeneous_helmholtz}
    \begin{alignedat}{2}
      - \varepsilon^2 \nabla \cdot (A \nabla u) - n^2 u &= \varepsilon^2 \widetilde{f} \quad &&\text{in } \Omega,\\
      \varepsilon^2 \partial_{n_A} u               &= \varepsilon( \varepsilon \widetilde{g} + iGu) \quad &&\text{on } \Gamma,
    \end{alignedat}
  \end{equation}
The regularity of solutions of these problems form has been addressed 
in \cite[Lemma~{4.13}]{melenk-sauter11} and \cite[Prop.~{5.4.5}, Rem.~{5.4.6}]{melenk02}. 
The procedure is as follows: the regularity can be analyzed locally and the boundary or the 
interface may be flattened by (local) changes of variables so that one is reduced 
to the analysis on balls (interior estimates and transmission problems) or half-balls (boundary estimates). 
Note that that Lemma~\ref{lemma:lemma2.6} (and \cite[Lemma~{4.3.4}]{melenk02}) provide that the 
analyticity classes are invariant under bi-analytic changes of variables. Hence, one may apply 
  \cite[Prop.~{5.5.1}]{melenk02} (interior analytic regularity),
  \cite[Prop.~{5.5.3}]{melenk02} (boundary analytic regularity for Neumann problems)
  as well as
  \cite[Prop.~{5.5.4}]{melenk02} (interface analytic regularity) if $\interface \neq \emptyset$,
	to problem \eqref{eq:rescaled_heterogeneous_helmholtz}. Considering as an example 
the case of boundary regularity, we apply \cite[Prop.~{5.5.3}]{melenk02} with the following choices 
  \begin{equation*}
    \begin{alignedat}{11}
      C_A &= O(C_A),           \ &&C_b      &&= 0, \ &&C_c      &&= O(C_{n^2}),
          \ &&C_f      &&= O(\varepsilon^2 C_{\widetilde{f}}), \ &&C_{G_1}      &&= O(\varepsilon C_{\widetilde{g}}), \ &&C_{G_2}      &&= O(1),\\
      \gamma_A &= O(\gamma_A), \ &&\gamma_b &&= 0, \ &&\gamma_c &&= O(\gamma_{n^2}),
          \ &&\gamma_f &&= O(\gamma_{\widetilde f}),          \ &&\gamma_{G_1} &&= O(\gamma_{\widetilde{g}}), \ &&\gamma_{G_2}      &&= O(1), 
    \end{alignedat}
  \end{equation*}
where $\gamma_{n^2}$ and $C_{n^2}$ are the analyticity constants of the coefficient $n^2$ and the $O(\cdot)$-notation 
absorbs the modification of constants due to the analytic change of variables when transforming to a half-ball. 
By combining the finitely many local contributions, we arrive at 
  \begin{align*}
		\| \nabla^p u \|_{L^2(\Omega \setminus \interface)}
    &\lesssim C \gamma^p \max\{k, p\}^p  ( k^{-2} C_{\widetilde{f}} + k^{-1} C_{\widetilde{g}} + k^{-1} \| \nabla u \|_{L^2(\Omega)} +  \| u \|_{L^2(\Omega)}) \\
    &\lesssim C \gamma^p \max\{k, p\}^p  ( k^{-2} C_{\widetilde{f}} + k^{-1} C_{\widetilde{g}} + \CsolveM k^{-1} (C_{\widetilde{f}} + k C_{\widetilde{g}})) \\
    &\lesssim C \gamma^p \max\{k, p\}^p  \CsolveM k^{-1} (C_{\widetilde{f}} + k C_{\widetilde{g}}) \qquad \forall p \geq 2, 
	\end{align*}
	where we applied the stability estimate \eqref{eq:stability_estimate} as well as $\CsolveM \gtrsim 1$. The stability estimate \eqref{eq:stability_estimate}
shows that the bound is valid for $p \in \{0,1\}$ as well, so that 
$\Skm(\widetilde{f},\widetilde{g}) \in \AnaClass{C \CsolveM k^{-1} (C_{\widetilde{f}} +k C_{\widetilde{g}})}{\gamma}{\partition}$, i.e., 
\ref{WP3} holds with $\Can = O(1)$.  

For the case that the matrix $A$ is not SPD but merely satisfies \eqref{eq:HH-robin-assumption-coefficients-a}, we note that the 
proof of \cite[Props.~{5.5.1}, {5.5.3}, {5.5.4}]{melenk02} relies on a (piecewise) $H^2$-regularity result for the principal part of the operator
(cf.~\cite[Lemmas~{5.5.5}, {5.5.8}, {5.5.9}]{melenk02}), which in turns hinges on the difference quotient technique of Nirenberg. We note that the condition
\eqref{eq:HH-robin-assumption-coefficients-a} provides a coercivity that makes the difference quotient technique applicable for Neumann and transmission
problems so that also \cite[Props.~{5.5.1}, {5.5.3}, {5.5.4}]{melenk02} holds for complex-valued coefficients $A$, $n$ satisfying 
\eqref{eq:HH-robin-assumption-coefficients}.

\emph{Proof of \ref{item:lemma:HH-robin-WP-ii}, \ref{item:lemma:HH-robin-WP-iii}:} By inspection. 
\end{proof}
Under the assumption \ref{WP2} we have the following quasi-optimality result for the heterogeneous Helmholtz equation
with Robin boundary conditions: 
\begin{theorem}[Robin b.c.\ and PML]
\label{thm:quasi-optimality-robin}
Let $t = 1/2$. 
Let Assumption~\ref{assumption:smoothness_domain_boundary_and_interface} 
and \eqref{eq:HH-robin-assumption-coefficients} be valid. Let $G$ be analytic on a tubular neighborhood of $\Gamma$. 
For problem \eqref{eq:Lkm} with $\TkGammaM = ik G$ 
assume the polynomial well-posedness of the solution operator $\Skm$ of Assumption~\ref{assumption:everything_is_polynomial_well_posed}. 
In the discretization setting of Assumption~\ref{assumption:quasi_uniform_regular_meshes}, for each $c_2 > 0$ there are constants $c_1$, $C> 0$
independent of $h$, $p$, and $k$ such that under the scale-resolution condition \eqref{eq:scale-resolution} the quasi-optimality
result \eqref{eq:quasi-optimality} holds. 
\end{theorem}
\begin{proof}
Combine Lemmas~\ref{lemma:HH-robin-AP} and \ref{lemma:HH-robin-WP} with the stability Assumption~\ref{assumption:everything_is_polynomial_well_posed}. 
\end{proof}
%--------------------------
\subsection{Heterogeneous Helmholtz problem in the full space}
\label{problem:HH_full_space}
%--------------------------

The heterogeneous Helmholtz problem with Sommerfeld radiation condition 
in full space $\mathbb{R}^d$ is
to find $U \in H^1_{\mathrm{loc}}(\mathbb{R}^d)$ such that
\begin{equation}
\label{eq:heterogeneous_helmholtz_problem_full_space}
	\begin{alignedat}{2}
-k^2 n^2 U 	- \nabla \cdot (A \nabla U)                  &= f                                   \quad  &&\text{in }  \mathbb{R}^d,\\
		\left| \partial_r U - i k U  \right|   &= o\left( \| x \|^{\frac{1-d}{2}} \right)      \quad  &&\text{for } \| x \| \to \infty,
	\end{alignedat}
\end{equation}
is satisfied in the weak sense.
Here, $\partial_r$ denotes the derivative in the radial direction.
We assume $f$, $A$, and $n$ (which are defined on ${\mathbb R}^d$) to be local
is the sense that $f$, $A-I$, and $n-1$ are compactly supported.

To approximate the problem with a FEM, 
we introduce a bounded Lipschitz domain $\Omega \subset \mathbb{R}^d$
such that $\mathrm{supp} \, f \subset \overline{\Omega}$,
$\mathrm{supp} \, (A-I) \subset \Omega$, 
and $\mathrm{supp} \, (n-1) \subset \overline{\Omega}$. 
Since we have a lot of freedom in the choice of $\Omega$,
we will assume that it has been designed so that
\begin{equation}
\label{eq:non-trapping}
\text{ $\Gamma := \partial \Omega$ is analytic
and $\Omega^+ = \mathbb R^d \setminus \overline{\Omega}$
is non-trapping, \cite[Def.~{1.1}]{baskin-spence-wunsch16}.
}
\end{equation}
%the boundary $\Gamma := \partial \Omega$ is analytic,
%and the exterior domain $\Omega^+ = \mathbb R^d \setminus \overline{\Omega}$
%is non-trapping \cite[Def.~{1.1}]{baskin-spence-wunsch16}.}

Problem~\eqref{eq:heterogeneous_helmholtz_problem_full_space}
can then be reformulated on $\Omega$ using the
Dirichlet-to-Neumann operator
$\DtN \colon H^{1/2}(\Gamma) \rightarrow H^{-1/2}(\Gamma)$,
which is given by
$g \mapsto \mathrm{DtN}_k g \coloneqq \partial_n w:=\normal \cdot \nabla w$,
and
$w \in H^1_{\mathrm{loc}}(\mathbb{R}^d \setminus \Omega)$ 
is the unique weak solution to
\begin{equation*}
	\begin{alignedat}{2}
    - \Delta w - k^2 w                     &= 0                                                    \quad  &&\text{in }  \Omega^+, \\% \mathbb{R}^d \setminus \Omega,\\
                     w                     &= g                                                    \quad  &&\text{on }  \Gamma,\\
		\left| \partial_r w - i k w  \right|   &= o\left( \| x \|^{\frac{1-d}{2}} \right)      \quad  &&\text{for } \| x \| \to \infty.
	\end{alignedat}
\end{equation*}
The heterogeneous Helmholtz problem in full space is to find $u \in H^1(\Omega)$ such that
\begin{equation}
\label{eq:heterogeneous_helmholtz_problem_DtN}
	\begin{alignedat}{2}
\Lkm u:= -k^2 n^2 u 	- \nabla \cdot (A \nabla u)  &= f     \quad  &&\text{in } \Omega,\\
\partial_{n_A} u - \TkGammaM u:= \partial_{n_A} u - \mathrm{DtN}_k u   &= g         \quad  &&\text{on } \Gamma.
	\end{alignedat}
\end{equation}
For $g = 0$, the solution $u$ of \eqref{eq:heterogeneous_helmholtz_problem_DtN}
is then the restriction to $\Omega$ of the solution of \eqref{eq:heterogeneous_helmholtz_problem_full_space}. 
%
% The coupling interface $\Gamma$ between the interior and exterior domain
% may be chosen to be arbitrary as long as the exterior domain is non-trapping
% \cite[Def.~{1.1}]{baskin-spence-wunsch16} and $\Gamma$ itself is analytic.
%
If $d = 2$, we assume that $\operatorname{diam} \Omega < 1$, which ensures that
the single layer operator $V_0:H^{-1/2}(\Gamma) \rightarrow H^{1/2}(\Gamma)$
appearing below is boundedly invertible.

We define the operator $\DtNz$ as the map $\DtNz: g \mapsto \partial_n u^{\rm ext}$, where 
$u^{\rm ext}$ solves
\begin{subequations}
\label{eq:dtn0}
\begin{align}
\label{eq:dtn0-a}
- \Delta u^{\rm ext} & = 0 \quad \mbox{in $\Omega^+$}, \\% \mbox{ in $\mathbb{R}^d \setminus \overline{\Omega}$}, \\
\label{eq:dtn0-b}
u^{\rm ext} & = g \quad \mbox{ on $\Gamma$}, \\
\label{eq:dtn0-c}
u^{\rm ext}(x)   & =
\begin{cases} O(1/\|x\|) & \mbox{ as $\|x\| \rightarrow \infty$ for $d = 3$} \\
              b \ln \|x\|  + O(1/\|x\|) & \mbox{ as $\|x\| \rightarrow \infty$ for some $b \in \mathbb{R}$ for $d = 2$.} 
\end{cases} 
\end{align}
\end{subequations}

\begin{remark}
The reason for choosing the particular condition \eqref{eq:dtn0-c} at $\infty$ is that it ensures the representation formula 
$$
u^{\rm ext}(x) = - (\widetilde{V}_0 \partial_n u^{\rm ext})(x) + (\widetilde{K}_0 u^{\rm ext})(x)
$$
with the usual single layer $\widetilde{V}_0$ and double layer potential $\widetilde{K}_0$ for the Laplacian; 
see, e.g., \cite[Thm.~{8.9}]{mclean00} or \cite[Lem.~{3.5}]{costabel-stephan85}. This particular choice will reappear
in the analysis of the difference $\DtN - \DtNz$ in Lemma~\ref{lemma:properties_dtn_operators}.
For our choice of condition at $\infty$, one has the Calder{\'o}n identity 
\begin{equation}
\label{eq:calderon} 
\left( 
\begin{array}{c} 
u^{\rm ext} \\ \partial_n u^{\rm ext}
\end{array}
\right)
 = 
\left(
\begin{array}{cc} 
\frac{1}{2} + K_0 & - V_0 \\
-D_0 & \frac{1}{2} - K^\prime_0
\end{array}
\right)
\left( 
\begin{array}{c} 
u^{\rm ext} \\ \partial_n u^{\rm ext}
\end{array}
\right), 
\end{equation}
where $V_0$, $K_0$, $K_0^\prime$, and $D_0$ denote the single layer, double layer, adjoint double layer, and the hypersingular
operator, respectively (see, e.g.,  \cite{sauter-schwab11,steinbach08} for the precise definition). By definition, we 
have $\partial_n u^{\rm ext} = \DtNz u^{\rm ext}$. From the Calder{\'o}n identity 
\eqref{eq:calderon}, one easily concludes 
\begin{equation}
\label{eq:DtN0-definite}
(\DtNz u^{\rm ext},u^{\rm ext})_{L^2(\Gamma)} = - (D_0 u^{\rm ext},u^{\rm ext})_{L^2(\Gamma)} - (V_0 \partial_n u^{\rm ext}, \partial_n u^{\rm ext})_{L^2(\Gamma)} \leq 0, 
\end{equation}
where we employed that $D_0$ is positive semidefinite and $V_0$ is positive definite, \cite[Thm.~{6.23}, Cor.~{6.25}]{steinbach08}. 
\eremk
\end{remark}
For the auxiliary problem \eqref{eq:Lkp}, we select $\TkGammaP = \DtNz$. 
%\begin{equation}
%\label{eq:Lkp-HH-DtN}
%\Lkp u = -\nabla \cdot (A(x) \nabla u) + k^2 u, 
%\qquad  \TkOmegaP = -k^2 u, 
%\qquad \TkGammaP = \DtNz. 
%\end{equation}

In order to show that Problem \eqref{eq:heterogeneous_helmholtz_problem_DtN} fits into 
the framework of Section~\ref{section:galerkin} we need to analyze the operators $\DtN$ and $\DtNz$: 
\begin{lemma}
\label{lemma:properties_dtn_operators}
  Let $\Omega \subset \mathbb{R}^d$, $d \in \{2, 3\}$, be a bounded Lipschitz domains with smooth boundary $\Gamma$.
  Then the following holds:
  \begin{enumerate}[nosep, start=1, label={(\roman*)},leftmargin=*]
    \item
    \label{item:lemma:properties_dtn_operators_dtn_zero_coercive}
    $ - \langle \mathrm{DtN}_0 u, u\rangle \geq 0$ for all $u \in H^{1/2}(\Gamma)$.
    \item 
    \label{item:lemma:properties_dtn_operators_dtn_zero_mapping-property}
    The operator $\DtNz$ is a bounded linear operator $H^{s+1/2}(\Gamma) \rightarrow H^{s-1/2}(\Gamma)$ for every $s \ge 0$. 
    \item
    \label{item:lemma:properties_dtn_operators_splitting_finite_ana_analytic_part}
Assume \eqref{eq:non-trapping}. 
Let the ball $B_R(0)$ satisfy $\overline{\Omega} \subset B_R(0)$, 
set $\Omega_R:= B_R(0) \setminus \Gamma$. Let $s \geq 0$ be given.
    Then
    \begin{equation*}
      \mathrm{DtN}_k - \mathrm{DtN}_0 =  k B + \llbracket \partial_n \tilde{A} \rrbracket
    \end{equation*}
    where the linear operators $ B \colon H^{s}(\Gamma) \rightarrow H^{s}(\Gamma) $
    and $ \tilde{A} \colon H^{s}(\Gamma) \rightarrow C^{\infty}(\Omega_R)$
    satisfy for all $u \in H^{s}(\Gamma)$
    \begin{equation*}
      \| B u \|_{s,\Gamma} \lesssim \| u \|_{s,\Gamma},
      \qquad
      \tilde{A} u \in \AnaClassmelenk{C k^\beta \| u \|_{s,\Gamma}}{ \gamma}{\Omega_R}
    \end{equation*}
    with $\beta = 7/2+d/2$, and constants $C$, $\gamma \geq 0$ independent of $k$. Here, the operator 
      $\llbracket \partial_n \tilde{A} \rrbracket$ realizes the jump of the normal derivative of $\tilde{A} v$ and is given by 
      $\llbracket \partial_n \tilde{A} \rrbracket v:= \normal \cdot \nabla (\tilde{A} v)|_{\Omega^+} - \normal\cdot \nabla (\tilde{A} v)|_{\Omega}$ 
(see \eqref{eq:jump}).
    \item
    \label{item:lemma:properties_dtn_operators_sphere}
    Let $\Gamma = \partial B_1(0) \subset {\mathbb R}^3$ be the unit ball in dimension $d = 3$. 
  	Then $\mathrm{DtN}_k - \mathrm{DtN}_0 \colon H^s(\Gamma) \rightarrow H^s(\Gamma)$ satisfies, for every $s \ge 0$, 
  	\begin{equation*}
  		\| \mathrm{DtN}_k u - \mathrm{DtN}_0 u \|_{s,\Gamma} \lesssim k \| u \|_{s,\Gamma} \qquad \forall u \in H^s(\Gamma).
  	\end{equation*}
  \end{enumerate}
\end{lemma}

\begin{proof}
\emph{Proof of \ref{item:lemma:properties_dtn_operators_dtn_zero_coercive}:}
  $ - \langle \mathrm{DtN}_0 u, u\rangle \geq 0$ is asserted in \eqref{eq:DtN0-definite}. %see \cite[Lem.~{2.2(vi)}]{aurada-melenk-praetorius15}.

\emph{Proof of \ref{item:lemma:properties_dtn_operators_dtn_zero_mapping-property}:} Follows from regularity properties of elliptic boundary value problems.
Alternatively, one could appeal to the representation of the (exterior) $\DtNz$ as $\DtNz = V^{-1}_0 (-1/2+K_0)$ and the mapping properties of 
$V_0$ and $K_0$ as given in, e.g., \cite[Thm.~{3.2.2}]{sauter-schwab11}. 

\emph{Proof of \ref{item:lemma:properties_dtn_operators_splitting_finite_ana_analytic_part}:} See Appendix~\ref{section:dtn_via_bios}.

\emph{Proof of \ref{item:lemma:properties_dtn_operators_sphere}:}
  For $\Gamma = \partial B_1(0) \subset {\mathbb R}^3$, the operator $\mathrm{DtN}_k$ has an explicit series representation in terms 
  of spherical harmonics. That is, denoting by $Y_\ell^m$ the standard spherical harmonics, a function $u \in L^2(\Gamma)$ can be expanded as 
  \begin{equation*}
  	u = \sum_{\ell = 0}^\infty \sum_{m = -\ell}^\ell u_\ell^m Y_\ell^m. 
  \end{equation*}
One has the norm equivalence % equivalences
$\|u\|^2_{H^s(\Gamma)} \sim \sum_{\ell=0}^\infty \sum_{m=-\ell}^{\ell} |u_\ell^m|^2 (\ell+1)^{2s}$ 
for every $s \ge 0$, see, e.g., \cite[Sec.~{2.5.1}]{nedelec01}.
  The operators $\DtNz$ and $\DtN$ can be written as
  \begin{align*}
  	\mathrm{DtN}_0 u &= - \sum_{\ell = 0}^\infty \sum_{m = -\ell}^\ell (\ell+1) u_\ell^m Y_\ell^m, & 
  	\mathrm{DtN}_k u &= \sum_{\ell = 0}^\infty \sum_{m = -\ell}^\ell z_\ell(k)  u_\ell^m Y_\ell^m,
  \end{align*}
  with explicit estimates for the symbol $z_\ell(k)$ given in \cite[Sec.~{2.6.3}]{nedelec01}. 
  The formulas for $\DtNz$ and $\DtN$ immediately give
	\begin{equation}
		\mathrm{DtN}_k u - \mathrm{DtN}_0 u = \sum_{\ell = 0}^\infty \sum_{m = -\ell}^\ell (z_\ell(k ) + \ell + 1)  u_\ell^m Y_\ell^m,
	\end{equation}
	From \cite[Lem.~{3.2}, eqn.~{(3.28)}]{demkowicz-ihlenburg01}, where the operator $\mathrm{DtN}_k$ has the opposite sign, we have
	\begin{equation}
\label{eq:demkowicz-ihlenburg}
		\ell+1-k \leq - \mathrm{Re} \, z_\ell(k) \leq \ell+1+k.
	\end{equation}
	Consequently, we immediately have
	\begin{equation}
\label{eq:demkowicz-ihlenburg-10}
		| \mathrm{Re} \, z_\ell(k) + \ell + 1| \leq k.
	\end{equation}
	From \cite[Thm.~{2.6.1}, eqn.~{(2.6.24)}]{nedelec01} we have
	\begin{equation}
\label{eq:symbol-estimate-nedelec}
		0 \leq \mathrm{Im} \, z_\ell(k)  \leq k.
	\end{equation}
Combining \eqref{eq:demkowicz-ihlenburg-10} and \eqref{eq:symbol-estimate-nedelec}  yields 
	\begin{equation}
\label{eq:symbol-estimate-DtN}
		| z_\ell(k) + \ell + 1| \leq 2k.
	\end{equation}
	For $u \in H^s(\Gamma)$ and with the previous estimate we have
	\begin{align*}
		\| \mathrm{DtN}_k u - \mathrm{DtN}_0 u \|_{s,\Gamma}^2
		&\sim  \sum_{\ell = 0}^\infty \sum_{m = -\ell}^\ell (\ell+1)^{2s} |z_\ell(k) + \ell + 1|^2  |u_\ell^m|^2 \\
		&\leq (2k)^2 \sum_{\ell = 0}^\infty \sum_{m = -\ell}^\ell (\ell+1)^{2s} |u_\ell^m|^2 Y_\ell^m
		\sim  (2k)^2 \| u \|_{s,\Gamma}^2,
	\end{align*}
	which yields the result.
\end{proof}

\begin{lemma}[\protect{\APs} for full space problem]
\label{lemma:HH-full-space-AP}
Let Assumption~\ref{assumption:smoothness_domain_boundary_and_interface} 
and \eqref{eq:HH-robin-assumption-coefficients}, \eqref{eq:non-trapping} be valid. 
For $\Lkm$, $\Lkp$  given by \eqref{eq:Lkm}, \eqref{eq:Lkp}
with $\TkGammaM = \DtN$ and $\TkGammaP = \DtNz$ the following holds 
with $t = 1/2$: 
\begin{enumerate}[nosep, start=1, label={(\roman*)},leftmargin=*]
\item 
\label{item:lemma:HH-DtN-AP-i}
$\bkp$ is bounded uniformly in $k$, and \ref{AP1} holds with $\sigma = 1$. 
\item 
\label{item:lemma:HH-DtN-AP-ii}
$\|\Skp(f,g)\|_{1,t,k} \leq C \left[ k^{-1} \|f\|_{0,\Omega} + k^{-1/2} \|g\|_{0,\Gamma}\right]$, and 
\ref{AP2} holds for any fixed $\smax \ge 0$. 
\item 
\label{item:lemma:HH-DtN-AP-iii}
The splitting of $\TkOmegaM - \TkOmegaP$ in \ref{AP3} is given by 
$\ROD u = k^2 (1+n^2) u$, $\AOD = 0$. The splitting of 
$\TkGammaM-\TkGammaP$ in \ref{AP3} is given by 
$\RGD u= k B$  and $\AGD =  \llbracket \partial_n \tilde{A} \rrbracket$ 
with the operators $B$, $\llbracket \partial_n \tilde{A} \rrbracket$ given by 
Lemma~\ref{lemma:properties_dtn_operators}\ref{item:lemma:properties_dtn_operators_splitting_finite_ana_analytic_part}. (See also 
Remark~\ref{remk:bdy-fcts}\ref{item:remk:bdy-fcts-iii} for the interpretation of  
$\llbracket \partial_n \tilde{A} \rrbracket$ as a mapping into analyticity classes of the form  $\AnaClassbdy{M}{\gamma}{T}{\Gamma}$.)
\item 
\label{item:lemma:HH-DtN-AP-iv}
For the case $\Gamma = \partial B_1(0) \subset {\mathbb R}^3$, the statement 
\ref{item:lemma:HH-DtN-AP-iii} holds true with the operator $\AGD = 0$ and the
operator $\RGD = B = \DtN - \DtNz$ of order zero bounded uniformly in $k$ as
described in Lemma~\ref{lemma:properties_dtn_operators}\ref{item:lemma:properties_dtn_operators_sphere}. 
\end{enumerate}
\end{lemma}
\begin{proof}
\emph{Proof of \ref{item:lemma:HH-DtN-AP-i}:}
Continuity of $\bkp$ follows by inspection. The coercivity \ref{AP1} of $\bkp$ 
with $\sigma = 1$ follows from \eqref{eq:HH-robin-assumption-coefficients-a} and
the positivity of $-\TkGammaP = -\DtNz$ by \eqref{eq:DtN0-definite}.

\emph{Proof of \ref{item:lemma:HH-DtN-AP-ii}:} Set $w:= \Skp(f,g)$. 
The {\sl a priori} bound 
\begin{equation}
\label{item:lemma:HH-DtN-AP-10}
\|w\|_{1,\Omega} + k \|w\|_{0,\Omega} + |w|_{1/2,\Gamma} 
\lesssim \|\Skp(f,g)\|_{1,t,k}
\lesssim k^{-1} \|f\|_{0,\Omega} + k^{-1/2} \|g\|_{0,\Gamma} 
\end{equation}
follows by Lax-Milgram. 
To show the shift theorem asserted in \ref{AP2}, we reformulate the problem solved by 
  $w = \Skp(f,g)$, i.e.,  
 \begin{align*} 
      - \nabla \cdot (A \nabla w) &= f - k^2 w \quad &&\text{in } \Omega,
& && 
      \partial_n w - \mathrm{DtN}_0 w &= g  \quad       &&\text{on } \Gamma.
  \end{align*}
as a transmission problem. We define the function 
\begin{equation*}
\widetilde w:= \begin{cases}
w & \mbox{ in $\Omega$} \\
w^{\rm ext}:= - \widetilde{V}_0 (\DtNz \gamma^{int} w) + \widetilde{K}_0 \gamma^{int} w& \mbox{ in $\Omega^+$}. 
\end{cases}
\end{equation*}
By the mapping properties of the potentials 
$\widetilde{V}_0$, $\widetilde{K}_0$, \cite{sauter-schwab11,steinbach08}, we have for any bounded domain $\widetilde{\Omega}$ 
\begin{equation}
\label{eq:transmission-problem-10}
\|\widetilde w\|_{H^1(\widetilde{\Omega})} 
\lesssim \|w\|_{1,\Omega}  = \|\Skp(f,g)\|_{1,t,k} \stackrel{\eqref{item:lemma:HH-DtN-AP-10}}{ \lesssim }
k^{-1} \|f\|_{0,\Omega}
+ k^{-1/2} \|g\|_{0,\Gamma}.
\end{equation}
The function $\widetilde{w}$ satisfies 
\begin{align*}
      - \nabla \cdot ( A \nabla \widetilde w ) + k^2 \widetilde w &= f    \quad   \text{in } \Omega,    
& - \Delta \widetilde w                            &= 0    \quad   \text{in } \Omega^+,    \\ 
      \llbracket\widetilde w\rrbracket                        &= 0    \quad   \text{on } \Gamma, 
& \llbracket\partial_n \widetilde w\rrbracket             &= g    \quad   \text{on } \Gamma.   
    \end{align*}
Let $\widetilde\Omega \supset\supset\Omega$ with $\partial\widetilde\Omega$ smooth. 
  Let $\chi\in C^\infty_0(\widetilde{\Omega})$ with $\chi \equiv 1$ on a neighborhood of $\overline{\Omega}$. 
Upon writing $\hat{A}$ for the extension of $A$ by the identity matrix and the characteristic function $\chi_\Omega$ of $\Omega$, we have  
  \begin{equation}
\label{eq:transmission-problem}
    \begin{alignedat}{2}
      - \nabla( \hat{A} \nabla (\chi \widetilde{w}) ) + k^2 \mathbbm{1}_{\Omega} \chi \widetilde{w} &= \hat{f} \quad   &&\text{in } \widetilde{\Omega},\\
      \llbracket \chi \widetilde{w} \rrbracket                                    &= 0 \quad   &&\text{on } \Gamma, \\
      \llbracket \partial_n (\chi \widetilde{w})\rrbracket                         &= g \quad   &&\text{on } \Gamma, \\
      \chi \widetilde{w}                                                 &= 0 \quad   &&\text{on } \partial\widetilde{\Omega}.
    \end{alignedat}
  \end{equation}
  with $\hat{f} = f$ in $\Omega$ and $\hat{f} = 2 \nabla \widetilde{w} \cdot \nabla \chi + \widetilde{w} \Delta \chi$ 
in $\widetilde{\Omega} \setminus \overline{\Omega}$.
Since $\hat{A}$ satisfies \eqref{eq:HH-robin-assumption-coefficients-a} with $\Omega$ replaced by $\widetilde{\Omega}$, 
the Lax-Milgram Lemma gives for the solution $\chi\widetilde{w}$ of 
\eqref{eq:transmission-problem}
\begin{align}
\label{eq:transmission-problem-20}
& 
\|\nabla (\chi \widetilde{w})\|_{L^2(\widetilde{\Omega})} + 
\|(1+k^2\chi_\Omega) (\chi \widetilde{w})\|_{L^2(\widetilde{\Omega})} 
 \lesssim \|\hat{f}\|_{H^{-1}(\widetilde{\Omega})} + \|g\|_{H^{-1/2}(\Gamma)} 
\\ 
\nonumber 
& \qquad \lesssim \|\hat{f}\|_{L^2(\widetilde{\Omega})} + \|g\|_{1/2,\Gamma}
\lesssim \|f\|_{0,\Omega} + \|g\|_{1/2,\Gamma} + \|\widetilde{w}\|_{1,\widetilde\Omega\setminus\Omega}. 
\end{align}
  %Since $f \in L^2(\Omega)$, $\hat{f} \in L^2(\widetilde{\Omega})$ and $g \in H^{1/2}(\Gamma)$ 
The function $\chi\widetilde{w}$ satisfies an elliptic transmission problem with piecewise smooth coefficients in the equation
and smooth interface $\Gamma\cup \interface$. Although the coefficient $\hat A$ is complex-valued, it satisfies
\eqref{eq:HH-robin-assumption-coefficients-a} and, as discussed in the proof of Lemma~\ref{lemma:HH-robin-AP}, (piecewise) $H^2$-regularity
as formulated in \cite[Prop.~{5.4.8}]{melenk02} is available. Such an $H^2$-regularity gives 
  $\chi \widetilde{w} \in H^2(\widetilde{\Omega} \setminus (\interface \cup \Gamma))$
  and consequently $w = \restr{(\chi \widetilde{w})}{\Omega} \in H^2(\Omega \setminus \interface)$
  together with 
  \begin{align*}
    \| w \|_{2,\Omega \setminus \interface}
    &= \| \chi \widetilde{w}\|_{2,\Omega \setminus \interface}
    \lesssim \| \hat{f} + k^2 \mathbbm{1}_{\Omega} \chi\widetilde{w} \|_{0,\widetilde{\Omega}} + \| g \|_{1/2,\Gamma} 
+ \|\chi\widetilde{w}\|_{1,\widetilde{\Omega}} \\
    &\lesssim \| f \|_{0,\Omega} +
      \| 2 \nabla (\chi\widetilde{w}) \cdot  \nabla \chi + (\chi\widetilde{w}) \Delta \chi \|_{0,\widetilde{\Omega} \setminus \Omega}  +
      k^2 \| w \|_{0, \Omega} + \| g \|_{1/2,\Gamma} \\
    & \!\!\!\stackrel{\eqref{eq:transmission-problem-20}}{ \lesssim}  \!
      \| f \|_{0,\Omega} + \| g \|_{1/2,\Gamma} + k^2 \| w \|_{0, \Omega} 
     \!\stackrel{ \eqref{item:lemma:HH-DtN-AP-10}} {\lesssim} \!
      \| f \|_{0,\Omega} +
       \| g \|_{1/2,\Gamma} + k^{1/2} \|g\|_{0,\Gamma}. 
  \end{align*}

\emph{Proof of \ref{item:lemma:HH-DtN-AP-iii}, \ref{item:lemma:HH-DtN-AP-iv}:} follows from Lemma~\ref{lemma:properties_dtn_operators}. 
\end{proof}
\begin{lemma}[\WPs for full space]
\label{lemma:HH-full-space-WP}
Let Assumption~\ref{assumption:smoothness_domain_boundary_and_interface} 
and \eqref{eq:HH-robin-assumption-coefficients}, \eqref{eq:non-trapping}  be valid. 
For $\Lkm$, $\Lkp$  given by \eqref{eq:Lkm}, \eqref{eq:Lkp}
with $\TkGammaM = \DtN$ and $\TkGammaP = \DtNz$ the following holds 
with $t = 1/2$ if \ref{WP2} holds: 
\begin{enumerate}[nosep, start=1, label={(\roman*)},leftmargin=*]
\item 
\label{item:lemma:HH-full-space-WP-i}
\ref{WP1} holds with $\CcontM = O(k^\beta)$ with $\beta$ given by Lemma~\ref{lemma:properties_dtn_operators}, and \ref{WP3} holds 
with $\Can = O(1)$ (uniformly in $k$). 
\item 
\label{item:lemma:HH-full-space-WP-ii}
Let the operators $B$ and 
$\llbracket \partial_n \tilde{A} \rrbracket$ be given by 
Lemma~\ref{lemma:properties_dtn_operators}\ref{item:lemma:properties_dtn_operators_splitting_finite_ana_analytic_part}. 
Then \ref{WP4} holds with $\ROM u = k^2 n u$, $\AOM  = 0$ and the choices 
$\RGM = \DtNz + k B$, $\AGM = \llbracket \partial_n \tilde{A} \rrbracket$. 
In particular, for the constants in Assumption~\ref{assumption:everything_is_polynomial_well_posed} we have 
$\CcontM + C^-_{A, k} + C^\Delta_{A,k} = O(k^\beta)$ with $\beta$ given by Lemma~\ref{lemma:properties_dtn_operators}. 
\item 
\label{item:lemma:HH-full-space-WP-iia}
For $\Gamma = \partial B_1(0) \subset {\mathbb R}^3$ 
the assumption \ref{WP4} holds with $\ROM u = k^2 n u$, $\AOM  = 0$ and the choices 
$\RGM = \DtN$ and  $\AGM = 0$. 
In particular, for the constants in Assumption~\ref{assumption:everything_is_polynomial_well_posed} we have 
$\CcontM + C^-_{A, k} + C^\Delta_{A,k} = O(1)$. 
\item 
\label{item:lemma:HH-full-space-WP-iii}
\ref{WP5} holds. 
\end{enumerate}
\end{lemma}
\begin{proof}
\emph{Proof of \ref{item:lemma:HH-full-space-WP-i}:} 
  The full space problem corresponds to the choice $\TkGammaM u = \mathrm{DtN}_k u$.
  As in the proof of Lemma~\ref{lemma:HH-full-space-AP}, the function $u$ is the restriction to $\Omega$ of the solution 
of the following full-space transmission problem: 
  \begin{equation*}
    \begin{alignedat}{2}
      - \nabla \cdot (A \nabla u) - k^2 n^2 u &= f \quad   &&\text{in } \Omega \cup \Omega^+,\\
      \llbracket u\rrbracket                           &= 0 \quad   &&\text{on } \Gamma, \\
      \llbracket \partial_n u\rrbracket               &= g \quad   &&\text{on } \Gamma.\\
                                            u &\text{ satisfies the radiation condition.}
    \end{alignedat}
  \end{equation*}
  From now on the proof is completely analogous to the Robin case of Lemma~\ref{lemma:HH-robin-WP} replacing 
the use of Neumann conditions by transmission conditions on $\Gamma$.  

\emph{Proof of \ref{item:lemma:HH-full-space-WP-ii}:} We decompose 
$\TkGammaM = \DtN = \DtNz + (\DtN - \DtNz) = (\DtNz + k B) + \llbracket \partial_n \tilde{A} \rrbracket =: \RGM + \AGM$. 
The operator $\AGM$ is of the form given in \ref{WP4}. For the operator $\RGM$, we observe using the mapping properties 
of $\DtNz$ and $B$
\begin{align*}
\left| \langle \RGM u,v\rangle \right| & \leq 
\left| \langle \DtNz u,v\rangle \right| + k 
\left| \langle B u,v\rangle \right| 
\lesssim \|u\|_{1/2,\Gamma} \|v\|_{1/2,\Gamma} + k \|u\|_{0,\Gamma} \|v\|_{0,\Gamma}  \\
& \stackrel{t = 1/2}{\lesssim} \|u\|_{1,t,k} \|v\|_{1,t,k}. 
\end{align*}

\emph{Proof of \ref{item:lemma:HH-full-space-WP-iii}:} By inspection. 
\end{proof}
Fixing the computational domain $\Omega$ and the coupling boundary $\Gamma$ on which the DtN-operator is employed, we have the 
following quasi-optimality result: 
\begin{theorem}[full space]
\label{thm:quasi-optimality-full-space}
Let $t = 1/2$. 
Let Assumption~\ref{assumption:smoothness_domain_boundary_and_interface} 
and \eqref{eq:HH-robin-assumption-coefficients}, \eqref{eq:non-trapping} be valid. For problem \eqref{eq:Lkm} with $\TkGammaM = \DtN$ 
assume the polynomial well-posedness of the solution operator $\Skm$ of Assumption~\ref{assumption:everything_is_polynomial_well_posed}. 
In the discretization setting of Assumption~\ref{assumption:quasi_uniform_regular_meshes}, for each $c_2 > 0$ there are constants $c_1$, $C> 0$
independent of $h$, $p$, and $k$ such that under the scale-resolution condition \eqref{eq:scale-resolution} the quasi-optimality
result \eqref{eq:quasi-optimality} holds. 
\end{theorem}
\begin{proof}
Combine Lemmas~\ref{lemma:HH-full-space-AP} and \ref{lemma:HH-full-space-WP} with the stability Assumption~\ref{assumption:everything_is_polynomial_well_posed}. 
\end{proof}

\begin{remark}
In practice, it is hard to realize the operator $\DtN$ exactly. Possible numerical realizations of $\DtN$ 
include FEM-BEM coupling \cite{mascotto-melenk-perugia-rieder20} or to truncate the  series in the case of 
spherical $\Gamma$ \cite{xu_tao_2021a}. An analysis of the additional variational crimes incurred is beyond the scope of the present analysis. 
%For the sake
%of simplicity, we do not take into account the additional level of discretization here.
\eremk
\end{remark}

%---------------------
\subsection{Heterogeneous Helmholtz problem with second order ABCs}
\label{problem:HH_with_second_order_abcs}
%---------------------
Various boundary conditions that are formally of higher order than the Robin boundary conditions have been
proposed in the literature, notably for case of a sphere in ${\mathbb R}^2$. 
Our theory covers those proposed by 
Bayliss-Gunzburger-Turkel \cite{bayliss-gunzburger-turkel82},
Enquist-Majda \cite{enquist-majda77, enquist-majda79} and
Feng \cite{feng83}, see \cite[Sec.~{3.3.3}, Table~{3.2}]{ihlenburg98}
for a comprehensive comparison.
We denote by $\Delta_\Gamma$ and $\nabla_\Gamma$
the surface Laplacian and the surface gradient on $\Gamma$ (see, e.g., \cite[Sec.~{2.5.6}]{nedelec01}).
The model problem for a heterogeneous Helmholtz equation with second order absorbing boundary conditions is
to find $u$ such that 
\begin{equation}
\label{eq:form-of-2nd-order-ABC}
  \begin{alignedat}{2}
    \Lkm u:= -k^2 n^2 u - \nabla \cdot (A \nabla u)  &= f  \quad  &&\text{in } \Omega,\\
    \partial_{n_A} u - \TkGammaM u := \partial_{n_A} u - (\beta u + \alpha \Delta_\Gamma u)              &= g  \quad  &&\text{on } \Gamma,
  \end{alignedat}
\end{equation}
%with $\TkGammaM$ of the form 
%\begin{equation}
%\label{eq:form-of-2nd-order-ABC}
%  \TkGammaM u = \beta u + \alpha \Delta_\Gamma u,
%\end{equation}
for $k$-dependent parameters $\alpha$, $\beta$ satisfying 
\begin{align}
\label{eq:alpha_dependence_in_second_order_abcs}
  \Im \alpha \neq 0
  \quad \text{ and } \quad \left| \Im \alpha\right|  \sim \frac{1}{k}
  \quad \text{ and } \quad | \Re \alpha | \lesssim \frac{1}{k^2},
\quad 
|\beta| \sim k 
\end{align}
for $k \geq k_0 > 0$. 
The choices of $\alpha$ and $\beta$ proposed by Bayliss-Gunzburger-Turkel, Engquist-Mayda, and Feng 
are given in Table~\ref{table:2nd-order-ABC} as presented in \cite[Sec.~{3.3.3}, Table~{3.2}]{ihlenburg98}.
The coefficients $A$, $n$ in \eqref{eq:form-of-2nd-order-ABC} are assumed to satisfy 
\eqref{eq:A-analytic}, \eqref{eq:n-analytic}
and, instead of \eqref{eq:HH-robin-assumption-coefficients-a}, $A$ 
is assumed to be pointwise SPD, viz., 
\begin{equation}
\label{eq:HH-robin-assumption-coefficients-aa}
\xi^H A(x) \xi \ge a_{\min} |\xi|^2 >  0 \qquad \forall \xi \in {\mathbb C}^d\setminus \{0\} \quad \forall x \in \Omega. 
\end{equation}
\begin{table}[t]
  \centering
  \begin{tabular}{|c|c|c|}
  \hline
  Bayliss-Gunzburger-Turkel & Engquist-Mayda & Feng                         \\ \hline
  $\alpha = -\frac{1+ik}{2(1+k^2)}$ & $\alpha = \frac{1+ik}{2k^2}$ & $\alpha = - \frac{i}{2k}$ \\ 
  $\beta = \frac{-2k^2 -\frac{3 i k}{2}+\frac{3}{4}}{2 (ik -1)}$ &
  $\beta = ik - \frac{1}{2}$ & 
  $\beta = ik - \frac{1}{2} - \frac{i}{8 k}$ \\ \hline 
  \end{tabular}
\caption{\label{table:2nd-order-ABC} Parameters for $2^{nd}$ -order ABCs 
for $\Gamma = \partial B_1(0) \subset {\mathbb R}^2$; taken from 
\cite[Table~{3.2}]{ihlenburg98}.}
\end{table}
\begin{remark} 
Our analysis covers operators $\TkGammaM$ of the form 
\eqref{eq:form-of-2nd-order-ABC} for smooth $\Gamma$, which is the form of the second order ABCs for 
spheres. It is expected that second order ABC for general $\Gamma$ have a similar structure and $k$-scaling properties.  
\eremk
\end{remark}
We select $\TkGammaP = \alpha \Delta_\Gamma$. Then we have: 
\begin{lemma}[\APs for $2^{nd}$ order ABCs]
\label{lemma:HH-2nd-order-ABC-AP}
Let Assumption~\ref{assumption:smoothness_domain_boundary_and_interface} be valid. 
Let the coefficients $A$, $n$ satisfy \eqref{eq:HH-robin-assumption-coefficients}
and $A$ additionally \eqref{eq:HH-robin-assumption-coefficients-aa}. 
For $\Lkm$, $\Lkp$ of \eqref{eq:Lkm}, \eqref{eq:Lkp} with 
$\TkGammaM$ given by \eqref{eq:form-of-2nd-order-ABC}--\eqref{eq:alpha_dependence_in_second_order_abcs}
and $\TkGammaP = \alpha \Delta_\Gamma$ the following is true for $t = 1$: 
\begin{enumerate}[nosep, start=1, label={(\roman*)},leftmargin=*] 
\item 
\label{item:lemma:HH-2nd-order-ABC-AP-i}
$\bkp$ is bounded uniformly in $k$, and \ref{AP1} holds for some $\sigma$ depending on $\alpha$ in 
\eqref{eq:alpha_dependence_in_second_order_abcs}. 
\item 
\label{item:lemma:HH-2nd-order-ABC-AP-ii}
$\|\Skp(f,g)\|_{1,t,k} \leq C \left[ k^{-1} \|f\|_{0,\Omega} + k^{-1/2} \|g\|_{0,\Gamma}\right]$, and 
\ref{AP2} holds for $\smax =  0$. 
\item 
\label{item:lemma:HH-2nd-order-ABC-AP-iii}
The splittings of $\TkOmegaM - \TkOmegaP$ and $\TkGammaM - \TkGammaP$ of \ref{AP3} are given by
$\ROD u = k^2 (n^2 +1) u$, $\AOD = 0$ and $\RGD u= \beta u$, $\AGD = 0$ and hence satisfy \ref{AP3}
for any fixed $\smax \ge 0$.  
\end{enumerate}
\end{lemma}
\begin{proof}
\emph{Proof of \ref{item:lemma:HH-2nd-order-ABC-AP-i}:}
  The Laplace-Beltrami operator
  $\Delta_\Gamma \colon H^{1}(\Gamma) \rightarrow H^{-1}(\Gamma)$
  is a bounded linear operator
  since $\Gamma$ has no boundary.
  Furthermore, we have
  $-\langle \alpha \Delta_\Gamma u , u \rangle = \alpha \langle \nabla_\Gamma u , \nabla_\Gamma u \rangle$.
Setting $\sigma:= \overline{\alpha}/|\alpha|$ we estimate 
  \begin{equation}
\label{eq:coercivity-2nd-order}
    -\Re ( \sigma \langle \TkGammaP u , u \rangle ) =   \Re (\sigma \alpha) \langle \nabla_\Gamma u, \nabla_\Gamma u \rangle =  
\Re (\sigma \alpha) | u |_{1, \Gamma}^2 {\gtrsim} k^{-1} | u |_{1, \Gamma}^2, 
  \end{equation}
where, in the last step, we employed the assumptions
$\Im \alpha \sim k^{-1}$ and $| \Re \alpha | \lesssim k^{-2}$ so that $|\alpha| \sim 1/k$. 
Since the coefficient $A$ is assumed to be SPD and satisfy \eqref{eq:HH-robin-assumption-coefficients-aa}, 
  the coercivity of $\bkp$ of \ref{AP1} holds. The uniform-in-$k$ boundedness
of $\bkp$ is easily seen. 

\emph{Proof of \ref{item:lemma:HH-2nd-order-ABC-AP-ii}:}
The {\sl a priori} estimate for $\Skp(f,g)$ follows from coercivity and continuity of $\bkp$. 
 The function $w = \Skp(f,g)$ satisfies
  \begin{align}
\label{eq:2nd-order-5}
      - \nabla \cdot (A \nabla w)            &= f - k^2 w \quad &&\text{in } \Omega,
&&&  
      \partial_{n_A} w                           &= g  + \alpha \Delta_\Gamma w\quad       &&\text{on } \Gamma.
  \end{align}
  Note that for $f \in L^2(\Omega)$ and $g \in L^2(\Gamma)$ we have by Lax-Milgram,
  which follows from \ref{item:lemma:HH-2nd-order-ABC-AP-i}, 
   and $|\alpha| \sim k^{-1}$ (cf.\ \eqref{eq:alpha_dependence_in_second_order_abcs})
  \begin{equation}
\label{eq:2nd-order-10}
    \| \nabla w \|_{0, \Omega} + k \| w \|_{0, \Omega} + k^{-1/2} | w |_{1, \Gamma}
    \lesssim
      k^{-1} \| f \|_{0,\Omega} + k^{-1/2} \| g \|_{0, \Gamma}.
  \end{equation}
  %Especially, we have
  %\begin{equation*}
  %  \| \nabla_\Gamma w \|_{0, \Gamma}
  %  \lesssim
  %  k^{-1/2} \| f \|_{0,\Omega} + \| g \|_{0, \Gamma}.
  %\end{equation*}
  The following surface PDE is satisfied
  \begin{equation*}
    \alpha \Delta_\Gamma w =  - g + \partial_{n_A} w  \quad  \text{on } \Gamma,
  \end{equation*}
  which gives, since $\Gamma$ is a smooth, compact manifold without boundary,  
  \begin{equation}
\label{eq:2nd-order-15}
    \| w \|_{2, \Gamma}
    \!\! \stackrel{|\alpha| \sim k^{-1}}{\lesssim} \!\!k \| g \|_{0, \Gamma} + k \| \partial_{n_A} w \|_{0, \Gamma} + \| w \|_{1, \Gamma}
    \!\!\stackrel{\eqref{eq:2nd-order-10}}{\lesssim} \!\! k^{-1/2} \| f \|_{0,\Omega} + k \| g \|_{0, \Gamma} + k \| \partial_{n_A} w \|_{0, \Gamma}.
  \end{equation}
  In order to estimate $ \| \partial_{n_A} w \|_{0, \Gamma}$
we introduce an auxiliary matrix-valued function $\hat A \in C^\infty(\Omega, \operatorname{GL}({\mathbb C}^d))$ that satisfies 
\eqref{eq:HH-robin-assumption-coefficients-a} as well as $\hat A = A$ in a neighborhood of $\Gamma$. 
  We introduce an auxiliary function $\tilde{w} \in H^1(\Omega)$ as the solution of 
  \begin{equation}
\label{eq:problem_for_w_tilde}
    \begin{alignedat}{2}
      - \nabla \cdot (\hat A \tilde{w})  &= 0  \quad  &&\text{in } \Omega, \qquad 
               \tilde{w} &= w  \quad  &&\text{on } \Gamma.
    \end{alignedat}
  \end{equation}
By the smoothness of $\Gamma$ and $\hat A$, the map $w \mapsto \partial_{n_A} \tilde{w}$ maps $H^{1/2}(\Gamma) \rightarrow H^{-1/2}(\Gamma)$
and $H^{3/2}(\Gamma)\rightarrow H^{1/2}(\Gamma)$ so that by interpolation it maps $H^{1}(\Gamma) \rightarrow L^2(\Gamma)$
with the {\sl a priori} estimate 
\begin{equation}
\label{eq:2nd-order-20}
\|\partial_{n_A} \tilde{w}\|_{0,\Gamma} \lesssim \|w\|_{1,\Gamma}. 
\end{equation}
By interior regularity, we have additionally that $\tilde{w} \in C^\infty(\Omega)$ and for each open 
$\Omega' \subset \subset \Omega$ there is $C_{\Omega'} > 0$ with 
\begin{equation}
\label{eq:2nd-order-30}
\|\tilde{w}\|_{2,\Omega'} \leq C_{\Omega'} \|\tilde{w}\|_{1,\Omega} \lesssim \|w\|_{1/2,\Gamma}. 
\end{equation}
Let $\chi \in C^\infty(\overline{\Omega})$ with $\chi \equiv 1$ in a neighborhood of $\Gamma$ and such that 
$A|_{\operatorname{supp} \chi} = \hat A|_{\operatorname{supp} \chi}$. 
  The function $w - \chi \tilde{w} \in H^1_0(\Omega)$ solves
  \begin{equation}
\label{eq:2nd-order-50}
    \begin{alignedat}{2}
      - \nabla \cdot( A \nabla (w - \chi \tilde{w})) &= f -k^2  w  + z \quad  &&\text{in } \Omega,\\
              w - \tilde{w} &= 0  \quad  &&\text{on } \Gamma
    \end{alignedat}
  \end{equation}
for the function $z := \nabla \cdot (A (\chi \tilde{w}))$. Since $A  = \hat A$ near $\Gamma$ and 
therefore $-\nabla \cdot (A \nabla(\chi \tilde{w})) = 0$ near $\Gamma$, we get in view of \eqref{eq:2nd-order-30} 
that $z \in L^2(\Omega)$ with 
\begin{equation} 
\label{eq:2nd-order-40}
\|z\|_{0,\Omega} \lesssim \|w\|_{1/2,\Gamma}. 
\end{equation} 
As discussed in the proof of Lemma~\ref{lemma:HH-robin-AP}, elliptic regularity implies $w - \chi \tilde{w} \in H^2(\Omega\setminus\interface)$ with 
\begin{equation}
\label{eq:2nd-order-60}
\|w-  \chi \tilde{w}\|_{2,\Omega\setminus\interface} \lesssim \|f\|_{0,\Omega} + k^2 \|w\|_{0,\Omega} + \|z\|_{0,\Omega}
\stackrel{\eqref{eq:2nd-order-10},\eqref{eq:2nd-order-40}}{\lesssim} \|f\|_{0,\Omega} + k^{1/2} \|g\|_{0,\Gamma}. 
\end{equation}
Since $\chi \equiv 1$ near $\Gamma$, we conclude 
\begin{align*}
\|\partial_{n_A} w\|_{0,\Gamma}  & \leq 
\|\partial_{n_A} \tilde{w}\|_{0,\Gamma} + \|\partial_{n_A}(w - \chi\tilde{w})\|_{0,\Gamma}  
\lesssim 
\|\partial_{n_A} \tilde{w}\|_{0,\Gamma} + C \|w - \chi\tilde{w}\|_{2,\Omega\setminus\interface}  \\
&
\stackrel{\eqref{eq:2nd-order-20},\eqref{eq:2nd-order-10},\eqref{eq:2nd-order-60}}{\lesssim}\|f\|_{0,\Omega} + k^{1/2} \|g\|_{0,\Gamma}. 
\end{align*}
Inserting this estimate in \eqref{eq:2nd-order-15} yields 
\begin{equation*}
\|\alpha \Delta_\Gamma w\|_{2,\Gamma} \lesssim \|f\|_{0,\Omega}  + k^{1/2} \|g\|_{0,\Gamma}. 
\end{equation*}
Noting that \eqref{eq:2nd-order-5} is a Neumann problem for $w$, we get from elliptic regularity 
for Neumann problems the desired statement \ref{AP2}. 

We remark in passing that the above arguments can be boostrapped to show \ref{AP2}
for $s_{\max} > 0$. We refer to \cite[Rem.~{6.5.7}]{bernkopf21} and in particular to
the bootstrapping argument for the analytic regularity assertion \ref{WP3} in
Appendix~B of \cite{preprint_bernkopf_chaumontfrelet_melenk_2022a} for more details. 

\emph{Proof of \ref{item:lemma:HH-2nd-order-ABC-AP-iii}:} By inspection. 
\end{proof}
\begin{lemma}[\WPs for $2^{nd}$ order ABC]
\label{lemma:HH-2nd-order-ABC-WP}
Let Assumption~\ref{assumption:smoothness_domain_boundary_and_interface} be valid. 
Let the coefficients $A$, $n$ satisfy \eqref{eq:HH-robin-assumption-coefficients}
and $A$ additionally \eqref{eq:HH-robin-assumption-coefficients-aa}. 
For $\Lkm$, $\Lkp$ of \eqref{eq:Lkm}, \eqref{eq:Lkp} with 
$\TkGammaM$ given by \eqref{eq:form-of-2nd-order-ABC}--\eqref{eq:alpha_dependence_in_second_order_abcs}
and $\TkGammaP = \alpha \Delta_\Gamma$ the following is true for $t = 1$
if \ref{WP2} holds:
\begin{enumerate}[nosep, start=1, label={(\roman*)},leftmargin=*]
\item 
\label{item:lemma:HH-2nd-order-ABC-WP-i}
\ref{WP3} holds with $\Can = O(1)$ (uniformly in $k$). 
\item 
\label{item:lemma:HH-2nd-order-ABC-WP-ii}
\ref{WP4} holds with $\ROM u = k^2 (n^2+1) u$, $\AOM  = 0$ and $\RGM = \alpha \Delta_\Gamma u + \beta u$, $\AGM = 0$. 
In particular, for the constants in Assumption~\ref{assumption:everything_is_polynomial_well_posed} we have 
$\CcontM  = O(1)$ and $C^-_{A,k} = C^\Delta_{A, k} = 0$. 
\item 
\label{item:lemma:HH-2nd-order-ABC-WP-iii}
\ref{WP5} holds. 
\end{enumerate}
\end{lemma}
\begin{proof}
\emph{Proof of \ref{item:lemma:HH-2nd-order-ABC-WP-i}:}
See \cite[App.~{B}]{preprint_bernkopf_chaumontfrelet_melenk_2022a} for the proof. 
To give some details, by local changes of variables and the invariance of
the analyticity classes under analytic changes of variables (see Lemma~\ref{lemma:lemma2.6}),
the analysis is reduced to balls or half-balls. For the regularity
of $\Skm(f,g)$ with $f \in \AnaClass{M_f}{\gamma_f}{\partition}$, $g \in \AnaClassbdy{M_g}{\gamma_g}{T}{\Gamma}$ one obtains 
from \cite[Thm.~B.5]{preprint_bernkopf_chaumontfrelet_melenk_2022a} (and the analytic changes of variables) for a tubular neighborhood $T$ of $\Omega$ for 
\begin{align*}
\|\nabla^{p+2} u\|_{L^2(T \cap \Omega)} \lesssim \max(p,k)^{p+2} \gamma^p \left[ k^{-2} M_f + k^{-1} M_g + k^{-1} \|u\|_{1,t,k}
\right] 
\quad \forall p \in {\mathbb N}_0, 
\end{align*}
where $\gamma$ depends on $\gamma_f$, $\gamma_g$, but is independent of $k$. Together with corresponding estimates
in the interior of $\Omega$, this shows \ref{WP3} with $\Can = O(1)$. 

\emph{Proof of \ref{item:lemma:HH-2nd-order-ABC-WP-ii}, \ref{item:lemma:HH-2nd-order-ABC-WP-iii}:} By inspection. 
\end{proof}
\begin{theorem}[$2^{nd}$ order ABC]
\label{thm:2nd-order-ABC}
Let $t = 1$. 
Let Assumption~\ref{assumption:smoothness_domain_boundary_and_interface} 
and \eqref{eq:HH-robin-assumption-coefficients} be valid. For problem \eqref{eq:Lkm} with $\TkGammaM$ given by 
% = \DtN$ given by 
\eqref{eq:form-of-2nd-order-ABC}--\eqref{eq:alpha_dependence_in_second_order_abcs}
assume the polynomial well-posedness of the solution operator $\Skm$ of Assumption~\ref{assumption:everything_is_polynomial_well_posed}. 
In the discretization setting of Assumption~\ref{assumption:quasi_uniform_regular_meshes}, for each $c_2 > 0$ there are constants $c_1$, $C> 0$
independent of $h$, $p$, and $k$ such that under the scale-resolution condition \eqref{eq:scale-resolution} the quasi-optimality
result \eqref{eq:quasi-optimality} holds. 
\end{theorem}
\begin{proof}
Combine Lemmas~\ref{lemma:HH-2nd-order-ABC-AP} and \ref{lemma:HH-2nd-order-ABC-WP} with the stability Assumption~\ref{assumption:everything_is_polynomial_well_posed}. 
\end{proof}

\appendix

% !TEX root = main.tex
\section{Dirichlet-to-Neumann maps via Boundary Integral Operators}\label{section:dtn_via_bios}

The main goal of the present section is prove
Lemma~\ref{lemma:properties_dtn_operators}\ref{item:lemma:properties_dtn_operators_splitting_finite_ana_analytic_part}. 
To that end, we rewrite the Dirichlet-to-Neumann operators
$\mathrm{DtN}_k$ and $\mathrm{DtN}_0$
in terms of boundary integral operators.

\subsection{Preliminaries}
Let $\Omega \subset \mathbb{R}^d$,
$d = 2$, $3$, be a bounded Lipschitz domain
with analytic boundary $\Gamma \coloneqq \partial \Omega$.
We denote by $\Omega^+$ the exterior domain,
i.e., $\Omega^+ \coloneqq \mathbb{R}^d \setminus \overline{\Omega}$.
Throughout this section we assume $\Omega^+$ to be non-trapping,
see~\cite[Def.~{1.1}]{baskin-spence-wunsch16}.
Furthermore, we assume that the open ball $B_R$ of radius $R$
around the origin contains $\overline{\Omega}$, i.e., $\overline{\Omega} \subset B_R$.
We set $\Omega_R \coloneqq (\Omega \cup \Omega^+) \cap B_R = B_R \setminus \Gamma$.
Following standard notation we introduce the interior and exterior trace operators
$\gamma_0^{int}$, $\gamma_0^{ext}$ by restricting to $\Gamma$ and the 
$\gamma_1^{int}$ and $\gamma_1^{ext}$ by setting, for sufficiently smooth functions $v$,
we have $\gamma_1^{int} v = \normal \cdot \gamma_0^{int} (\nabla v)$ and 
$\gamma_1^{ext} v = \normal \cdot \gamma_0^{ext} (\nabla v)$, where $\normal$ is the outer normal
vector on $\Gamma$. Furthermore, we denote by $V_k$, $K_k$, $K_k^\prime$ and $D_k$
the single layer, double layer, adjoint double layer
and hypersingular boundary integral operators, see \cite[Sec.~{6.9} and~{7.9}]{steinbach08}.
The single layer and double layer potentials are denoted $\widetilde V_k$ and $\widetilde K_k$. 
Finally, given a coupling parameter
$\eta \in \mathbb{R}\setminus \{ 0 \}$
we introduce the combined field operator $A^\prime_{k, \eta}$ by
\begin{equation*}
  A^\prime_{k, \eta} \coloneqq \frac{1}{2} + K^\prime_k + i \eta V_k.
\end{equation*}
We remind the reader of the exterior Calder{\'o}n identities
\begin{equation*}
  \begin{pmatrix}
    \gamma_0^{ext} u \\
    \gamma_1^{ext} u
  \end{pmatrix}
  =
  \begin{pmatrix}
    \frac{1}{2} + K_k & -V_k \\
    -D_k & \frac{1}{2} - K_k^\prime
  \end{pmatrix}
  \cdot
  \begin{pmatrix}
    \gamma_0^{ext} u \\
    \gamma_1^{ext} u
  \end{pmatrix},
\end{equation*}
%where, for a sufficiently smooth function $v$ defined on $\Omega^+$ we denote by $\gamma_1^{ext} v:= \normal \cdot \nabla v$ the exterior normal derivative
%on $\Gamma$ with $\normal$ pointing into $\Omega^+$. Correspondingly, for $v$ defined on $\Omega$, we introduce 
%$\gamma_1^{int} v:= \normal \cdot \nabla v$ as the interior normal derivative. 
Given Dirichlet data $u$, the Dirichlet-to-Neumann
operator $\DtN$ can be expressed for any $k \geq 0$ by a complex
linear combination of the two equations in the Calder{\'o}n identity:
For any $\eta \in \mathbb{R} \setminus \{ 0 \}$ we have
\begin{equation}\label{eq:combined_field_equation}
A^\prime_{k,\eta} (\mathrm{DtN}_k u) = 
  \left( \frac{1}{2} + K^\prime_k + i \eta V_k \right)
  \mathrm{DtN}_k u
  =
  \left(-D_k + i \eta (-\frac{1}{2} + K_k) \right) u.
\end{equation}
%or using the combined field operator $A^\prime_{k, \eta}$ we have
%\begin{equation}\label{eq:combined_field_equation_A_keta}
%  A^\prime_{k, \eta} \mathrm{DtN}_k u
%  =
%  \left(-D_k + i \eta (-\frac{1}{2} + K_k) \right) u.
%\end{equation}
Our analysis relies on invertibility of the
combined field operator $A^\prime_{k, \eta}$
as an operator mapping $H^s(\Gamma)$ into itself.
Wavenumber-explicit estimates for 
$\| (A^\prime_{k, \eta})^{-1} \|_{L^2(\Gamma) \leftarrow L^2(\Gamma)}$
are discussed in~\cite[Sec.~{1.4}]{baskin-spence-wunsch16}.  
For non-trapping $\Omega^+ \subset \mathbb{R}^d$, $d = 2$, $3$ it is known that
%\begin{equation}
%\label{eq:combined_field_operator_norm_suboptimal}
$\displaystyle 
  \| (A^\prime_{k, \eta})^{-1} \|_{L^2(\Gamma) \leftarrow L^2(\Gamma)}
  \lesssim
  k^{5/4} \left( 1 + \frac{k^{3/4}}{|\eta|} \right)
$
%\end{equation}
for all $k \geq k_0$ and $\eta \in \mathbb{R}\setminus \{ 0 \}$,
see~\cite[Thm~{1.11}]{spence14}.
This bound can be sharpened assuming $|\eta| \sim k$.
In fact, for non-trapping $\Omega^+ \subset \mathbb{R}^d$, $d = 2, 3$
and $|\eta| \sim k$ there holds
\begin{equation}\label{eq:combined_field_operator_norm_optimal}
  \| (A^\prime_{k, \eta})^{-1} \|_{L^2(\Gamma) \leftarrow L^2(\Gamma)}
  \lesssim
  1
\end{equation}
for all $k \geq k_0$, see~\cite[Thm.~{1.13}]{baskin-spence-wunsch16}.
In Proposition~\ref{proposition:mapping_properties_bios}, we collect 
mapping properties of some boundary integral operators: 
\begin{proposition}\label{proposition:mapping_properties_bios}
  Let $\Omega^+$ be non-trapping, $\Gamma$ be analytic and $\eta \in \mathbb{R} \setminus \{ 0 \}$ fixed.
  If $d = 2$, assume additionally $\mathrm{diam} \, \Omega < 1$.
  Then
  \begin{enumerate}[nosep, start=1, label={(\roman*)},leftmargin=*]
    \item
    \label{proposition:mapping_properties_bios_comined_field_operator_laplacian_without_prime}
    $A_{0,-\eta}
    :=
    \frac{1}{2} + K_0 - i \eta V_0
    \colon H^s(\Gamma) \rightarrow H^s(\Gamma)$
    is boundedly invertible for $s \geq 0$.
    \item
    \label{proposition:mapping_properties_bios_comined_field_operator_laplacian}
    $A^\prime_{0,\eta}
    =
    \frac{1}{2} + K^\prime_0 + i \eta V_0
    \colon H^s(\Gamma) \rightarrow H^s(\Gamma)$
    is boundedly invertible for $s \geq -1$.
    \item
    \label{proposition:mapping_properties_bios_comined_field_operator_helmholtz}
    For $k > 0$ the combined field operator
    $A^\prime_{k,\eta} = \frac{1}{2} + K^\prime_k + i \eta V_k
    \colon H^s(\Gamma) \rightarrow H^s(\Gamma)$
    is boundedly invertible for $s \geq -1$.
  %\end{enumerate}
\item 
\label{item:proposition:mapping_properties_bios-iv}
  For $k \geq 0$ and $s \ge -1/2$ the following operators are bounded: 
  \begin{equation}\label{proposition:mapping_properties_bios_helmholtz_standard_result}
    \begin{alignedat}{2}
      &V_k         \colon H^{-1/2 + s}(\Gamma) \rightarrow H^{1/2 + s}(\Gamma), \quad \quad \quad
      &&K_k        \colon H^{1/2 + s}(\Gamma) \rightarrow H^{1/2 + s}(\Gamma),   \\
      &K_k^\prime  \colon H^{-1/2 + s}(\Gamma) \rightarrow H^{-1/2 + s}(\Gamma), \quad \quad \quad
      &&D_k        \colon H^{1/2 + s}(\Gamma) \rightarrow H^{-1/2 + s}(\Gamma)
    \end{alignedat}
  \end{equation}
  \item 
\label{item:proposition:mapping_properties_bios-v}
For $k \geq k_0 > 0$ one can decompose 
  \begin{equation}\label{proposition:mapping_properties_bios_helmholtz_splitting}
    \begin{alignedat}{2}
      &V_k - V_k                = S_{V}        + \gamma_0^{int} \widetilde{A}_{V}, \quad \quad
      &&K_k - K_0               = S_{K}        + \gamma_0^{int} \widetilde{A}_{K}, \\
      &K_k^\prime - K_0^\prime  = S_{K^\prime} + \gamma_1^{int} \widetilde{A}_{V}, \quad \quad
      &&D_k - D_0               = S_{D}        - \gamma_1^{int} \widetilde{A}_{K}
    \end{alignedat}
  \end{equation}
  with linear maps
  $\widetilde{A}_{V} \colon H^{-3/2}(\Gamma) \rightarrow C^{\infty}(\overline{\Omega})$
  and $\widetilde{A}_{K} \colon H^{-1/2}(\Gamma) \rightarrow C^{\infty}(\overline{\Omega})$
  and bounded linear operators $S_{V}$, $S_{K}$, $S_{K^\prime}$, and $S_{D}$ having the following mapping properties for $s \geq -1$: 
  \begin{equation}
   \label{proposition:mapping_properties_bios_helmholtz_splitting_mapping_properties}
    \begin{alignedat}{1}
      \| S_{V} u \|_{-1/2 + s', \Gamma}                &\leq C_{s, s^\prime} k^{-(1+s-s^\prime)} \| u \|_{-1/2 + s, \Gamma}, \qquad 1/2 \leq s^\prime \leq s+3, \\
      \| S_{K} u \|_{-1/2 + s', \Gamma}                 &\leq C_{s, s^\prime} k^{-(1+s-s^\prime)} \| u \|_{1/2 + s, \Gamma}, \qquad 1/2 \leq s^\prime \leq s+3, \\
      \| S_{K^\prime} u \|_{-3/2 + s', \Gamma}         &\leq C_{s, s^\prime} k^{-(1+s-s^\prime)} \| u \|_{-1/2 + s, \Gamma}, \qquad 3/2 \leq s^\prime \leq s+3, \\
      \| S_{D} u \|_{-3/2 + s', \Gamma}                 &\leq C_{s, s^\prime} k^{-(1+s-s^\prime)} \| u \|_{1/2 + s, \Gamma}, \qquad 3/2 \leq s^\prime \leq s+3. 
    \end{alignedat}
  \end{equation}
  The operators $\widetilde{A}_{V}$ and $\widetilde{A}_{K}$ have the mapping property
  \begin{align}
\label{eq:mapping_prop_A_tilde_V}
    \widetilde{A}_{V} f
    \in
    \AnaClassmelenk{C_V k \| f \|_{-3/2, \Gamma}}{\gamma_V}{\Omega}
    \qquad \forall f \in H^{-3/2}(\Gamma), \\
\label{eq:mapping_prop_A_tilde_K}
    \widetilde{A}_{K} f
    \in
    \AnaClassmelenk{C_K k \| f \|_{-1/2, \Gamma}}{\gamma_K}{\Omega}
    \qquad \forall f \in H^{-1/2}(\Gamma),
  \end{align}
  with constants $C_V$, $\gamma_V$, $C_K$, $\gamma_K$ independent of $k \geq k_0$.
  For $t \geq 0$ the following mapping properties hold true: 
  \begin{equation}\label{proposition:mapping_properties_bios_helmholtz_splitting_mapping_properties_simple}
    \begin{alignedat}{2}
      &\| S_{V} u \|_{t, \Gamma}                \leq C_{t} k^{-1}  \| u \|_{t, \Gamma}, \quad \quad \quad
      &&\| S_{K} u \|_{t, \Gamma}                \leq C_{t}         \| u \|_{t, \Gamma},  \\
      &\| S_{K^\prime} u \|_{t, \Gamma}         \leq C_{t}         \| u \|_{t, \Gamma}, \quad \quad \quad
      &&\| S_{D} u \|_{t, \Gamma}                \leq C_{t}  k      \| u \|_{t, \Gamma}. 
    \end{alignedat}
  \end{equation}
\end{enumerate}
\end{proposition}
\begin{proof}
  For Item~\ref{proposition:mapping_properties_bios_comined_field_operator_laplacian_without_prime}
  see~\cite[Lem.~{3.5}(ii)]{melenk12}.
  For Item~\ref{proposition:mapping_properties_bios_comined_field_operator_laplacian} in the case $s \geq 0$
  see~\cite[Lemma~{3.5}(iv)]{melenk12}.
  We turn to the case $s \in [-1, 0]$.
  Note that the adjoint of $A_{0,-\eta}$ is precisely the operator $A_{0,\eta}^\prime$.
  Furthermore, by Item~\ref{proposition:mapping_properties_bios_comined_field_operator_laplacian_without_prime}
  the operator $A_{0,-\eta} \colon H^t(\Gamma) \rightarrow H^t(\Gamma)$ is boundedly invertible in particular for $t \in [0,1]$.
  Hence, due to the adjoint of $A_{0,-\eta}$ being $A_{0,\eta}^\prime$,
  we find that $A_{0,\eta}^\prime  \colon H^s(\Gamma) \rightarrow H^s(\Gamma)$ is also boundedly invertible for $s \in [-1, 0]$.
  For Item~\ref{proposition:mapping_properties_bios_comined_field_operator_helmholtz}
  see~\cite[Thm.~{2.27}]{chandler-wilde-graham-langdon-spence12}
  in the case $s \in [-1, 0]$ as well as~\cite[Sec.~{6.1}]{baskin-spence-wunsch16}.
  Consequently, by~\cite[Lem.~{2.14}]{melenk12} invertibility holds for any $s \geq 0$.
  The mapping properties \eqref{proposition:mapping_properties_bios_helmholtz_standard_result}
in Item \ref{item:proposition:mapping_properties_bios-iv} are standard, see, e.g., \cite[Rem.~{3.1.18}(c)]{sauter-schwab11}. 
For 
  \eqref{proposition:mapping_properties_bios_helmholtz_splitting}
  and
  \eqref{proposition:mapping_properties_bios_helmholtz_splitting_mapping_properties}
in Item~\ref{item:proposition:mapping_properties_bios-v}  
  see~\cite[Lem.~{A.1}]{mascotto-melenk-perugia-rieder20}
  for the cases $1/2 < s^\prime$ and $3/2 < s^\prime$. 
  The limiting cases $s' = 1/2$ and $s' = 3/2$ follow by inspection of the proof,
  a multiplicative trace estimate, and the estimates for 
the potentials ${\widetilde {\mathcal S}}_{\mathcal V}$, ${\widetilde {\mathcal S}}_{\mathcal K}$ appearing in 
  \cite[Lem.~{A.1}]{mascotto-melenk-perugia-rieder20}. 
  \eqref{eq:mapping_prop_A_tilde_K} is likewise shown in 
  \cite[Lem.~{A.1}]{mascotto-melenk-perugia-rieder20}.
  \eqref{proposition:mapping_properties_bios_helmholtz_splitting_mapping_properties_simple}
  is just a simplification of \eqref{proposition:mapping_properties_bios_helmholtz_splitting_mapping_properties}.
\end{proof}

\subsection{Decomposition of the Dirichlet-to-Neumann map}

Before proceeding with the proof of 
Lemma~\ref{lemma:properties_dtn_operators}\ref{item:lemma:properties_dtn_operators_splitting_finite_ana_analytic_part}
let us introduce the jumps of the trace operators:
\begin{equation}
\label{eq:jump}
  \llbracket u \rrbracket \coloneqq \gamma_0^{ext} u - \gamma_0^{int} u,
  \qquad
  \llbracket \partial_n u \rrbracket \coloneqq \gamma_1^{ext} u - \gamma_1^{int} u.
\end{equation}
For linear operators $\widetilde{A}$ mapping into spaces of piecewise defined functions
we define the operator $\llbracket \widetilde{A} \rrbracket$ and $\llbracket \partial_n \widetilde{A} \rrbracket$
analogously, e.g., $\llbracket \widetilde{A} \rrbracket u \coloneqq \llbracket \widetilde{A} u \rrbracket  $.

We now collect further technical results from \cite{melenk12}.
We closely follow the notation and results of \cite{melenk12}.
As in \cite{melenk12} we assume
\begin{equation}\label{eq:eta-comb-field-estimate}
  C_\eta^{-1} k \leq |\eta| \leq C_\eta k
\end{equation}
for some $C_\eta > 0$ independent of $k$.

In the following Proposition~\ref{prop:lemma_melenk_6.3} we extend the results of
\cite[Lem.~{6.3}]{melenk12} to a wider range of Sobolev spaces.

\begin{proposition}[{\cite[Lemma~{6.3}]{melenk12}}]\label{prop:lemma_melenk_6.3}
  Let $\Omega \subset \mathbb{R}^d$ be a bounded Lipschitz domain with an analytic boundary $\Gamma$.
  Let $q \in (0,1)$.
  Then one can construct operators $L^{neg}_{\Gamma,q}$, $H^{neg}_{\Gamma,q}$ on $H^{-1}(\Gamma)$ with the following properties:
  \begin{enumerate}[nosep, start=1, label={(\roman*)},leftmargin=*]
    \item
    \label{item:lemma63_1}
    $L^{neg}_{\Gamma,q} f + H^{neg}_{\Gamma,q} f = f$ for all $f \in H^{-1}(\Gamma)$.
    \item
    \label{item:lemma63_2}
    For $-1 \leq s^\prime \leq s$ there holds $\|H^{neg}_{\Gamma,q} f\|_{s^\prime, \Gamma} \leq C_{s, s^\prime} (q/k)^{s-s^\prime} \| f \|_{s, \Gamma}$.
    \item
    \label{item:lemma63_3}
    $L^{neg}_{\Gamma,q} f$ is the restriction to $\Gamma$ of a function that is analytic in a tubular neighborhood $T$ of $\Gamma$ and satisfies
    \begin{align*}
      \| \nabla^n L^{neg}_{\Gamma,q} f \|_{0, T} &\leq C_q k^{d/2} \gamma_q^n \max \{ k,n \}^n \|f\|_{-1/2, \Gamma} \qquad \forall n \in \mathbb{N}_0, \\
      \| \nabla^n L^{neg}_{\Gamma,q} f \|_{0, T} &\leq C_q k^{d/2+1} \gamma_q^n \max \{ k,n \}^n \|f\|_{-1, \Gamma} \qquad \forall n \in \mathbb{N}_0.
    \end{align*}
  \end{enumerate}
  Here, $C_{s, s^\prime}$ is independent of $q$, $k$;
  the constants $C_q$, $\gamma_q > 0$ are independent of $k$.
\end{proposition}
\begin{proof}
  For Items~\ref{item:lemma63_1},~\ref{item:lemma63_2} in the case $-1 \leq s^\prime \leq s \leq 1$ and Item~\ref{item:lemma63_3} 
see \cite[Lem.~{6.3} and Rem.~{6.4}]{melenk12}.
  The crucial extension is the estimate stated in Item~\ref{item:lemma63_2} in the case $-1 \leq s^\prime \leq s$ for $s \geq 1$.
  In the proof of \cite[Lem.~{6.3}]{melenk12} the operators $\HNegGamma$ and $\LNegGamma$ are explicitly constructed.
  We collect the important ingredients of the proof of \cite[Lem.~{6.3}]{melenk12} in the following.
  On the compact manifold $\Gamma$ consider the eigenvalue problem for the Laplace-Beltrami operator
  \begin{equation}\label{eq:laplace_beltrami_eigenvalue_problem}
    -\Delta_\Gamma \varphi = \lambda^2 \varphi \quad \text{on } \Gamma.
  \end{equation}
  There are countably many eigenfunctions $\varphi_m$, $m \in \mathbb{N}_0$,
  with corresponding eigenvalues $\lambda_m^2 \geq 0$,
  which we assume to be sorted in ascending order.
  Without loss of generality, these eigenfunctions are normalized in $L^2(\Gamma)$.
  The functions $(\varphi_m)_{m \in \mathbb{N}_0}$ may be assumed to be an orthonormal basis of $L^2(\Gamma)$
  and an orthogonal basis of $H^1(\Gamma)$.
  With $u_m \coloneqq (u, \varphi_m)$ we have
  \begin{equation}\label{eq:hshdots_tmp0}
    \| u \|_{0, \Gamma}^2 = \sum_{m = 0}^\infty |u_m|^2 \quad \text{ and } \quad
    \| u \|_{1, \Gamma}^2 = \sum_{m = 0}^\infty (1+\lambda_m^2) |u_m|^2.
  \end{equation}
  For $s \in \mathbb{R}$ we introduce the sequence space $h^s$ by
  \begin{equation*}
    h^s \coloneqq \left\{ (u_m)_{m \in \mathbb{N}} \colon \sum_{m = 0}^\infty (1+\lambda_m^2)^s |u_m|^2 < \infty \right\}.
  \end{equation*}
  The mapping $\iota \colon u \mapsto (\, \langle u, \varphi_m\rangle  \, )_{m \in \mathbb{N}_0}$ then provides an isomorphism between
  the Sobolev space $H^s(\Gamma)$ and the sequence space $h^s$ for $s \in [-1,1]$, with corresponding norm equivalence,
  see \cite[Lem.~{C.3}]{melenk12}.
  However, as we will see below $\iota$ is in fact an isomorphism for all $s \geq -1$.
  Inspection of the proof of \cite[Lem.~{6.3}]{melenk12},
  in particular the proof of the estimate for $\HNegGamma$,
  reveals that
  \begin{equation*}
    \|H^{neg}_{\Gamma,q} f\|_{s^\prime, \Gamma} \leq C_{s, s^\prime} (q/k)^{s-s^\prime} \| f \|_{s, \Gamma}
  \end{equation*}
  holds for all $-1 \leq s^\prime \leq s$,
  for which $\iota \colon H^s(\Gamma) \rightarrow h^s$
  and $\iota \colon H^{s^\prime}(\Gamma) \rightarrow h^{s^\prime}$
  are isomorphisms.
  Hence, the proof is complete once we establish that
  $\iota \colon H^s(\Gamma) \rightarrow h^s$ is an isomorphism for all $s > 1$.
  We show the case $s=2$.
  \\
  \textbf{The inclusion $h^2 \xhookrightarrow{} H^2(\Gamma)$}:
  Let $u = \sum_{m = 0}^\infty u_m \varphi_m$ be such that $\sum_{m = 0}^\infty (1+\lambda_m^2)^2 |u_m|^2 < \infty$.
  Let $u^N = \sum_{m = 0}^N u_m \varphi_m$.
  By the above construction, $u^N \rightarrow u$ in $H^1(\Gamma)$ and
  $\| u \|_{1,\Gamma} = \| (u_m)_{m\in \mathbb{N}} \|_{h^1}$.
  Furthermore, we have
  \begin{align*}
    \lVert \Delta_\Gamma (u^N -  u^{M-1}) \rVert_{0, \Gamma}^2
    %&= \left\lVert \Delta_\Gamma \sum_{m = 1}^N u_m \varphi_m - \Delta_\Gamma \sum_{m = 1}^{M-1} u_m \varphi_m \right\rVert_{0, \Gamma}^2
    \!=\! \left\lVert \sum_{m = M}^N u_m \Delta_\Gamma \varphi_m \right\rVert_{0, \Gamma}^2 
%\\ &
\!=\! \left\lVert \sum_{m = M}^N u_m \lambda_m^2 \varphi_m \right\rVert_{0, \Gamma}^2
    \!=\!\! \sum_{m = M}^N |u_m|^2 \lambda_m^4,% \rightarrow 0,
  \end{align*}
which converges to zero as $N$, $M \rightarrow \infty$. 
  Here we used~\eqref{eq:laplace_beltrami_eigenvalue_problem},
  the fact that the eigenfunctions are an orthonormal basis of $L^2(\Gamma)$,
  and the assumed convergence $\sum_{m = 0}^\infty (1+\lambda_m^2)^2 |u_m|^2 < \infty$.
  Therefore, $(u^N)_{N \in {\mathbb N}_0}$ is a Cauchy sequence in $H^1(\Gamma, \Delta_\Gamma) = \{u \in H^1(\Gamma) \colon \Delta_\Gamma u \in L^2(\Gamma)\}$,
  endowed with the graph norm.
  Consequently $(u^N)_{N \in {\mathbb N}_0}$ converges in $H^1(\Gamma, \Delta_\Gamma)$.
  Since $\Delta_\Gamma: H^1(\Gamma, \Delta_\Gamma) \to L^2(\Gamma)$ is continuous,
  we conclude $\Delta_\Gamma u = \sum_{m \in \mathbb{N}_0} u_m \Delta_\Gamma \varphi_m = - \sum_{m \in \mathbb{N}_0} u_m \lambda_m^2 \varphi_m$.
  Finally, by elliptic regularity we can now estimate
  \begin{equation}\label{eq:hshdots_tmp1}
    |u|_{2,\Gamma}^2 \lesssim \lVert \Delta_\Gamma u \rVert_{0,\Gamma}^2 = \sum_{m \in \mathbb{N}_0} |f_m|^2 \lambda_m^4 = |(u_m)_{m \in \mathbb{N}_0}|^2_{h^2},
  \end{equation}
and 
  %\begin{equation*}
$\displaystyle 
    \|u\|_{2,\Gamma}^2 \lesssim \|(u_m)_{m \in \mathbb{N}_0}\|^2_{h^2}
  %\end{equation*}
$
  follows by~\eqref{eq:hshdots_tmp1} together with~\eqref{eq:hshdots_tmp0}.
\newline 
  \textbf{The inclusion $H^2(\Gamma) \xhookrightarrow{} h^2$}:
  Let $u \in H^2(\Gamma)$ be given with the representation $u = \sum_{m = 0}^{\infty} u_m \varphi_m$,
  where the sum converges in $H^1(\Gamma)$.
  Since $u \in H^2(\Gamma)$ we have $-\Delta_\Gamma u \eqqcolon f \in L^2(\Gamma)$.
  In the following we express the coefficient $u_m$ in terms of $f_m$.
  Note that
  \begin{equation*}
    \lambda_m^2 u_m = \lambda_m^2 \langle u, \varphi_m\rangle  = \langle  \nabla_\Gamma u, \nabla_\Gamma  \varphi_m\rangle  = \langle  f, \varphi_m\rangle = f_m.
  \end{equation*}
  Hence, we have $\lambda_m^2 u_m = f_m$ and consequently
  \begin{equation*}
    \sum_{m = 0}^\infty \lambda_m^4 |u_m|^2 =  \sum_{m = 0}^\infty |f_m|^2 < \infty.
  \end{equation*}
  Finally, using~(\ref{eq:hshdots_tmp1}) as well as the fact that $\| u \|_{1,\Gamma} = \| (u_m)_{m\in \mathbb{N}_0} \|_{h^1}$,  we find
  \begin{equation*}
    \|(u_m)_{m \in \mathbb{N}_0}\|^2_{h^2}
    = \|(u_m)_{m \in \mathbb{N}_0}\|^2_{h^1} + |(u_m)_{m \in \mathbb{N}_0}|^2_{h^2}
    = \| u \|_{1,\Gamma}^2 + \| \Delta_\Gamma u \|_{0,\Gamma}^2 \lesssim \|u\|_{2,\Gamma}^2.
  \end{equation*}
  This concludes the proof for $s=2$.
  Interpolation between $s=1$ and $s=2$ yields the result for $s \in (1,2)$,
  see \cite[Lem.~{C.3}]{melenk12}.
  Inductively one proceeds for the space $H^{2n}(\Gamma)$ by similar arguments.
  Instead of $\Delta_\Gamma$ one performs the same arguments for $\Delta_\Gamma^n$.
\end{proof}

%{\tiny \MB{We can probably remove the remark below as well}
%\begin{remark}
%  A natural question arising from the proof of Proposition~\ref{prop:lemma_melenk_6.3}
%  is whether or not a similar construction allows for high and low pass filters in the volume $\Omega$.
%  The volume filters in Proposition~\ref{proposition:high_low_pass_filter_Omega_pw}
%  only allow for estimates in negative Sobolev norms for $-1/2 \leq s^\prime$.
%  In fact similar arguments as in the proof of Proposition~\ref{prop:lemma_melenk_6.3}
%  allow to construct high and low pass filters via the eigenvalue problem
%  \begin{equation*}
%    \begin{alignedat}{2}
%      - \Delta \varphi  &= \lambda^2 \varphi  \quad   &&\text{in } \Omega ,\\
%      \partial_n u                &= 0  \quad   &&\text{on } \Gamma.
%    \end{alignedat}
%  \end{equation*}
%  However, the corresponding high pass filter only allows for estimates in the range $-1 \leq s^\prime \leq s \leq 1$,
%  because of the additional boundary terms.
%\end{remark}
%}

In the following we will prove an extension of \cite[Thm.~{2.9}]{melenk12}
in Theorem~\ref{thm:extension_of_thm_melenk_2.9} below.
The proof of Theorem~\ref{thm:extension_of_thm_melenk_2.9} relies on a decomposition of
the volume potential $\widetilde{V}_k$,
which we present below for the readers' convenience.
\begin{proposition}[{\cite[Thm.~{5.3}]{melenk12}}]\label{prop:thm_melenk_5.3}
  Let $\Gamma$ be analytic and $q \in (0,1)$.
  Then
  \begin{equation*}
    \widetilde{V}_k = \widetilde{V}_0 + \widetilde{S}_{V,pw} + \widetilde{\mathcal{A}}_{V,pw},
  \end{equation*}
  where the linear operators $\widetilde{S}_{V,pw}$ and $\widetilde{\mathcal{A}}_{V,pw}$ satisfy the following for every $s \geq -1$:
  \begin{enumerate}[nosep, start=1, label={(\roman*)},leftmargin=*]
    \item
    \label{item:thm5.3_1}
    $\widetilde{S}_{V,pw} \colon H^{-1/2+s}(\Gamma) \rightarrow H^{3+s}(\Omega_R)$ with
    \begin{equation*}
      \| \widetilde{S}_{V,pw} \varphi \|_{s^\prime, \Omega_R} \leq C_{s^\prime, s} q^2 (q k^{-1})^{1+s-s^\prime} \| \varphi \|_{-1/2 + s, \Gamma}, \quad 0 \leq s^\prime \leq s+3.
    \end{equation*}
    Here, the constant $C_{s^\prime, s} > 0$ is independent of $q$ and $k\geq k_0$.
    \item
    \label{item:thm5.3_2}
    $\widetilde{\mathcal{A}}_{V,pw} \colon H^{-1/2+s}(\Gamma) \rightarrow H^2(B_R) \cap C^{\infty}(\Omega_R)$ with
    \begin{equation*}
      \| \nabla^n \widetilde{\mathcal{A}}_{V,pw} \varphi \|_{0, \Omega_R} \leq C_q \gamma_q \max\{n+1,k\}^{n+1} \| \varphi \|_{-3/2, \Gamma} \quad \forall n \in \mathbb{N}_0.
    \end{equation*}
    Here, $C_q$, $\gamma_q > 0$ are independent of $k\geq k_0$ but may depend on $q$.
  \end{enumerate}
\end{proposition}

\begin{theorem}[Extension of {\cite[Thm.~{2.9}]{melenk12}}]\label{thm:extension_of_thm_melenk_2.9}
  Let $\Gamma$ be analytic, and let $s \geq 0$.
  Fix $q \in (0,1)$.
  Then the operator $A^\prime_{k, \eta}$ can be written in the form
  \begin{equation*}
    A^\prime_{k, \eta} = A_{0,1}^\prime + R_{A^\prime} + k \llbracket \widetilde{\mathcal{A}}_{1} \rrbracket + \llbracket \partial_n \widetilde{\mathcal{A}}_{2} \rrbracket,
  \end{equation*}
  where the linear operator $R_{A^\prime}$ satisfies
  \begin{subequations}\label{eq:Rmapping}
  \begin{align}
    \| R_{A^\prime} u \|_{s+1, \Gamma} &\leq C k \| u \|_{s, \Gamma},   \label{eq:R_splus1s} \\
    \| R_{A^\prime} u \|_{s, \Gamma}   &\leq C k \| u \|_{s-1, \Gamma}, \label{eq:R_ssminus1}\\
    \| R_{A^\prime} u \|_{s, \Gamma} &\leq q \| u \|_{s, \Gamma},       \label{eq:R_ss}\\
    \| R_{A^\prime} u \|_{s-1, \Gamma} &\leq q \| u \|_{s-1, \Gamma},   \label{eq:R_sminus1sminus1}
  \end{align}
  \end{subequations}
  and the linear operators $\widetilde{\mathcal{A}}_{1}$, $\widetilde{\mathcal{A}}_{2} \colon H^{-1}(\Gamma) \rightarrow C^{\infty}(T)$ satisfy
  \begin{subequations}
  \begin{alignat}{2}
    \widetilde{\mathcal{A}}_{1} \varphi  &\in \AnaClassmelenk{C_q C_{1,\varphi}}{\gamma_q}{T},
    \quad \quad
    C_{1,\varphi} &&= k \| \varphi \|_{-3/2, \Gamma} + k^{d/2} \| \varphi \|_{-1, \Gamma}, \\
    \widetilde{\mathcal{A}}_{2} \varphi  &\in \AnaClassmelenk{C_q C_{2,\varphi}}{\gamma_q}{T},
    \quad \quad
    C_{2,\varphi} &&= k \| \varphi \|_{-3/2, \Gamma}.
  \end{alignat}
  \end{subequations}
  The constant $C$ and the tubular neighborhood $T$ of $\Gamma$ are independent of $k \geq k_0$ and $q$;
  the constants $C_q$, $\gamma_q > 0$ are independent of $k \geq k_0$ (but may depend of $q$).
\end{theorem}

\begin{proof}
  We perform a similar splitting as in the proof of~\cite[Thm.~{2.9}]{melenk12}.
  The starting point of our analysis is the decomposition
  \begin{equation*}
    A^\prime_{k, \eta}
    = \frac{1}{2} + K_0^\prime + \gamma_1^{int}(\widetilde{S}_{V,pw} + \widetilde{\mathcal{A}}_{V,pw}) + i \eta \gamma_0^{int}(\widetilde{V}_0 + \widetilde{S}_{V,pw} + \widetilde{\mathcal{A}}_{V,pw}),
  \end{equation*}
  with $\widetilde{S}_{V,pw}$ and $\widetilde{\mathcal{A}}_{V,pw}$ as in Proposition~\ref{prop:thm_melenk_5.3},
  see~\cite[Eq.~{(6.4)}]{melenk12}.
  Adding and subtracting $i V_0$ and noting $V_0 = \gamma_0^{int} \widetilde{V}_0$ we find
  \begin{align*}
%    \nonumber
    A^\prime_{k, \eta}
    &= \frac{1}{2} + K_0^\prime + i V_0 + \gamma_1^{int}(\widetilde{S}_{V,pw} + \widetilde{\mathcal{A}}_{V,pw}) + i (\eta - 1) \gamma_0^{int} \widetilde{V}_0 + i \eta \gamma_0^{int}(\widetilde{S}_{V,pw} + \widetilde{\mathcal{A}}_{V,pw}) \\
%\nonumber 
%\label{combined_field_decomp_starting_point}
    &= A_{0,1}^\prime + \gamma_1^{int}(\widetilde{S}_{V,pw} + \widetilde{\mathcal{A}}_{V,pw}) + i (\eta - 1) \gamma_0^{int} \widetilde{V}_0 + i \eta \gamma_0^{int}(\widetilde{S}_{V,pw} + \widetilde{\mathcal{A}}_{V,pw}). 
  \end{align*}
  Using the filters $H^{neg}_{\Gamma, q}$ and $L^{neg}_{\Gamma, q}$ of Proposition~\ref{prop:lemma_melenk_6.3} we define
  \begin{subequations}\label{splitting_operator_definition}
  	\begin{alignat}{1}
      R_{A^\prime}
      &= H^{neg}_{\Gamma, q}\left( \gamma_1^{int} \widetilde{S}_{V,pw} + i \eta \gamma_0^{int} \widetilde{S}_{V,pw} + i (\eta - 1) V_0 \right), \label{splitting_operator_definition-1}\\
      \nonumber
      \widetilde{\mathcal{A}}_1
      &= - k^{-1} \chi_\Omega \left( i \eta \widetilde{\mathcal{A}}_{V,pw} + L^{neg}_{\Gamma, q} \left( \gamma_1^{int} \widetilde{S}_{V,pw}  + i \eta \gamma_0^{int} \widetilde{S}_{V,pw} + i (\eta - 1) V_0 \right) \right), \\
      \nonumber
      \widetilde{\mathcal{A}}_2
      &= - \chi_\Omega \widetilde{\mathcal{A}}_{V,pw}.
  	\end{alignat}
  \end{subequations}
  The mapping properties of $\widetilde{\mathcal{A}}_1$ and $\widetilde{\mathcal{A}}_2$
  stay the same as in~\cite[Thm.~{2.9}]{melenk12}.
  We are left with the mapping properties of $R_{A^\prime}$.
  In the following the parameter $q$ appearing in
  Proposition~\ref{prop:lemma_melenk_6.3} and~\ref{prop:thm_melenk_5.3} is
  still at our disposal\footnote{Do not confuse this $q$ with the one appearing in the statement of the present theorem.}.
  We fix it at the end of the proof to ensure the estimates~\eqref{eq:R_ss} and~\eqref{eq:R_sminus1sminus1}.

  \textbf{Step 1}:
  We estimate the term $i (\eta - 1) \HNegGamma V_0$ in various norms.
  We heavily use the estimates for $\HNegGamma$ and $V_0$
  given in Proposition~\ref{prop:lemma_melenk_6.3} and~\eqref{proposition:mapping_properties_bios_helmholtz_splitting} in
  Proposition~\ref{proposition:mapping_properties_bios}. 
  First estimating $\eta$,
  then using the properties of $\HNegGamma$ in Proposition~\ref{prop:lemma_melenk_6.3}
  and finally the mapping properties of $V_0$ we find
  \begin{alignat*}{3}
    \| i (\eta - 1) \HNegGamma V_0 u \|_{s+1, \Gamma}
    &\leq C k \| \HNegGamma V_0 u\|_{s+1, \Gamma}
    &&\leq C k \| V_0 u \|_{s+1, \Gamma}
    &&\leq C k \| u \|_{s, \Gamma}, \\
    \| i (\eta - 1) \HNegGamma V_0 u \|_{s, \Gamma}
    &\leq C k \| \HNegGamma V_0 u\|_{s, \Gamma}
    &&\leq C k \| V_0 u \|_{s, \Gamma}
    &&\leq C k \| u \|_{s-1, \Gamma}, \\
    \| i (\eta - 1) \HNegGamma V_0 u \|_{s, \Gamma}
    &\leq C k \| \HNegGamma V_0 u\|_{s, \Gamma}
    &&\leq C k (q/k) \| V_0 u \|_{s+1, \Gamma}
    &&\leq C q \| u \|_{s, \Gamma}, \\
    \| i (\eta - 1) \HNegGamma V_0 u \|_{s-1, \Gamma}
    &\leq C k \| \HNegGamma V_0 u\|_{s-1, \Gamma}
    &&\leq C k (q/k) \| V_0 u \|_{s, \Gamma}
    &&\leq C q \| u \|_{s-1, \Gamma}.
  \end{alignat*}

  In the Steps 2 and 3 below we will again heavily use the properties of
  $\HNegGamma$
  given in Proposition~\ref{prop:lemma_melenk_6.3}.
  Furthermore, we often apply the results of
  Proposition~\ref{prop:thm_melenk_5.3},
  especially Item~\ref{item:thm5.3_1}.
  Below, we will write certain exponents in a nonsimplified way 
  to indicate the corresponding choices of Sobolev exponents
  when applying Proposition~\ref{prop:thm_melenk_5.3}.

  \textbf{Step 2}:
 We claim:  
  \begin{align}
\label{eq:thm:extension_of_thm_melenk_2.9-200}
    \| \HNegGamma \gamma_1^{int} \widetilde{S}_{V,pw} u \|_{s+1, \Gamma}
    &\leq C q k \| u \|_{s, \Gamma}, \\
\label{eq:thm:extension_of_thm_melenk_2.9-300}
    \| \HNegGamma \gamma_1^{int} \widetilde{S}_{V,pw} u \|_{s, \Gamma}
    &\leq  C qk \| u \|_{s-1, \Gamma}, \\
\label{eq:thm:extension_of_thm_melenk_2.9-400}
    \| \HNegGamma \gamma_1^{int} \widetilde{S}_{V,pw} u \|_{s, \Gamma}
    &\leq  C q^2 \| u \|_{s, \Gamma},\\
\label{eq:thm:extension_of_thm_melenk_2.9-500}
    \| \HNegGamma \gamma_1^{int} \widetilde{S}_{V,pw} u \|_{s-1, \Gamma}
    &\leq  C q^2 \|u\|_{s-1,\Gamma}.
  \end{align}
To see (\ref{eq:thm:extension_of_thm_melenk_2.9-200}), we calculate 
  \begin{align*}
    \| \HNegGamma \gamma_1^{int} \widetilde{S}_{V,pw} u \|_{s+1, \Gamma}
    &\leq  C\| \gamma_1^{int} \widetilde{S}_{V,pw} u \|_{s+1, \Gamma}
    \leq  C\| \widetilde{S}_{V,pw} u \|_{s+5/2, \Omega} \\
    &\leq  C q^2 (qk^{-1})^{1+(s+1/2)-(s+5/2)} \| u \|_{s, \Gamma}
    = C q k \| u \|_{s, \Gamma}.
  \end{align*}
By a similar calculation, we obtain 
(\ref{eq:thm:extension_of_thm_melenk_2.9-400}): 
  \begin{align*}
    \| \HNegGamma \gamma_1^{int} \widetilde{S}_{V,pw} u \|_{s, \Gamma}
    \leq  C q/k \| \gamma_1^{int} \widetilde{S}_{V,pw} u \|_{s+1, \Gamma}
    \leq  C q^2 \| u \|_{s, \Gamma}.
  \end{align*}
  In the case $s \in [0,1/2)$, we perform a multiplicative trace inequality and find
  \begin{align*}
    \| & \HNegGamma  \gamma_1^{int}  \widetilde{S}_{V,pw} u \|_{s-1, \Gamma}
    \leq  C (q/k)^{-s+1} \| \gamma_1^{int} \widetilde{S}_{V,pw} u \|_{0, \Gamma} \\
    &\leq  C (q/k)^{-s+1} \| \widetilde{S}_{V,pw} u \|_{1, \Omega}^{1/2} \| \widetilde{S}_{V,pw} u \|_{2, \Omega}^{1/2} \\
    &\leq  C (q/k)^{-s+1} \left[ q^2 (qk^{-1})^{1+(s-1/2)-1} \right]^{1/2} \left[ q^2 (qk^{-1})^{1+(s-1/2)-2} \right]^{1/2} \|u\|_{s-1,\Gamma} \\
    &=  C q^2 \|u\|_{s-1,\Gamma}.
  \end{align*}
  In the case $s \geq 1/2$ we perform a standard trace estimate and find
  \begin{align}
\nonumber 
    \| \HNegGamma \gamma_1^{int} \widetilde{S}_{V,pw} u \|_{s-1, \Gamma}
    &\leq  C q/k \| \gamma_1^{int} \widetilde{S}_{V,pw} u \|_{s, \Gamma}
    \leq  C q/k \| \widetilde{S}_{V,pw} u \|_{s+3/2, \Omega} \\
\nonumber 
    &\leq  C q/k q^2 (qk^{-1})^{1+(s-1/2)-(s+3/2)}  \|u\|_{s-1,\Gamma} \\
\label{eq:thm:extension_of_thm_melenk_2.9-100}
    &=  C q^2 \|u\|_{s-1,\Gamma}.
  \end{align}
The two previous estimates show (\ref{eq:thm:extension_of_thm_melenk_2.9-500}). 
Analogously to (\ref{eq:thm:extension_of_thm_melenk_2.9-100}), we find (\ref{eq:thm:extension_of_thm_melenk_2.9-300}): 
  \begin{align*}
    \| \HNegGamma \gamma_1^{int} \widetilde{S}_{V,pw} u \|_{s, \Gamma}
    \leq C \| \gamma_1^{int} \widetilde{S}_{V,pw} u \|_{s, \Gamma}
    \leq  C qk \| u \|_{s-1, \Gamma}.
  \end{align*}
  \textbf{Step 3}:
We claim: 
  \begin{align}
\label{eq:thm:extension_of_thm_melenk_2.9-600}
    \| \eta \HNegGamma \gamma_0^{int} \widetilde{S}_{V,pw} u \|_{s+1, \Gamma}
    &\leq C q^2 k \| u \|_{s, \Gamma}, \\
\label{eq:thm:extension_of_thm_melenk_2.9-700}
    \| \eta \HNegGamma \gamma_0^{int} \widetilde{S}_{V,pw} u \|_{s, \Gamma}
    &\leq  C q^2 k \| u \|_{s-1, \Gamma}, \\
\label{eq:thm:extension_of_thm_melenk_2.9-800}
    \| \eta \HNegGamma \gamma_0^{int} \widetilde{S}_{V,pw} u \|_{s, \Gamma}
    &\leq  C q^3 \| u \|_{s, \Gamma},\\
\label{eq:thm:extension_of_thm_melenk_2.9-900}
    \| \eta \HNegGamma \gamma_0^{int} \widetilde{S}_{V,pw} u \|_{s-1, \Gamma}
    &\leq  C q^3 \|u\|_{s-1,\Gamma}.
  \end{align}
For (\ref{eq:thm:extension_of_thm_melenk_2.9-600}) we estimate 
  \begin{align*}
    \| \eta \HNegGamma \gamma_0^{int} \widetilde{S}_{V,pw} u \|_{s+1, \Gamma}
    &\leq  C k \| \gamma_0^{int} \widetilde{S}_{V,pw} u \|_{s+1, \Gamma}
     \leq  C k \| \widetilde{S}_{V,pw} u \|_{s+3/2, \Omega} \\
    &\leq  C k q^2 (qk^{-1})^{1+(s+1/2)-(s+3/2)} \| u \|_{s, \Gamma}
    = C q^2 k \| u \|_{s, \Gamma}.
  \end{align*}
By a similar calculation, we show (\ref{eq:thm:extension_of_thm_melenk_2.9-700}):  
  \begin{align*}
    \| \eta \HNegGamma \gamma_0^{int} \widetilde{S}_{V,pw} u \|_{s, \Gamma}
    \leq C k q/k \| \gamma_0^{int} \widetilde{S}_{V,pw} u \|_{s+1, \Gamma}
    \leq  C q^3 \| u \|_{s, \Gamma}.
  \end{align*}
  In the case $s \in [0,1/2)$, we perform a multiplicative trace inequality and find
  \begin{align*}
    \| \eta & \HNegGamma \gamma_0^{int}  \widetilde{S}_{V,pw} u \|_{s-1, \Gamma}
    \leq  C k (q/k)^{-s+1} \| \gamma_0^{int} \widetilde{S}_{V,pw} u \|_{0, \Gamma} \\
    &\leq  C k (q/k)^{-s+1} \| \widetilde{S}_{V,pw} u \|_{0, \Omega}^{1/2} \| \widetilde{S}_{V,pw} u \|_{1, \Omega}^{1/2} \\
    &\leq  C k (q/k)^{-s+1} \left[ q^2 (qk^{-1})^{1+(s-1/2)-0} \right]^{1/2} \left[ q^2 (qk^{-1})^{1+(s-1/2)-1} \right]^{1/2} \|u\|_{s-1,\Gamma} \\
    &=  C q^3 \|u\|_{s-1,\Gamma}.
  \end{align*}
  In the case $s \geq 1/2$ we perform a standard trace estimate and find
  \begin{align}
\nonumber 
    \| \eta \HNegGamma \gamma_0^{int} \widetilde{S}_{V,pw} u \|_{s-1, \Gamma}
    &\leq  C k q/k \| \gamma_0^{int} \widetilde{S}_{V,pw} u \|_{s, \Gamma}
    \leq  C k q/k \| \widetilde{S}_{V,pw} u \|_{s+1/2, \Omega} \\
\nonumber 
    &\leq  C k q/k q^2 (qk^{-1})^{1+(s-1/2)-(s+1/2)}  \|u\|_{s-1,\Gamma} \\
\label{eq:thm:extension_of_thm_melenk_2.9-1000}
    &=  C q^3 \|u\|_{s-1,\Gamma}.
  \end{align}
The two previous estimates show (\ref{eq:thm:extension_of_thm_melenk_2.9-900}).  
Analogously to (\ref{eq:thm:extension_of_thm_melenk_2.9-1000}), we assert 
(\ref{eq:thm:extension_of_thm_melenk_2.9-800}):  
  \begin{align*}
    \| \eta \HNegGamma \gamma_0^{int} \widetilde{S}_{V,pw} u \|_{s, \Gamma}
    \leq  k \| \gamma_0^{int} \widetilde{S}_{V,pw} u \|_{s, \Gamma}
    \leq  C q^2 k \| u \|_{s-1, \Gamma}.
  \end{align*}
  \textbf{Step 4}:
  The definition of the operator $R_{A^\prime}$ in~\eqref{splitting_operator_definition-1},
  the triangle inequality,
  and appropriate choice of $q$ yields mapping properties of $R_{A^\prime}$ as stated in~\eqref{eq:Rmapping}.
\end{proof}

Finally, a simple application of \cite[Cor.~{7.5}]{melenk12}
for non-trapping $\Omega^+$ with analytic boundary is
the following: 
\begin{lemma}\label{lemma:comb-field-ana-rhs}
  Let $\Omega^+$ be non-trapping (\cite[Def.~{1.1}]{baskin-spence-wunsch16}).
  Let $\Gamma$ be analytic,
  $T$ be a tubular neighborhood of $\Gamma$
  and $C_{g_1}$, $C_{g_2}$, $\gamma_g >0$.
  Then there exist constants $C$, $\gamma > 0$ independent of $k \geq k_0$
  such that for all $g_1 \in \AnaClass{C_{g_1}}{\gamma_g}{T}$,
  $g_2 \in \AnaClass{C_{g_2}}{\gamma_g}{T}$
  the solution $\varphi \in L^2(\Gamma)$ of
  \begin{equation*}
    A_{k,\eta}^\prime \varphi = k \llbracket g_1 \rrbracket + \llbracket \partial_n g_2 \rrbracket
  \end{equation*}
  satisfies
  \begin{equation*}
    \varphi = - \llbracket \partial_n v \rrbracket,
    \qquad
    v \in \AnaClassmelenk{C k^{5/2} (C_{g_1} + C_{g_2}) }{\gamma}{\Omega_R}.
  \end{equation*}
\end{lemma}
\begin{proof}
  We apply \cite[Cor.~{7.5}]{melenk12}
  with $s_A = 0$.
  By 
  Proposition~\ref{proposition:mapping_properties_bios}\ref{proposition:mapping_properties_bios_comined_field_operator_helmholtz}
  the operator $A_{\eta,k}^\prime \colon L^2(\Gamma) \rightarrow L^2(\Gamma)$
  is boundedly invertible.
  The result follows immediately from \cite[Cor.~{7.5}]{melenk12}
  together with the bound~\eqref{eq:combined_field_operator_norm_optimal}.
\end{proof}

\begin{proof}
  \emph{Proof of Lemma~\ref{lemma:properties_dtn_operators}\ref{item:lemma:properties_dtn_operators_splitting_finite_ana_analytic_part}:} 

  \textbf{Step 1}:
  We derive a splitting of $(A_{k,\eta}^\prime)^{-1}$ similar to the results of \cite[Thm.~{2.11}]{melenk12}.
  Fix $\hat{q} \in (0,1)$.
  Let
  \begin{equation*}
    q \coloneqq \hat{q} \min\left\{1 , \frac{1}{\|( A_{0,1}^\prime )^{-1}\|_{H^{s}(\Gamma) \leftarrow H^{s}(\Gamma)} },
     \frac{1}{\|( A_{0,1}^\prime )^{-1}\|_{H^{s-1}(\Gamma) \leftarrow H^{s-1}(\Gamma)} }\right\}.
  \end{equation*}
  Note that by Proposition~\ref{proposition:mapping_properties_bios} the operator $A_{0,1}^\prime \colon H^{t}(\Gamma) \rightarrow H^{t}(\Gamma)$
  is boundedly invertible for $t \geq -1$ and therefore $q$ is well defined and $q \in (0,1)$.
  Theorem~\ref{thm:extension_of_thm_melenk_2.9} applied with this $q$ gives a decomposition
  \begin{equation*}
    A^\prime_{k, \eta} = A_{0,1}^\prime + R + \llbracket \mathcal{A} \rrbracket,
  \end{equation*}
  with $R = R_{A^\prime}$ and $\mathcal{A} =  k \widetilde{\mathcal{A}}_{1} + \partial_n \widetilde{\mathcal{A}}_{2}$,
  as in Theorem~\ref{thm:extension_of_thm_melenk_2.9}.
  By construction
  \begin{equation}\label{eq:geom-series-estimates}
    \| ( A_{0,1}^\prime )^{-1} R \|_{H^{s}(\Gamma) \leftarrow H^{s}(\Gamma)} \leq \hat{q} \quad \text{and} \quad
    \| ( A_{0,1}^\prime )^{-1} R \|_{H^{s-1}(\Gamma) \leftarrow H^{s-1}(\Gamma)} \leq \hat{q}.
  \end{equation}
  Hence, $A_{0,1}^\prime + R $ is boundedly invertible by a geometric series argument, since
  \begin{equation}\label{eq:formula-A-plus-R-inverse}
    (A_{0,1}^\prime + R )^{-1} =  ( I + ( A_{0,1}^\prime )^{-1} R )^{-1} ( A_{0,1}^\prime )^{-1}
  \end{equation}
  with the norm estimates
  \begin{subequations}
    \begin{align}
      \|(A_{0,1}^\prime + R )^{-1} \|_{H^{s}(\Gamma) \leftarrow H^{s}(\Gamma)} &\leq (1-\hat{q})^{-1} \| ( A_{0,1}^\prime )^{-1} \|_{H^{s}(\Gamma) \leftarrow H^{s}(\Gamma)},  \label{eq:AplusRinvnormss} \\
      \|(A_{0,1}^\prime + R )^{-1} \|_{H^{s-1}(\Gamma) \leftarrow H^{s-1}(\Gamma)} &\leq (1-\hat{q})^{-1} \| ( A_{0,1}^\prime )^{-1} \|_{H^{s-1}(\Gamma) \leftarrow H^{s-1}(\Gamma)}. \label{eq:AplusRinvnormsminus1sminus1}
    \end{align}
  \end{subequations}
  By Proposition~\ref{proposition:mapping_properties_bios} the operator
  $A^\prime_{k, \eta} \colon H^{t}(\Gamma) \rightarrow H^{t}(\Gamma)$
  is boundedly invertible for $t \geq -1$.
  We may decompose $(A^\prime_{k, \eta})^{-1}$ as follows
  \begin{equation}
\label{eq:combined-field-operator-inverse-splitting}
    (A^\prime_{k, \eta})^{-1} = ( A_{0,1}^\prime + R )^{-1} + Q, 
\qquad Q = - (A^\prime_{k, \eta})^{-1} \llbracket \mathcal{A} \rrbracket ( A_{0,1}^\prime + R )^{-1},
  \end{equation}
  since
  \begin{align*}
    I &= (A^\prime_{k, \eta}) (A^\prime_{k, \eta})^{-1}
    = (A^\prime_{k, \eta}) ( A_{0,1}^\prime + R )^{-1} + (A^\prime_{k, \eta})Q \\
    &= ( A_{0,1}^\prime + R + \llbracket \mathcal{A} \rrbracket) ( A_{0,1}^\prime + R )^{-1} + (A^\prime_{k, \eta})Q \\
    &= I + \llbracket \mathcal{A} \rrbracket ( A_{0,1}^\prime + R )^{-1} + (A^\prime_{k, \eta})Q.
  \end{align*}

  \textbf{Step 2}
  We rewrite the difference $\mathrm{DtN}_k - \mathrm{DtN}_0$ using the combined field operators of \eqref{eq:combined_field_equation}. 
  Using $\eta$ as in~\eqref{eq:eta-comb-field-estimate} for $\mathrm{DtN}_k$ and $\eta = 1$ for $\mathrm{DtN}_0$ we find
  \begin{equation}\label{eq:dtn_k_dtn_0_without_splitting}
    \mathrm{DtN}_k - \mathrm{DtN}_0
    =
    (A_{k,\eta}^\prime)^{-1}
    \left[ - D_k + i \eta \left(-\nicefrac{1}{2} + K_k\right) \right]
    -
    (A_{0,1}^\prime)^{-1}
    \left[ - D_0 + i \left(-\nicefrac{1}{2} + K_0\right) \right].
  \end{equation}
  Adding and subtracting $D_0$ and $K_0$ in \eqref{eq:dtn_k_dtn_0_without_splitting},
  employing the splitting of $D_k - D_0$ and $K_k - K_0$ given in
  \eqref{proposition:mapping_properties_bios_helmholtz_splitting}
  in Proposition~\ref{proposition:mapping_properties_bios},
  and applying the splitting of $(A_{k,1}^\prime)^{-1}$ in
  \eqref{eq:combined-field-operator-inverse-splitting}
  we find
  \begin{align*}
    \mathrm{DtN}_k - \mathrm{DtN}_0
    &=  - (A_{k,\eta}^\prime)^{-1} \left[ D_k - D_0 \right]
        - (A_{k,\eta}^\prime)^{-1} D_0 \\
    & \quad
        + i \eta (A_{k,\eta}^\prime)^{-1} \left[ K_k - K_0 \right]
        + i \eta (A_{k,\eta}^\prime)^{-1} \left[-\nicefrac{1}{2} + K_0 \right] \\
    & \quad
       + (A_{0,1}^\prime)^{-1}  D_0
       - i (A_{0,1}^\prime)^{-1}   \left[  -\nicefrac{1}{2} + K_0  \right] \\
    & = (A_{0,1}^\prime)^{-1} D_0  -( A_{0,1}^\prime + R )^{-1} D_0 - Q D_0  \\
    &\quad  - ( A_{0,1}^\prime + R )^{-1} S_D - Q S_D + (A_{k,\eta}^\prime)^{-1} \gamma_1^{int} \widetilde{A}_K \\
    &\quad + i \eta  ( A_{0,1}^\prime + R )^{-1} S_K + i \eta Q S_K + i \eta (A_{k,\eta}^\prime)^{-1} \gamma_0^{int} \widetilde{A}_K  \\
    &\quad + i \eta ( A_{0,1}^\prime + R )^{-1} \left[ -\nicefrac{1}{2} + K_0  \right] + i \eta Q \left[ -\nicefrac{1}{2} + K_0  \right] \\
    &\quad - i (A_{0,1}^\prime)^{-1} \left[ -\nicefrac{1}{2} + K_0  \right] \\
    &= \mathrm{FSO} + \mathrm{ASO}
  \end{align*}
  with the Finite Shift Operators ($\mathrm{FSO}$) and the Analytic Shift Operators ($\mathrm{ASO}$) given by
  \begin{align*}
    \mathrm{FSO} &\coloneqq
      (A_{0,1}^\prime)^{-1} D_0  -( A_{0,1}^\prime + R )^{-1} D_0 
%      \\& \quad 
- ( A_{0,1}^\prime + R )^{-1} S_D
      + i \eta  ( A_{0,1}^\prime + R )^{-1} S_K 
\\ &\quad 
+ i \eta ( A_{0,1}^\prime + R )^{-1} \left[ -\nicefrac{1}{2} + K_0  \right] 
%\\ &\quad 
- i (A_{0,1}^\prime)^{-1} \left[ -\nicefrac{1}{2} + K_0  \right], \\
    \mathrm{ASO} &\coloneqq
      - Q D_0 - Q S_D + (A_{k,\eta}^\prime)^{-1} \gamma_1^{int} \widetilde{A}_K \\
      &\quad + i \eta Q S_K + i \eta (A_{k,\eta}^\prime)^{-1} \gamma_0^{int} \widetilde{A}_K
      + i \eta Q \left[ -\nicefrac{1}{2} + K_0  \right].
  \end{align*}

  \textbf{Step 3 (Analysis of $\mathrm{FSO}$):}
  We claim 
  \begin{equation}\label{eq:assertion-fos}
    \mathrm{FSO} = k B,
  \end{equation}
  where $B: H^s(\Gamma) \rightarrow H^s(\Gamma)$ is a bounded linear operator with 
  $\| B u \|_{s,\Gamma} \lesssim \| u \|_{s,\Gamma}$, as asserted in the present lemma.
  Using the mapping properties of $( A_{0,1}^\prime + R)^{-1}$
  in~\eqref{eq:AplusRinvnormss} as well as \eqref{proposition:mapping_properties_bios_helmholtz_standard_result},
  \eqref{proposition:mapping_properties_bios_helmholtz_splitting_mapping_properties_simple}
  and Proposition~\ref{proposition:mapping_properties_bios}\ref{proposition:mapping_properties_bios_comined_field_operator_laplacian}
we find
  \begin{align*}
    \| ( A_{0,1}^\prime + R )^{-1} S_D u \|_{s, \Gamma}
      &\lesssim \| S_D u \|_{s, \Gamma}
      \lesssim k \| u \|_{s, \Gamma},   \\
    k \| ( A_{0,1}^\prime + R )^{-1} S_K u \|_{s, \Gamma}
      &\lesssim k \| S_K u \|_{s, \Gamma}
      \lesssim k \| u \|_{s, \Gamma},   \\
    k \| ( A_{0,1}^\prime + R )^{-1} \left[ -\nicefrac{1}{2} + K_0  \right] u \|_{s, \Gamma}
      &\lesssim  k \| \left[ -\nicefrac{1}{2} + K_0  \right] u \|_{s, \Gamma}
      \lesssim k \| u \|_{s, \Gamma},  \\
    \| (A_{0,1}^\prime)^{-1} \left[ -\nicefrac{1}{2} + K_0  \right] u \|_{s, \Gamma}
      &\lesssim \| \left[ -\nicefrac{1}{2} + K_0  \right] u  \|_{s, \Gamma}
      \lesssim   \| u \|_{s, \Gamma}.
  \end{align*}
The assertion \eqref{eq:assertion-fos} therefore follows once we have shown 
  \begin{align*}
    \| (A_{0,1}^\prime)^{-1} D_0  -( A_{0,1}^\prime + R )^{-1} D_0 u  \|_{s, \Gamma}
      \lesssim k \| u \|_{s, \Gamma} .
  \end{align*}
To that end, we write using~\eqref{eq:formula-A-plus-R-inverse} 
  \begin{align*}
    &(A_{0,1}^\prime)^{-1} D_0  - (A_{0,1}^\prime + R)^{-1} D_0
    =  (A_{0,1}^\prime)^{-1} D_0  - (I + (A_{0,1}^\prime)^{-1} R)^{-1} (A_{0,1}^\prime)^{-1} D_0  \\
    &=  \left[ I  - (I + (A_{0,1}^\prime)^{-1} R)^{-1} \right] (A_{0,1}^\prime)^{-1} D_0  
    =  - \left[ \sum_{n=1}^\infty (-1)^n ((A_{0,1}^\prime)^{-1} R)^n \right] (A_{0,1}^\prime)^{-1} D_0.
  \end{align*}
  Applying the previous calculations,
  a geometric series argument with~\eqref{eq:geom-series-estimates},
  the mapping properties of $(A_{0,1}^\prime)^{-1}$ 
  in Proposition~\ref{proposition:mapping_properties_bios}\ref{proposition:mapping_properties_bios_comined_field_operator_laplacian}, 
  the estimate
  $ \| R u \|_{s, \Gamma} \lesssim k \| u \|_{s-1, \Gamma}$
  given by Theorem~\ref{thm:extension_of_thm_melenk_2.9},
  again the mapping properties of $(A_{0,1}^\prime)^{-1}$, 
  and finally the mapping properties of $D_0$ given
  in~\eqref{proposition:mapping_properties_bios_helmholtz_standard_result}
  in Proposition~\ref{proposition:mapping_properties_bios}, we find
  \begin{align*}
    \| (A_{0,1}^\prime)^{-1} D_0  - & (A_{0,1}^\prime + R)^{-1} D_0 u \|_{s, \Gamma} \\
    &= \Bigl\| \Bigl[ \sum_{n=1}^\infty (-1)^n ((A_{0,1}^\prime)^{-1} R)^{n-1} \Bigr] ((A_{0,1}^\prime)^{-1} R) (A_{0,1}^\prime)^{-1} D_0 u \Bigr\|_{s, \Gamma}  \\
    &\leq  \frac{1}{1-\hat{q}} \| (A_{0,1}^\prime)^{-1} R (A_{0,1}^\prime)^{-1} D_0 u \|_{s, \Gamma}
    \lesssim \| R (A_{0,1}^\prime)^{-1} D_0 u \|_{s, \Gamma}  \\
    &\lesssim k \| (A_{0,1}^\prime)^{-1} D_0 u \|_{s-1, \Gamma}
    \lesssim k \| D_0 u \|_{s-1, \Gamma}
    \lesssim k \| u \|_{s, \Gamma}.
  \end{align*}
  Hence, the assertion in~\eqref{eq:assertion-fos} follows, which concludes the analysis of the finite shift operators $\mathrm{FSO}$.
  %Summarizing, so far we have found that
  %\begin{equation*}
  %  \mathrm{DtN}_k - \mathrm{DtN}_0 = k B + \mathrm{ASO},
  %\end{equation*}
  %with $B$ as in the assertions of the present lemma.

  \textbf{Step 4 (Analysis of Analytic Shift Operators $\mathrm{ASO}$):}
  We have
  \begin{align*}
    \mathrm{ASO} &=
      - Q D_0 - Q \left[ S_D - i \eta S_K - i \eta \left( -\nicefrac{1}{2} + K_0  \right) \right]
%\\ &\quad  
+ (A_{k,\eta}^\prime)^{-1} [ i \eta \gamma_0^{int} \widetilde{A}_K + \gamma_1^{int} \widetilde{A}_K ].
  \end{align*}

  \textbf{Step 4a (Analysis of $- Q D_0$):}
  With the definition of $Q$ in \eqref{eq:combined-field-operator-inverse-splitting} we have for $f \in H^s(\Gamma)$
  \begin{align*}
    - Q D_0 f
    &= (A_{k,\eta}^\prime)^{-1} \llbracket \mathcal{A} \rrbracket (A_{0,1}^\prime + R)^{-1} D_0 f \\
    &= (A_{k,\eta}^\prime)^{-1}
    \left\{ k \llbracket \mathcal{A}_1 \rrbracket (A_{0,1}^\prime + R)^{-1} D_0 f + \llbracket \partial_n \mathcal{A}_2 \rrbracket (A_{0,1}^\prime + R)^{-1} D_0 f \right\}.
  \end{align*}
  In order to apply Lemma~\ref{lemma:comb-field-ana-rhs},
  we use the mapping properties of $\mathcal{A}_1$ and $\mathcal{A}_2$
  given in Theorem~\ref{thm:extension_of_thm_melenk_2.9}
  and estimate
  \begin{align*}
    k \|(A_{0,1}^\prime + R)^{-1} D_0 &  f \|_{-3/2, \Gamma} + k^{d/2} \| (A_{0,1}^\prime + R)^{-1} D_0 f \|_{-1, \Gamma}  \\
    &\lesssim
    k^{d/2} \| (A_{0,1}^\prime + R)^{-1} D_0 f \|_{s-1, \Gamma} 
    \lesssim
    k^{d/2} \| D_0 f \|_{s-1, \Gamma} 
    \lesssim
    k^{d/2} \| f \|_{s, \Gamma},
  \end{align*}
  where we used the trivial embedding $H^{s-1}(\Gamma) \subset  H^{-1}(\Gamma) \subset H^{-3/2}(\Gamma)$,
  the fact that $k+k^{d/2}\lesssim k^{d/2}$,
  the mapping property~\eqref{eq:AplusRinvnormsminus1sminus1}, 
  and finally the mapping properties of $D_0$ given
  in~\eqref{proposition:mapping_properties_bios_helmholtz_standard_result}. 
  %in Proposition~\ref{proposition:mapping_properties_bios}.
  Hence, for the tubular neighborhood $T$
  given in Theorem~\ref{thm:extension_of_thm_melenk_2.9}
  we find
  \begin{alignat*}{2}
    \mathcal{A}_1 (A_{0,1}^\prime + R)^{-1} D_0 f  &\in \AnaClassmelenk{C_1 k^{d/2} \| f \|_{s, \Gamma}}{\gamma_1}{T}, \\
    \mathcal{A}_2 (A_{0,1}^\prime + R)^{-1} D_0 f  &\in \AnaClassmelenk{C_1 k       \| f \|_{s, \Gamma}}{\gamma_1}{T},
  \end{alignat*}
  for constants $C_1$, $\gamma_1 > 0$ independent of $k$.
  Lemma~\ref{lemma:comb-field-ana-rhs} is applicable and 
  yields, for constants $\widetilde{C}_1$, $\widetilde{\gamma}_1 > 0$ independent of $k$, the representation 
  \begin{equation}\label{eq:rep1}
    - Q D_0 f = \llbracket \partial_n v^1_f \rrbracket,
    \qquad
    v^1_f \in \AnaClass{\widetilde{C}_1 k^{5/2 + d/2} \| f \|_{s, \Gamma}}{\widetilde{\gamma}_1}{\Omega_R}.
  \end{equation}

  \textbf{Step 4b (analysis of
    $- Q \left[ S_D - i \eta S_K - i \eta \left( -\nicefrac{1}{2} + K_0  \right) \right]$):}
  We estimate
  \begin{align*}
    k^{d/2} \| (A_{0,1}^\prime + R)^{-1} & \left[ S_D - i \eta S_K - i \eta \left( -\nicefrac{1}{2} + K_0  \right) \right] f \|_{-1, \Gamma} \\
    &\lesssim
    k^{d/2} \| (A_{0,1}^\prime + R)^{-1} \left[ S_D - i \eta S_K - i \eta \left( -\nicefrac{1}{2} + K_0  \right) \right] f \|_{s, \Gamma} \\
    &\lesssim
    k^{d/2} \| \left[ S_D - i \eta S_K - i \eta \left( -\nicefrac{1}{2} + K_0  \right] \right) f \|_{s, \Gamma} \\
    &\lesssim
    k^{d/2+1} \| f \|_{s, \Gamma},
  \end{align*}
  where we first use the trivial embedding $H^{s}(\Gamma) \subset H^{-1}(\Gamma)$,
  the mapping property~\eqref{eq:AplusRinvnormss},
  the mapping properties of $S_D$, $S_K$, and $K_0$ given 
  in~\eqref{proposition:mapping_properties_bios_helmholtz_splitting_mapping_properties_simple}
  and~\eqref{proposition:mapping_properties_bios_helmholtz_standard_result}
  %in Proposition~\ref{proposition:mapping_properties_bios}
  as well as $|\eta| \lesssim k$.
  Proceeding as in Step~{4a} we find the representation
  \begin{equation}\label{eq:rep2}
    - Q \left[ S_D - i \eta S_K - i \eta \left( -\nicefrac{1}{2} + K_0  \right) \right] = \llbracket \partial_n v^2_f \rrbracket,
    \quad
    v^2_f \in \AnaClass{\widetilde{C}_2 k^{5/2 + d/2 + 1} \| f \|_{s, \Gamma}}{\widetilde{\gamma}_2}{\Omega_R}
  \end{equation}
  to hold true, for constants $\widetilde{C}_2$, $\widetilde{\gamma}_2 > 0$ independent of $k$.

  \textbf{Step 4c (analysis of 
  $  (A_{k,\eta}^\prime)^{-1} [ i \eta \gamma_0^{int} \widetilde{A}_K + \gamma_1^{int} \widetilde{A}_K ]$):}
  For $f \in H^s(\Gamma)$ the mapping properties of $\widetilde{A}_K$ imply 
  $\widetilde{A}_{K} f
  \in
  \AnaClassmelenk{C_K k \| f \|_{-1/2, \Gamma}}{\gamma_K}{\Omega}$,
  see~\eqref{eq:mapping_prop_A_tilde_K}. 
  %in Proposition~\ref{proposition:mapping_properties_bios}.
  Upon extending $\widetilde{A}_{K} f$ by zero outside of $\Omega$,
  we find Lemma~\ref{lemma:comb-field-ana-rhs} to be applicable,
  which yields 
  \begin{equation}\label{eq:rep3}
    (A_{k,\eta}^\prime)^{-1} [ i \eta \gamma_0^{int} \widetilde{A}_K + \gamma_1^{int} \widetilde{A}_K ] f = \llbracket \partial_n v^3_f \rrbracket,
    \qquad
    v^3_f \in \AnaClassmelenk{\widetilde{C}_3 k^{5/2+1} \| f \|_{-1/2, \Gamma}}{\widetilde{\gamma}_3}{\Omega_R},
  \end{equation}
  with constants $\widetilde{C}_3$, $\widetilde{\gamma}_3 > 0$ independent of $k$.

  \textbf{Step 5}:
  Collecting the representations~\eqref{eq:rep1},~\eqref{eq:rep2}, and~\eqref{eq:rep3} gives 
  \begin{equation*}
    \mathrm{ASO} = \llbracket \partial_n \widetilde{A} \rrbracket,
    \qquad
    \widetilde{A} u \in \AnaClassmelenk{C k^{7/2+d/2} \| u \|_{s, \Gamma}}{ \gamma}{\Omega_R}
  \end{equation*}
  with $\widetilde{A}$ as in the assertions of the present lemma.
  Hence, the splitting
  \begin{equation*}
    \mathrm{DtN}_k - \mathrm{DtN}_0 = kB + \llbracket \partial_n \widetilde{A} \rrbracket
  \end{equation*}
  with $B$ and $\widetilde{A}$ as asserted, holds true.
  This concludes the proof.
\end{proof}

\textbf{Acknowledgement}:
MB and JMM gratefully acknowledge financial support by the Austrian Science Fund (FWF) through 
the doctoral school \textit{Dissipation and dispersion in nonlinear PDEs} (grant W1245) and the SFB \textit{Taming Complexity in Partial Differential Systems} 
(grant SFB F65, \href{https://doi.org/10.55776/F65}{DOI:10.55776/F65}).

\bibliographystyle{alpha}
\bibliography{literature.bib}

\end{document}